\documentclass[10pt,a4paper]{amsart}

\input{ivan_style.sty}

\usepackage[
backend=biber,
style=alphabetic,
sorting=nyt,
maxbibnames=99,
firstinits=true,
isbn=false,
doi=false,
url=false,
eprint=false
]{biblatex}

\AtEveryBibitem{\clearfield{month}}
\AtEveryBibitem{\clearfield{day}}
\DeclareFieldFormat[article, inbook]{title}{#1} 
\DeclareFieldFormat[article]{volume}{\mkbibbold{#1}}
\DeclareFieldFormat[article]{number}{\mkbibparens{#1}}
\DeclareFieldFormat{pages}{#1}
\renewbibmacro{in:}{}
\begin{document}
\setstcolor{red}
\title[A polytopal generalization of Apollonian packings and Descartes' theorem]{A polytopal generalization of Apollonian packings and Descartes' theorem} 

\thanks{$^\dagger$ Partially supported by IEA-CNRS.\\
$^*$ Partially supported by CNRS and the Austrian Science Fund (FWF), projects F-5503 and P-34763}
\author[Jorge L. Ram\'irez Alfons\'in]{Jorge L. Ram\'irez Alfons\'in$^\dagger$}
\address{IMAG Univ.\ Montpellier, CNRS, France}
\email{jorge.ramirez-alfonsin@umontpellier.fr}
\author[Iv\'an Rasskin]{Iv\'an Rasskin$^*$}
\address{LIS, Aix-Marseille Université, CNRS, France}
\email{ivan.rasskin@lis-lab.fr}

\subjclass[2010]{51M20, 52C26, 05B40}

\keywords{Apollonian packings, Descartes' theorem, Polytopes, Arithmetic groups, Platonic solids}

\begin{abstract} 
We present a generalization of Descartes' theorem for the family of polytopal sphere packings arising from uniform polytopes. The corresponding quadratic equation is expressed in terms of geometric invariants of uniform polytopes which are closely connected to canonical realizations of edge-scribable polytopes. We use our generalization to construct integral Apollonian packings based on the Platonic solids. Additionally, we also introduce and discuss a new spectral invariant for edge-scribable polytopes. 
\end{abstract}

\maketitle

\section{Introduction}

	Apollonian packings and their generalizations have proven to be effective tools for studying various structures in natural science. They also appear in many branches of mathematics, such as geometric group theory \cite{zhang2023elementary}, fractal geometry \cite{mandelbrot2004fractals}, discrete geometry \cite{chen2016}, knot theory \cite{RR2024links}, and number theory \cite{apoNumber} (see \cite{sheard2020on} for an excellent survey).  The process to construct them begins with an initial configuration of four pairwise tangent circles on the plane. By adding the inscribed circle to the interstice between each triple of circles and repeating this process ad infinitum, we obtain an \textit{Apollonian packing of circles}, as illustred in Figure \ref{fig:tetrahedral}.
	
	\begin{figure}[H]
		\centering
		\includegraphics[width=.48\textwidth]{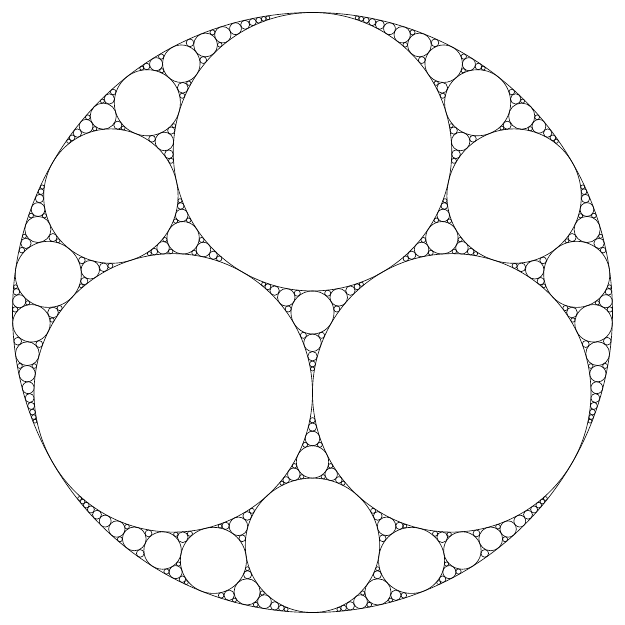}
		\caption{An Apollonian packing.}   
		\label{fig:tetrahedral}
	\end{figure}
	
	The connection  with number theory was noted by Soddy in \cite{soddy1936} and relies on the existence of \textit{integral} Apollonian packings, where the bends\footnote{also called curvatures} of the circles (the reciprocals of the signed radii) are all integers (see Figure \ref{fig:apoint}). 
	
	\begin{figure}[H]
		\centering
		\includegraphics[width=.48\textwidth]{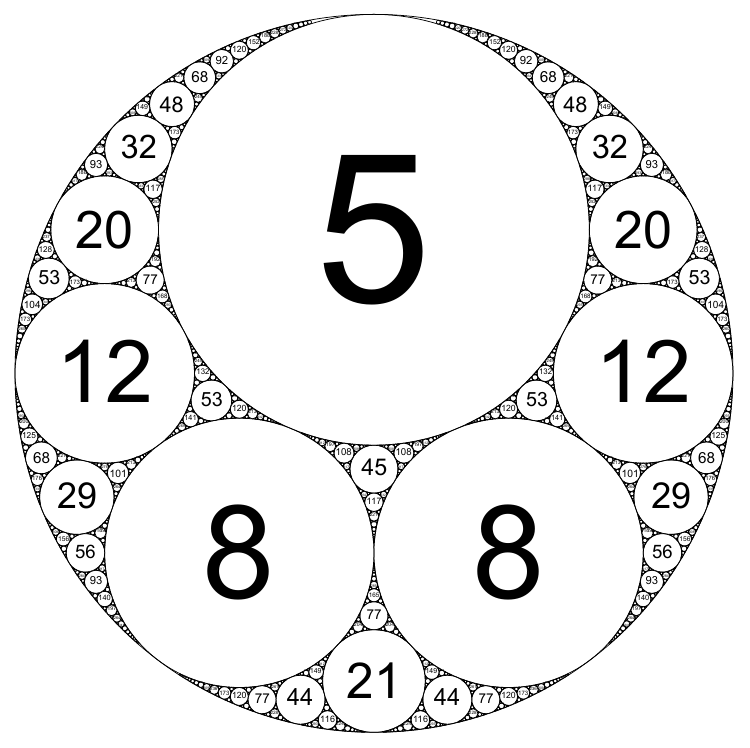}\hspace{.2cm}
		\includegraphics[width=.48\textwidth]{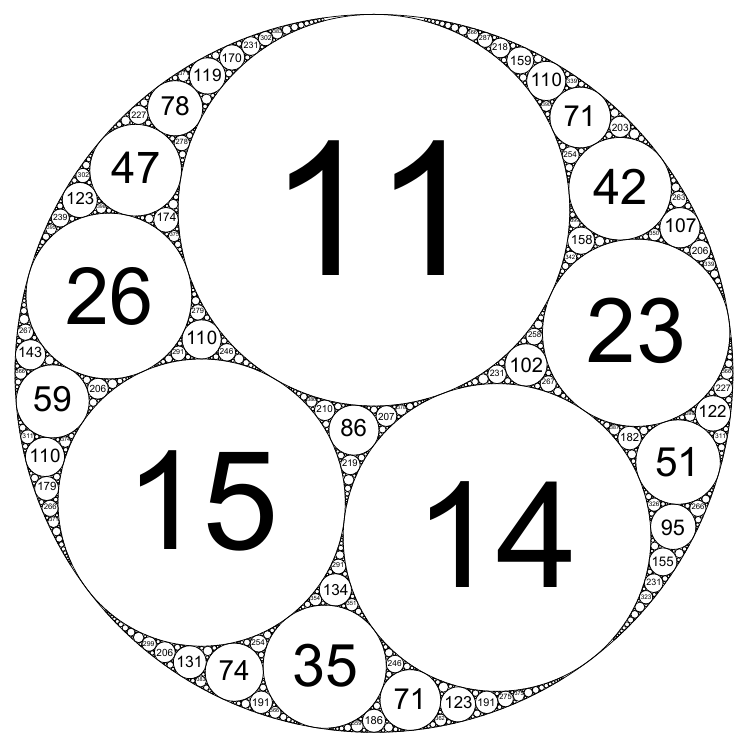}
		\caption{Two integral Apollonian packings. The labels are the bends, and the bends of the outer circles are $-3$ (left) and $-6$ (right).}
		\label{fig:apoint}
	\end{figure}

Soddy's observation arises from an old algebraic relation involving the bends of four pairwise tangent circles on the plane. This relation, known as \textit{Descartes' theorem}, was documented in the correspondence between Descartes and the Princess Elizabeth of Bohemia around 1643 \cite{bosDescartesBohemia}.

\begin{thm}[Descartes] The bends of four pairwise tangent circles on the plane satisfy
	\begin{align}\label{eq:DescartesCircleTh}
		(b_1+b_2+b_3+b_4)^2= 2(b_1^2+b_2^2+b_3^2+b_4^2).
	\end{align}
\end{thm}

Both Apollonian packings and Descartes’ theorem have been widely generalized. For example, the higher-dimensional analogues of Descartes' theorem for configurations of $d+2$ pairwise tangent spheres in $d$-dimensional Euclidean space were provided by Soddy for $d=3$  \cite{soddy1936} and by Gosset for every $d\ge2$ \cite{gosset}. Another type of generalization in the plane involves modifying the initial configuration of circles according to the vertices of a polyhedron whose edges are tangent to the unit sphere. This approach leads to Apollonian packings derived from the tetrahedron, but it can also be used to obtain the \textit{octahedral} \cite{boyd,guettler}, \textit{cubic} \cite{stange2015bianchi}, or \textit{icosahedral} \cite{aporingpacks} analogues of Apollonian packings. Each of these works has its own generalization of Descartes' theorem and can traced back to the most general version discussed by Boyd in \cite{boyd}. These generalizations fall under the family of \textit{polyhedral crystallographic packings} described in \cite{KontorovichNakamura}. In this paper, we explore a related family that we term \textit{polytopal sphere packings}. The connection between polytopes and sphere packings has been previously investigated in the works of Boyd \cite{boyd}, Maxwell \cite{maxwell},  Eppstein, Kuperberg and Ziegler \cite{eppstein2002}, Chen \cite{chen2016,chen2016even} and Chen and Labbé \cite{chenlabbe}. Our main contribution is the following generalization of Descartes' theorem for the class of polytopal sphere packings arising from uniform polytopes in every dimension.

\begin{thm}
	Let $\mathcal S_\P$ be a polytopal sphere packing where $\P$ is a uniform $(d+1)$-polytope with $d\ge2$. The polytopal curvatures of $\mathcal S_\P$ with respect to the faces in any flag $(f_0,\ldots,f_d,f_{d+1}=\P)$ satisfy
	\begin{align}\label{eq:poldesth0}
		(\kappa_{f_0}-\kappa_{f_1})^2+\ell_{f_2}^2(\kappa_{f_1}-\kappa_{f_2})^2+\sum_{i=2}^{d}\frac{1}{\ell_{f_{i+1}}^{-2}-\ell_{f_i}^{-2}}(\kappa_{f_i}-\kappa_{f_{i+1}})^2=\ell_\P^2\kappa_{\P}^2
	\end{align}
\end{thm}

One of the special features of this generalization is that the quadratic equation \eqref{eq:poldesth0}, which relates the notions of \textit{canonical lengths} $\ell_{f_i}$ with \textit{polytopal curvatures} $\kappa_{f_i}$, is expressed in terms of the geometry of the underlying polytope. Furthermore, the equation we derive encompasses several generalizations of Descartes' theorem and allows us to  determine the integrality conditions based solely on the combinatorial information of the polytope, without needing the coordinates of an initial packing. This approach is exemplified by our analysis of the truncated tetrahedron shown in Figure \ref{fig:apotruncs}.


\begin{figure}[H]
	\centering
	\includegraphics[width=.48\textwidth]{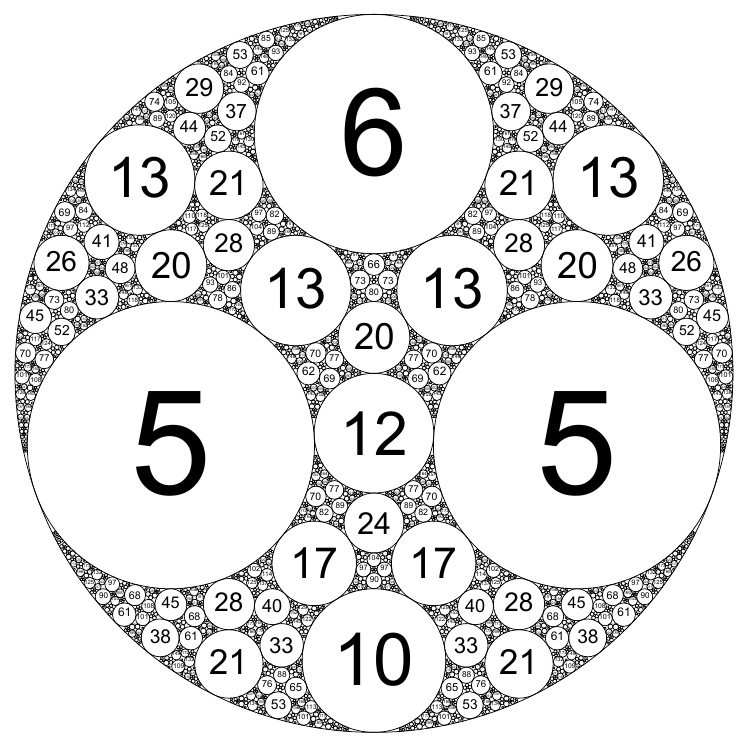}\hspace{.2cm}	\includegraphics[width=.48\textwidth]{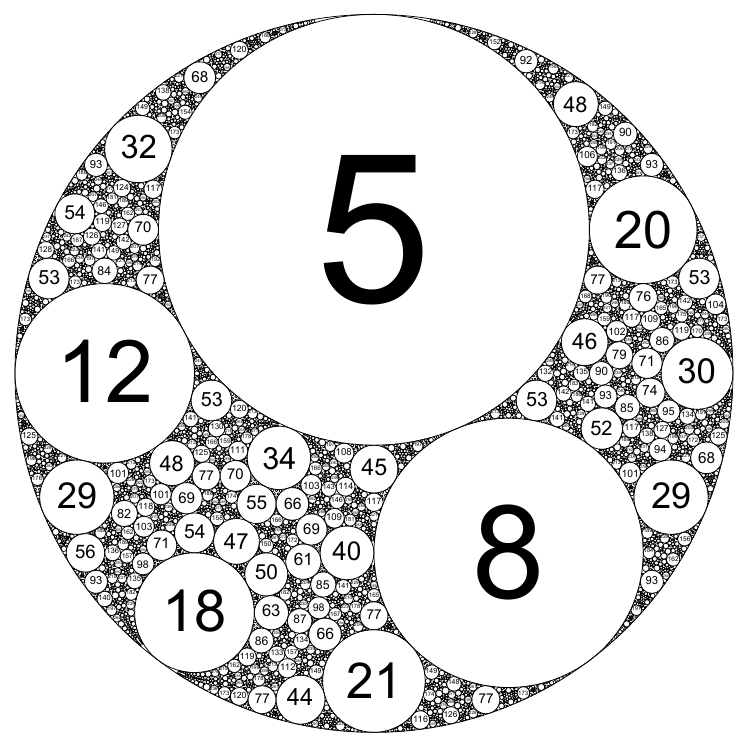}
	\caption{Two polyhedral crystallographic packings based on the cube (left) and the truncated tetrahedron (right).}
	\label{fig:apotruncs}
\end{figure}

We apply our main theorem to study the crystallographic packings induced by the Platonic solids and rediscover their integral structure under a unified Descartes' theorem expressed in terms of the corresponding Schläfli symbol (Proposition \ref{prop:platonicdesth}). In higher dimensions, the second author classify all the crystallographic sphere packings induced by regular polytopes in \cite{rasskin2024regular} and used our main theorem to examine their integrality. Beyond regular polytopes, our main theorem also applies to uniform polytopes, and we are confident it could be extended to other families, such as \textit{quasi-uniform polytopes}. In dimension $3$, there is only one quasi-uniform polyhedron that is not uniform: the 37th Johnson solid  \cite{grunbaum1967convex}. We have verified that Theorem \ref{thm:poldescartes} holds for this polyhedron.

\subsection{Organization of the paper}
In Section \ref{sec:polytopalball}, we provide the theoretical background on the Lorentzian and affine model of the space of spheres, and polytopes, which are essential for understanding the concepts of polytopal sphere packings. We then explore various aspects and endowed structures of these packings, including their dual arrangement, the full symmetry group, the Apollonian arrangement, and their uniqueness under Möbius transformations. Additionally, we introduce a novel spectral invariant of edge-scribable polytopes, which we term the {\em Möbius spectrum}.

\medskip

 In Section \ref{sec:DescartesRegPol}, we prove the main theorem (Theorem \ref{thm:poldescartes}). Section \ref{sec:platonic} applies this theorem to derive formulas for Descartes' theorem (Proposition \ref{prop:platonicdesth}), integral matrix representations of the full symmetry groups (Proposition \ref{prop:groups}), integrality conditions (Corollary \ref{cor:integrality}), and parametrizations of the integral Platonic crystallographic packings (Corollary \ref{cor:parametrization}), in terms of the Schläfli symbol for each Platonic solid. 
 
 \medskip
 The Appendix contains a catalogue of the Platonic crystallographic packings, detailing all the properties discussed in this paper.

\subsection{Acknowledgements}
We would like to thank the referee for his valuable insights and constructive criticisms,
which have greatly enhanced the presentation of this work.

\section{Polytopal sphere packings: background}\label{sec:polytopalball}
For any $d,m,n\in \mathbb N$ with $d=m+n$, let $\mathbb R^{m,n}$ denote the real vector space of dimension $d$, equipped with an inner product $\langle\cdot,\cdot\rangle$ of signature $(m,n)$. The \textit{unit sphere} $\mathbb S(\mathbb R^{m,n})$ is the space $\{\mathbf x\in\mathbb R^{m,n}\,|\, \langle \mathbf x, \mathbf x\rangle=1\}$.
We will use $\mathbb R^d$ and $\mathbb S^d$ to refer to $\mathbb R^{d,0}$ and $\mathbb S(\mathbb R^{d+1})$, respectively. An \textit{oriented hypersphere}, or simply a \textit{sphere}, in $\wrd:=\mathbb{R}^d\cup\{\infty\}$, is the image of a spherical cap on $\mathbb S^d$ under stereographic projection. Depending on the position of the North Pole of $\mathbb S^d$ relative to the spherical cap, there are three types of spheres: \textit{solid sphere} (positive radius), \textit{hollow sphere} (negative radius), or \textit{half-space} (infinite radius). 

\medskip

Each sphere $S$ is uniquely determined by its centre $c\in\wrd$ and its bend $b\in\mathbb R$ (the inverse of the signed radius). If $S$ is a half-space, it is determined by its unit normal vector $\widehat n\in \mathbb R^d$, which points to the interior, and the signed distance $\delta\in\ru$ between its boundary and the origin. 

\subsection{The Lorentzian model of the space of spheres} The space $\mathbb R^{d,1}$, along with its corresponding inner product, is referred to as the \textit{Lorentzian space} and the \textit{Lorentzian product} in dimension $d+1$, respectively.  A vector  $X=(x_1,\ldots,x_{d+1})^\top\in\mathbb R^{d,1}$ is \textit{future-directed} (resp. \textit{past-directed)} if $x_{d+1}>0$ (resp. $x_{d+1}<0$). There is a well-known bijection between the space of spheres of $\wrd$ and the Lorentzian unit sphere $\mathbb S(\mathbb R^{d+1,1})$ (see \cite{wilker} as well as \cite[Section 2]{RR20} for further details on this bijection). For any sphere $S$ of $\wrd$, we denote by $X_S\in\mathbb S(\mathbb R^{d+1,1})$ the Lorentzian vector corresponding to $S$. The \textit{inversive product }of two spheres is defined as $\langle S,S' \rangle:=\langle X_S, X_{S'}\rangle$. If $X_S, X_{S'}$ are not both past-directed then
\begin{equation}\label{eq:lprod}
	\langle S,S'\rangle
	\begin{cases}
		<-1&\text{if } S\cap S'=\emptyset, \\
		=-1&\text{if  $\partial S$ and $\partial S'$ are tangent and } \mathrm{int}(S)\cap \mathrm{int}(S')=\emptyset,  \\
		=0 \quad&\text{if  $\partial S$ and $\partial S'$ are orthogonal,}  \\
		=1&\text{if  $\partial S$ and $\partial S'$ are tangent and }  S\subseteq S'\text{ or }S'\subseteq S,  \\
		>1&\text{if } \partial S\cap \partial S'=\emptyset\text{ and } S\subset S'\text{ or }S'\subset S.  \\
	\end{cases}
\end{equation}  

The \textit{inversive coordinates} of $S$ correspond to the coordinates of $X_S$, which are given by the $(d+2)$-dimensional real vector
\begin{align}\label{eq:invcoord}
	X_S=
	\begin{cases}
		(bc,\dfrac{\overline b-b}{2},\dfrac{\overline b+b}{2})^\top&\text{ if }b\not=0,  \\
		\quad\\
		(\widehat n,\delta,\delta)^\top&\text{otherwise}. \\
	\end{cases}
\end{align}
where $\overline{b}=b\|c\|^2-\frac{1}{b}$ is the \textit{co-bend} of $S$. The co-bend is the bend of $S$ after inversion through the unit sphere.
With this coordinate system, the inversive product is obtained by 
\begin{align}\label{eq:invprodmatrix}
	\langle S,S'\rangle = X_S^\top \mathbf{Q}_{d+2} X_{S'}
\end{align}
where $\mathbf Q_{d+2}=\mathrm{diag}(1,\ldots,1,-1)$. The group of Möbius transformations of $\wrd$ preserves the inversive product and acts linearly on the inversive coordinates as an orthogonal subgroup of $\mathrm {SL}_{d+2}(\mathbb R)$ with respect to $\mathbf Q_{d+2}$. An \textit{arrangement} of spheres $\mathcal A=(S_1,S_2,\ldots)$ of $\wrd$, possibly infinite, is a \textit{packing} if their interiors are mutually disjoint. The \textit{Gramian} of a finite arrangement $\mathcal A=(S_1,\ldots,S_n)$ is the matrix $\mathrm{Gram}(\mathcal{A})=(\langle S_i,S_j\rangle)_{1\le i,j\le n}$.  

\subsection{The affine model of the space of spheres.}

The Lorentzian unit sphere  $\mathbb S(\mathbb R^{d+1,1})$ can be regarded within the \textit{oriented projective space} $\mathbb{P}_+\ed=\{X\in\ed\setminus 0\}/_\sim$,
where $X\sim Y$ if there is a real number $\lambda>0$ such that $X=\lambda Y$. This space $\mathbb{P}_+\ed$ is in bijection with the Euclidean unit sphere $\mathbb{S}^{d+1}\subset\ed$ which, under the gnomonic projection, becomes the union of two affine hyperplanes $\Pi_{\pm1}=\{x_{d+2}=\pm1\}$, both of which can be identified with $\eud$, along with $\Sigma_0=\{(X,0)\mid X \in \mathbb{S}^{d+1}\}$. The composition of the isomorphism, which maps the space of spheres of $\wrd$ to $\mathbb{S}(\ed)$, with the projection
\begin{align*}
	\begin{array}{ccc}
		\mathbb{S}(\ed)&\longrightarrow& \Pi_1\cup\Sigma_0\cup\Pi_{-1}\\
		X&\longmapsto&\left\lbrace\begin{array}{ccc}
			X&\text{if }x_{d+2}=0\\
			\frac{1}{|x_{d+2}|} X&\text{otherwise}
		\end{array}\right.
	\end{array}	
\end{align*}
provides an isomorphism between the space of spheres of $\wrd$ and $
\Pi_1^\circ\cup \Sigma_0\cup\Pi_{-1}^\circ$, where $\Pi_{\pm1}^\circ$ is the set of points in 
$\Pi_{\pm1}$ whose Euclidean norm is greater than 1. Such a point will be referred to as an \textit{outer point} of $\eud$. We call $
\Pi_1^\circ\cup \Sigma_0\cup\Pi_{-1}^\circ$ the \textit{affine model of the space of spheres}. In this manner, we can construct a bijection between the set of spheres whose Lorentzian vector is future-directed and the set of outer points of $\eud$. The reciprocal bijection between an outer point $v\in\eud$ and a sphere $S_v$ of $\wrd$ can be obtained geometrically by positioning a  \textit{light source} which illuminates $\sd$ from $v$. The illuminated region on $\sd$ is the spherical cap $\{u\in\mathbb R^{d+1}\mid u\cdot v\ge1\}\cap\mathbb S^d$, where $\cdot$ denotes the Euclidean inner product of $\eud$. We denote by $S_v$ the sphere obtained by the stereographic projection of the illuminated region of $v$, and call it the \textit{stereographic sphere} of $v$. The \textit{Lorentzian vector} of $v$ is defined as $X_v:=X_{S_v}$ (see Figure \ref{fig:stereo}). 

\begin{figure}[H]
	\centering
	\includestandalone[scale=1]{tikzs/lightsource5} 
	\vspace{-1.5cm} 
	\caption{(Left) An outer point of $\eud$ and its stereographic sphere in $\wrd$; (right) same setting in the affine model of the space of spheres, together with the corresponding Lorentzian vector.}  
	\label{fig:stereo}
\end{figure}

The Lorentzian vector of $v$ can be derived from $v$ using the following equation 
\begin{equation} \label{affinecoordinates}
	X_{v}=(\|v\|^2-1)^{-1/2}(v,1)^\top
\end{equation} 
This equation implies that for any two spheres, whose Lorentzian vectors are future-directed, their inversive product is related to the Euclidean inner product of their corresponding light sources via the following formula
\begin{equation}\label{eq:prodlorentztoeuclid}
	\langle S_u,S_v\rangle =((\|v\|^2-1)(\|v\|^2-1))^{-1/2}(u\cdot v-1).
\end{equation}


\subsection{Polytopes} 

We recall some basic notions of polytopes needed for the rest of the paper. We refer the reader to \cite{schulte04} for further details. A $d$\textit{-polytope} $\P$ is the convex hull of a finite collection of points in $\rd$. A $2$-polytope and a $3$-polytope are usually called \textit{polygon} and \textit{polyhedron}, respectively. For every $0\le k\le d$, we denote by $F_k(\P)$ the set of $k$-faces of $\P$ and by $\mathcal{F}(\P)=\{\emptyset\}\cup\bigcup_{k=0}^{d}F_k(\P)$ where $F_d(\P)=\{\P\}$.  The elements of $V(\P):=F_0(\P)$, $E(\P):=F_1(\P)$, $F_{d-2}(\P)$ and $F_{d-1}(\P)$ are called \textit{vertices}, \textit{edges}, \textit{ridges} and \textit{facets} of $\P$, respectively. The \textit{graph of} $\P$ is the graph induced by the vertices and the edges of $\P$. The \textit{face lattice} $(\mathcal F(\P),\subset)$ encodes all the
combinatorial information about $\P$. A \textit{flag} of $\P$ is a sequence of faces $\Phi=(f_0,f_1,\ldots,f_{d-1},f_d=\P)$ where for each $k=0,\ldots,d-1$,  $f_k\subset f_{k+1}$. Two polytopes $\P$ and $\P'$ are combinatorially equivalent if there exists an isomorphism between their face lattices, and one is said to be a \textit{realisation} of the other.

\medskip

The \textit{polar} of a subset $\Omega\subset\rd$  is defined as $\Omega^*=\{u\in\rd\mid u\cdot v\leq 1\text{ for all }v\in \Omega\}$. If $\mathcal{P}$ is a  $d$-polytope containing the origin in its interior, then  $\mathcal{P}^*$ is also a $d$-polytope containing the origin in its interior and holds the dual relation $(\mathcal{P}^*)^*=\mathcal{P}$. There is a bijection between $\mathcal{F}(\P)$ and $\mathcal{F}(\P^*)$ which reverses incidences and maps every facet $f$ of $\P$ to a vertex $v_f$ of $\P^*$. For every vertex $u\in f$, one has 
\begin{align}\label{eq:facepolar}
	u\cdot v_f=1.
\end{align}

For every $0\leq k\leq d-1$, a $d$-polytope $\mathcal{P}$ is $k$-\textit{scribed} if all its $k$-faces are tangent to the unit sphere $\mathbb S^{d-1}\subset\eud$. A $k$-scribed $d$-polytope is called \textit{inscribed}, \textit{edge-scribed}, \textit{ridge-scribed} and \textit{circumscribed} if $k=0,1,d-2,d-1$ respectively. A $d$-polytope is said to be \textit{$k$-scribable} if it admits a realisation which is $k$-scribed. If $\P$ is a $d$-polytope containing the origin in its interior then $\P$ is $k$-scribed if and only if $\P^*$ is $(d-k-1)$-scribed \cite{chenpadrol}. Concerning edge-scribability, all $d$-polytopes are edge-scribable for $d=2,3$ \cite{bright-sch}. In dimension $d\ge4$, there are examples of non-edge-scribable polytopes \cite{schulte87}.

\medskip

An edge-scribed $d$-polytope is said to be \textit{canonical} \cite{ziegler2012lectures} if the barycentre of all its tangency points with $\mathbb S^d$ is the origin. It follows from the work of Springborn in \cite{springborn2005unique}, that for any edge-scribable polytope $\P$, there is a unique canonical realisation $\P_0$, up to Euclidean isometries.  For $3$-polytopes, canonical realisations are also called \textit{Springborn realisations} \cite{belotti2022algebraic}. 
\begin{figure}[H]
	\includegraphics[width=.3\linewidth]{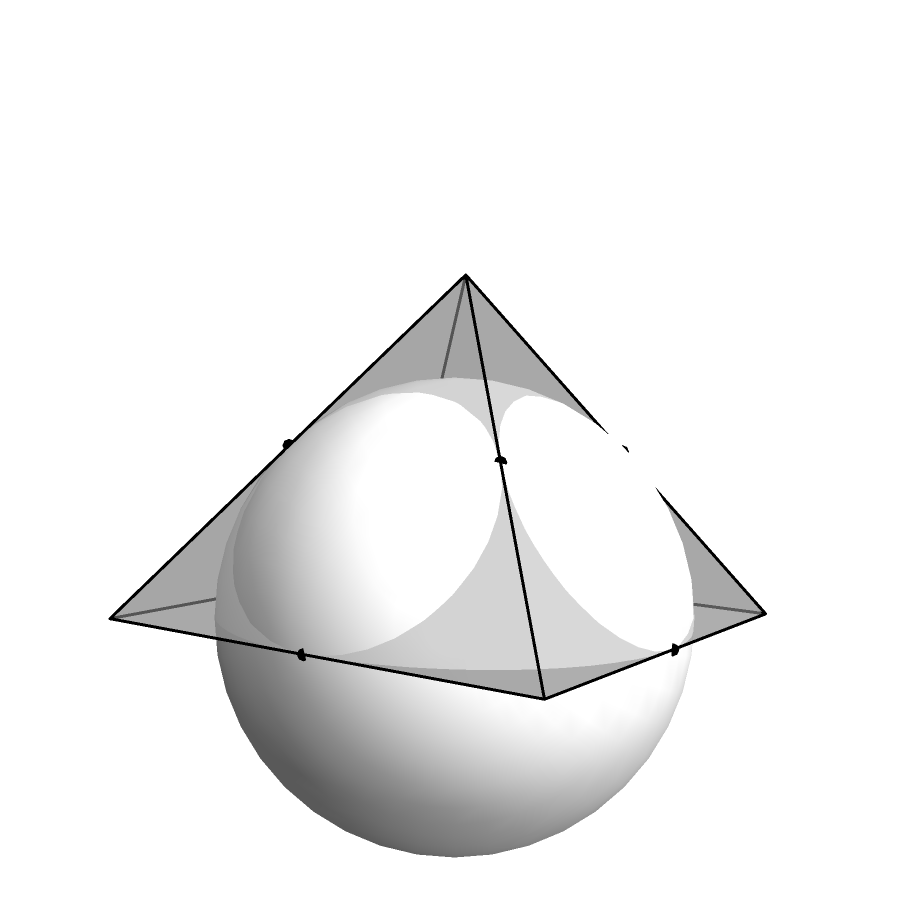}	\includegraphics[width=.3\linewidth]{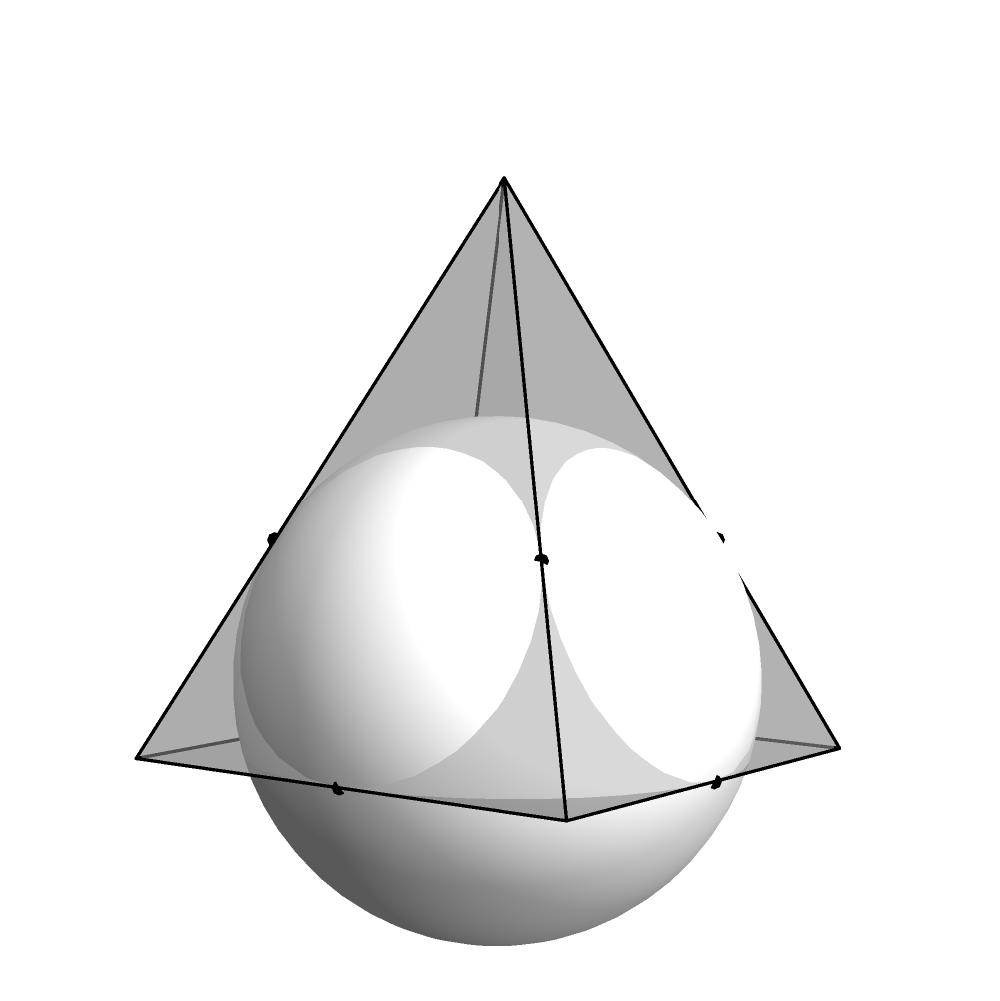}\includegraphics[width=.3\linewidth]{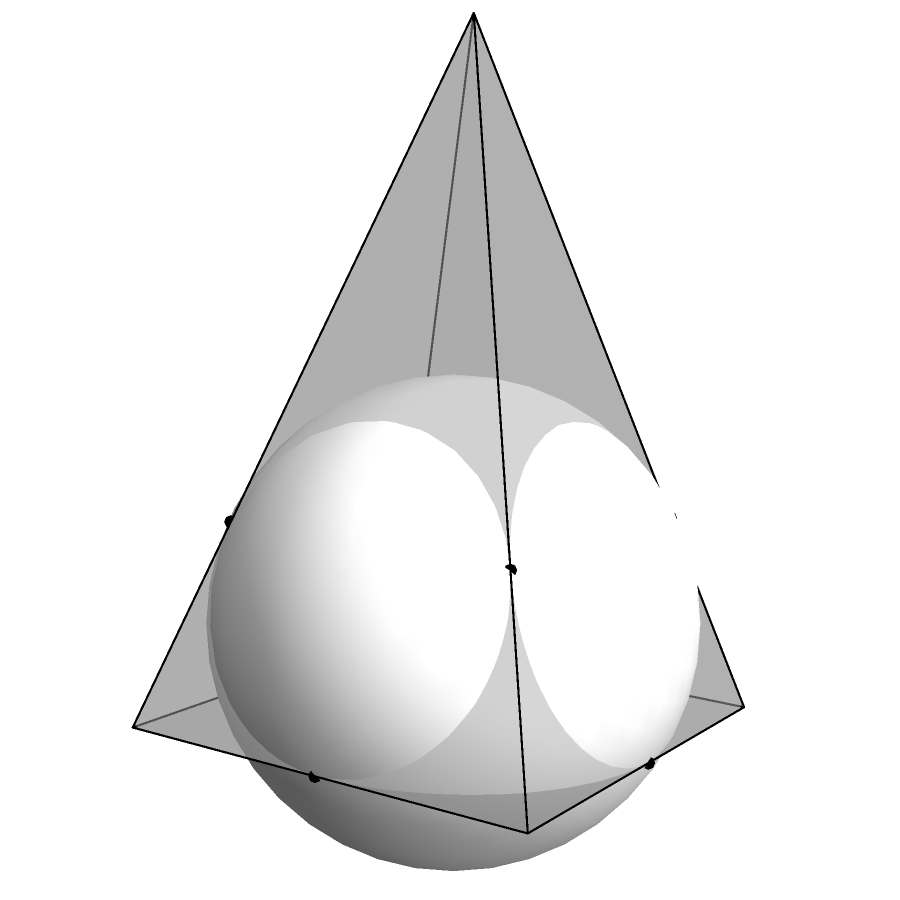}
	\vspace{-.3cm}
	\caption{Three edge-scribed realisations of a $4$-pyramid. The barycentre of the second one is the origin. In the third realisation, the barycentre of the contact points of the edges with the sphere  is the origin, so it is canonical.} 
	\label{fig:canonical}
\end{figure}

The \textit{symmetry group} $\mathfrak S(\P)$ of a $d$-polytope $\P$ is defined as the group of Euclidean isometries of $\rd$ preserving $\P$. A polytope is \textit{regular} if its symmetry group acts transitively on the set of its flags. A $d$-polytope with $d\ge2$ is called \textit{uniform} if it is regular for $d=2$ and for $d>2$, its facets are uniform, and its symmetry group is vertex-transitive. It is well known that the only uniform $3$-polytopes are the 5 Platonic solids, the 13 Archimedean solids, and the infinite families of prisms and antiprisms \cite{grunbaum1967convex}. Clearly, any regular $(d+1)$-polytope is \textit{$k$-scribable} for every $0\leq k\leq d$. With a little extra effort, it can be proved that uniform polytopes are inscribable, edge-scribable and, in general, non-circumscribable. Moreover, we have the following.

\begin{lem}\label{lem:unicanonical}
	Every edge-scribed uniform polytope is canonical and its barycentre is the origin.
\end{lem}

\begin{proof}The $2$-dimensional case is trivial. Let $\P$ be a uniform $d$-polytope with $d\ge3$. Let us first clarify that $\P$ admits a \textit{midsphere}, i.e. a sphere tangent to every edge. Every symmetry of $\P$ fixes its barycentre $\mathrm{bar}(\P)$. The vertex-transivity then implies that the distance from each vertex to $\mathrm{bar}(\P)$ is constant. Since every $2$-face of $\P$ is regular, every edge of $\P$ has equal length. By combining these two facts, we get that for every $e\in E(\P)$, $\|\mathrm{bar}(e)-\mathrm{bar}(\P)\|$ is constant and $e$ is orthogonal to the line passing through $\mathrm{bar}(e)$ and $\mathrm{bar}(\P)$. Thus, the sphere centred at $\mathrm{bar}(\P)$ with radius $\|\mathrm{bar}(e)-\mathrm{bar}(\P)\|$ is tangent to every edge of $\P$ at the barycentres of the edges, so it is the midsphere of $\P$.
	
	\medskip
	Let us now suppose that $\P$ is also edge-scribed, so its midsphere is the unit sphere $\mathbb S^{d-1}$. Therefore, $\mathrm{bar}(\P)$ is the origin $\mathbf{0}\in\mathbb R^d$. On the other hand, the vertex-transitivity implies that every vertex has same degree $\delta\ge d$, which gives us $2|E(\P)|= \delta |V(\P)|$. We show that $\P$ is canonical by proving that the barycentre of the contact points of the edges of $\P$ with $\mathbb S^{d-1}$ is also the origin.
	\begin{align*}
		\frac{1}{|E(\P)|}\sum_{x\in E(\P)\cap\mathbb S^d} x=\frac{1}{|E(\P)|}\sum_{e\in E(\P)} \mathrm{bar}(e)=&\frac{1}{|E(\P)|}\sum_{u,v\in e\in E(\P)} \frac12(u+v)\\
		=&\frac{1}{2|E(\P)|}\sum_{v\in V(\P)} \delta v=\frac{\delta}{\delta|V(\P)|}\sum_{v\in V(\P)} v=\mathrm{bar}(\P)=\mathbf 0\qedhere
	\end{align*}
\end{proof}
We note that an edge-scribed polytope satisfying one of the properties of Lemma \ref{lem:unicanonical} might not necessarily satisfy the other one (see Figure \ref{fig:canonical}). As it is described in the proof above, uniform polytopes have equal edge lengths. We define the \textit{canonical length} $\ell_\P$ of a uniform polytope $\P$  as the half the edge length of its canonical realisation. The following lemma establishes a relation between the canonical lengths and the Lorentzian vectors of uniform polytopes.


\begin{lem}	\label{lem:keylemma} For each $f\in\mathcal F(\P)$ of an edge-scribed uniform $d$-polytope $\mathcal P$,
	$\langle X_f,X_\P\rangle=-\ell_\P^{-2}$
	where $X_f:=\frac{1}{|V(f)|}\sum_{v\in V(f)}X_{v}.$
\end{lem}
\begin{proof}By Lemma \ref{lem:unicanonical}, $\P$ is canonical and the barycentre of $\P$ is the origin $\mathbf 0\in\rd$. Thus, the Euclidean norm of any vertex $v$ of $\P$ and the canonical length of $\P$ are related by $\|v\|^2=\ell_\P^2+1$. Therefore, by Equation (\ref{affinecoordinates}), the Lorentzian vector $X_v=\ell_\P^{-1}(v,1)^\top$ which belongs to the affine hyperplane $\Pi=\{\langle X,\varepsilon_{d+1}\rangle=-\ell_\P^{-1}\}\subset\mathbb R^{d,1}$, where $\varepsilon_{i}$ is the $i^{\text{th}}$ canonical of $\mathbb R^{d,1}$. Therefore, for each dimensional face $f$ of $\P$, $X_f\in \Pi$, which implies that $X_f-X_{\mathcal P}\in\Pi$. On the other hand, since $\mathrm{bar}(\P)=\mathbf 0\in\rd$, then $X_{\P}=(\mathbf 0,\ell_\P^{-1})^\top$. The line spanned by $X_{\mathcal P}$ is orthogonal to $\Pi$. Therefore, $	\langle X_f-X_{\P},X_{\P}\rangle=0$ which implies that $\langle X_f,X_{\P}\rangle=\langle X_{\P},X_{\P}\rangle=-\ell_{\P}^{-2}.$
\end{proof}

\subsection{Polytopal sphere packings}

Let $\P$ be a $d$-polytope whose vertices are outer points. The \textit{arrangement projection} of $\P$ is defined as the arrangement $\mathcal A_\P$ formed by the stereographic spheres of the vertices of $\P$ (see Figure \ref{fig:nonpolytopal}). For any edge $uv$ of $\P$, the spheres $S_u$ and $S_v$ are disjoint, tangent or have overlapping interiors, if and only if $uv$ cuts transversely, is tangent or avoids  $\sd$, respectively. Therefore, if $\P$ is edge-scribed, then $\mathcal A_\P$ is a packing. 	For every $d\ge2$, a sphere packing $\mathcal{S}_\P$ in $\wrd$ is \textit{polytopal} if there is an edge-scribable $(d+1)$-polytope $\P$ and a Möbius transformation $\mu$ such that $\mathcal{S}_\P=\mu\cdot\mathcal A_{\P_0}$, where $\P_0$ is a canonical realisation of $\P$. Clearly, there are sphere packings which are not polytopal (see Figure \ref{fig:nonpolytopal}). As Chen noticed in \cite{chen2016}, if $\mathcal A_\P$ is a packing, then the tangency graph of $\mathcal A_\P$ is a spanning subgraph of the graph of $\P$. Furthermore, if $\P$ is edge-scribed, then the two graphs are isomorphic. Therefore, the tangency relations of a polytopal sphere packing $\mathcal{S}_\P$ are encoded by the edges of $\P$.


\vspace{-.25cm}

\begin{figure}[H]
	\centering
	\begin{tabular}{ccc}
		\includegraphics[trim=10 15 0 50,clip,height=4cm,align=c]{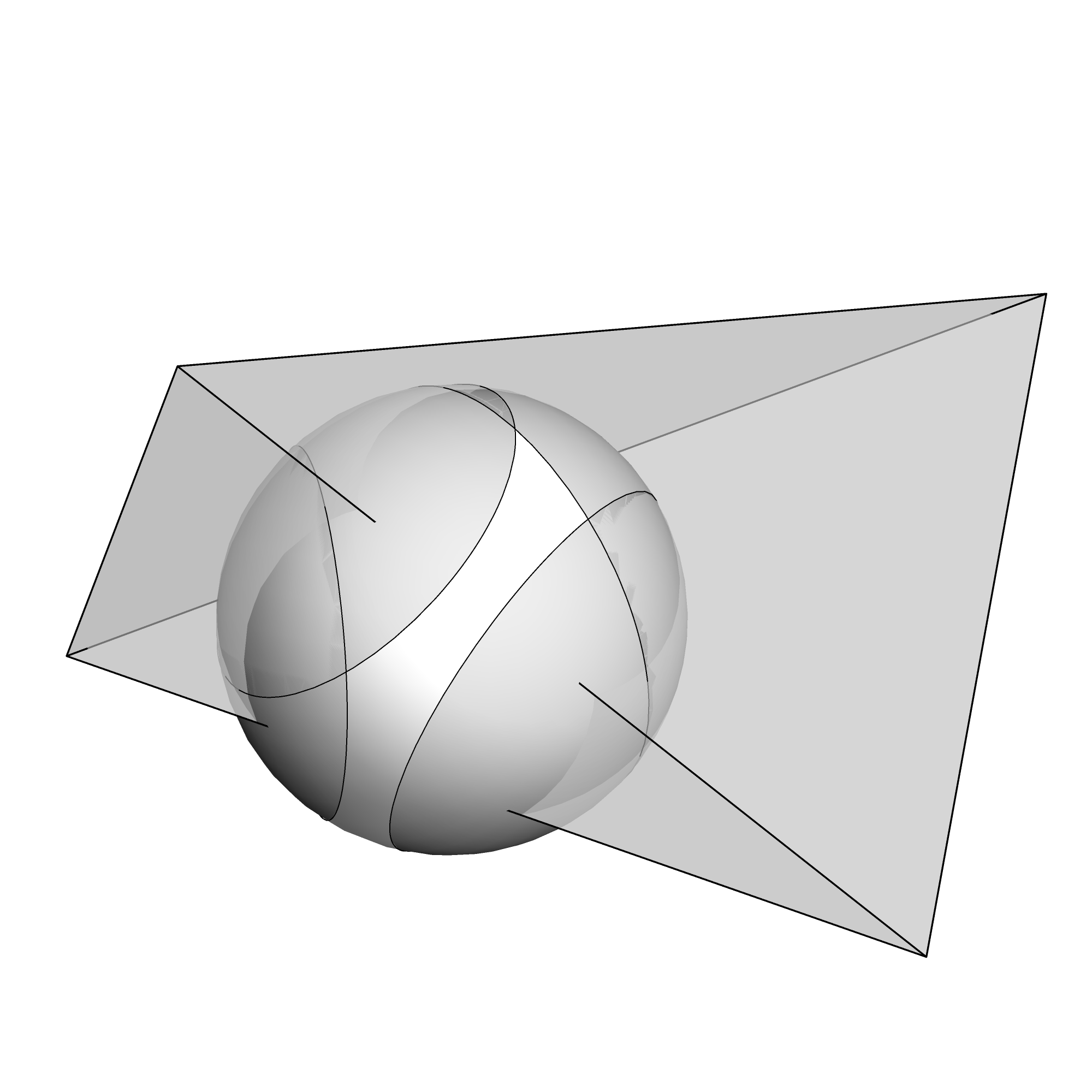}&\includegraphics[height=2.6cm,align=c]{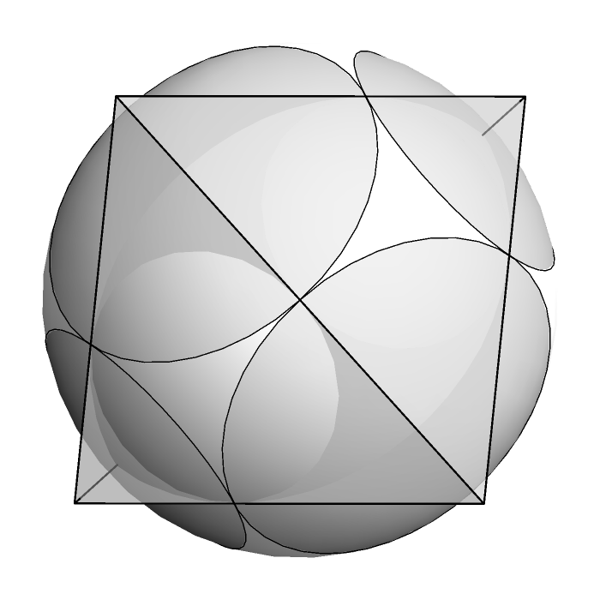}	 &\includegraphics[trim=0 0 20 30,clip,height=3cm,align=c,vshift=.1cm]{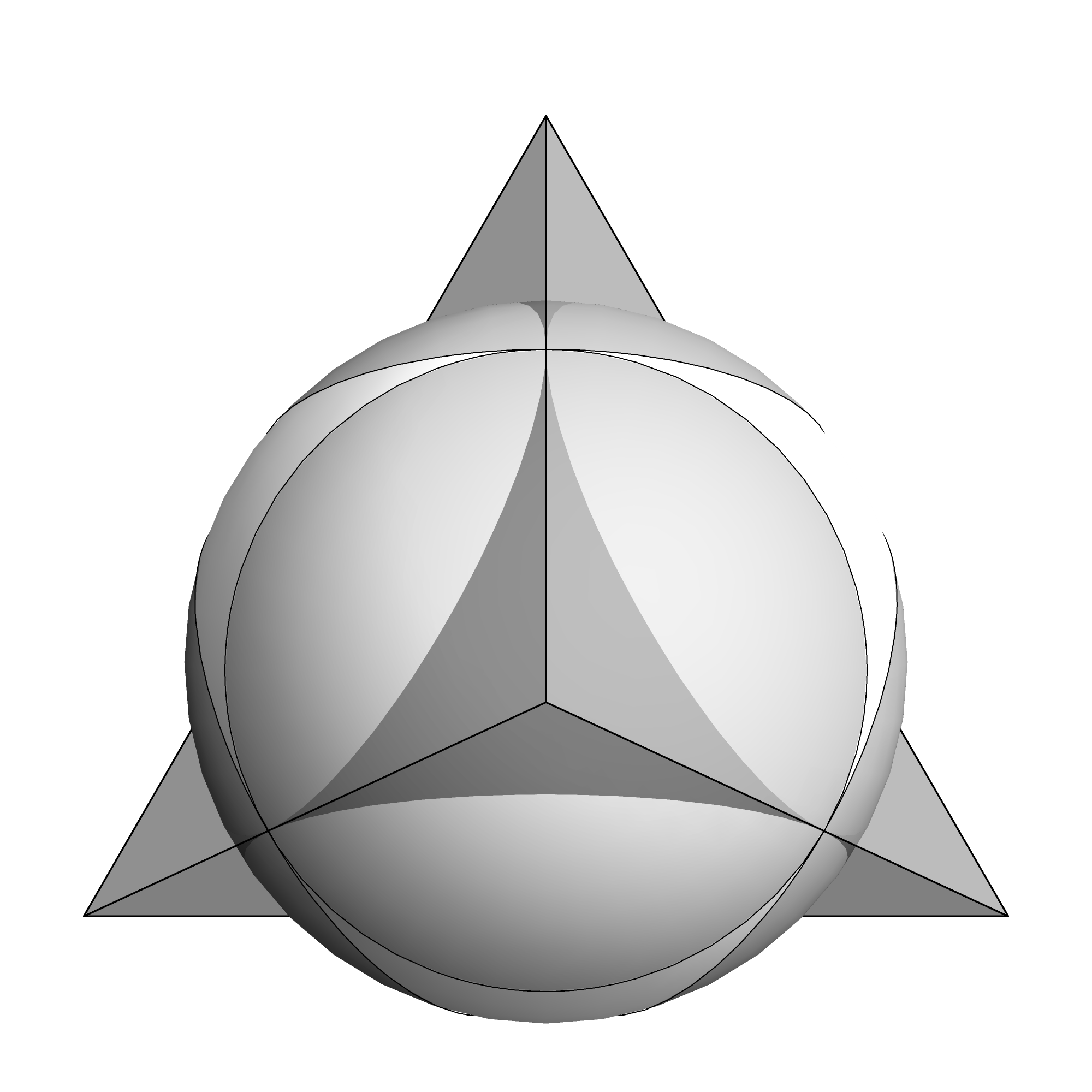} \\[-1.3cm]
		\reflectbox{\rotatebox[origin]{-150}{\begin{tikzpicture}[scale=.36] 
\clip (-8,-\prange) rectangle (5,8);
\draw[very thin, draw=black, fill=darkgray, fill opacity=\opac] (-2.062,0.2840) circle (1.689cm);
\draw[very thin, draw=black, fill=darkgray, fill opacity=\opac] (-2.754,3.696) circle (3.989cm);
\draw[very thin, draw=black, fill=darkgray, fill opacity=\opac] (0.4923,0.03083) circle (0.7515cm);
\draw[very thin, draw=black, fill=darkgray, fill opacity=\opac] (0.9696,-2.644) circle (2.733cm);

\end{tikzpicture}}}&\begin{tikzpicture}[scale=1] 
\begin{scope}[scale=2.1547]

\draw[very thin, draw=black, fill=darkgray, fill opacity=\opac] (0.4641,-0.2679) circle (0.4641cm);
\draw[very thin, draw=black, fill=darkgray, fill opacity=\opac] (-0.4641,-0.2679) circle (0.4641cm);
\draw[very thin, draw=black, fill=darkgray, fill opacity=\opac] (0,0.5359) circle (0.4641cm);
\draw[very thin, draw=black, fill=darkgray, fill opacity=\opac] (0.9282,0.5359) circle (0.4641cm);

\end{scope}

\end{tikzpicture}& \begin{tikzpicture}[scale=1] 
\begin{scope}[scale=2.1547]

\draw[very thin, draw=black, fill=darkgray, fill opacity=\opac] (0.4641,-0.2679) circle (0.4641cm);
\draw[very thin, draw=black, fill=darkgray, fill opacity=\opac] (-0.4641,-0.2679) circle (0.4641cm);
\draw[very thin, draw=black, fill=darkgray, fill opacity=\opac] (0,0.5359) circle (0.4641cm);
\draw[very thin, draw=black, fill=darkgray, fill opacity=\opac] (0,0) circle (0.07180cm);

\end{scope}

\end{tikzpicture}\\[-.5cm]
	\end{tabular}
	\vspace{-.25cm}
	\caption{
		(Top figures) Three polyhedra with the illuminated regions of their vertices.
		(Bottom figures) The arrangement projection of the three polyhedra. The last two are packings but only the third one is polytopal. 	}
	\label{fig:nonpolytopal}
\end{figure}

Let $\mathcal S_\P=\mu\cdot\mathcal A_{\P_0}$ be a polytopal sphere packing.  We define the \textit{dual arrangement} of $\mathcal{A}_\P$ as the arrangement $\mathcal{S}_\P^*:=\mu\cdot\mathcal A_{\P_0^*}$. Notice that $\mathcal{S}_\P^*$ is well-defined since $\P_0$ contains the origin in its interior. We call the spheres of $\mathcal{S}_\P^*$ the \textit{dual spheres} of $\mathcal S_\P$ and denote by $S_f:=S_{v_f}$ where $v_f$ is the vertex of $\P^*$ corresponding to a facet $f$ of $\mathcal P$.
By combining Equations \eqref{eq:lprod}, \eqref{eq:prodlorentztoeuclid}, \ref{eq:facepolar}, we have that for any vertex $v$ and any facet $f$ of $\P$ containing $v$, the spheres $S_v$ and $S_{f}$ are orthogonal. Since $\P_0$ is edge-scribed, $\P_0^*$ is ridge-scribed. If $\mathcal{S}_\P$ is in $\widehat{\mathbb{R}^2}$, then $\P$ is a $3$-polytope and $\P_0^*$ is also edge-scribed. Therefore, in this dimension, $\mathcal{S}_\P^*$ is also a packing. The union $\mathcal{S}_\P\cup\mathcal{S}_\P^*$ has been called a \textit{primal-dual circle representation}\cite{Felsner2019prim-dual}. Brightwell and Scheinerman proved  in \cite{bright-sch} the existence and the uniqueness up to Möbius transformations of primal-dual circle representations for every polyhedron. This can be seen as a stronger version of the Koebe-Andreev-Thurston Circle packing theorem \cite{bobenko}. In dimensions $d\ge3$, $\mathcal{S}_\P^*$ is no longer a packing since the dual spheres overlap \cite{rasskin2024regular}.
\begin{figure}[H]
	\centering
	\includegraphics[align=c,width=.64\textwidth]{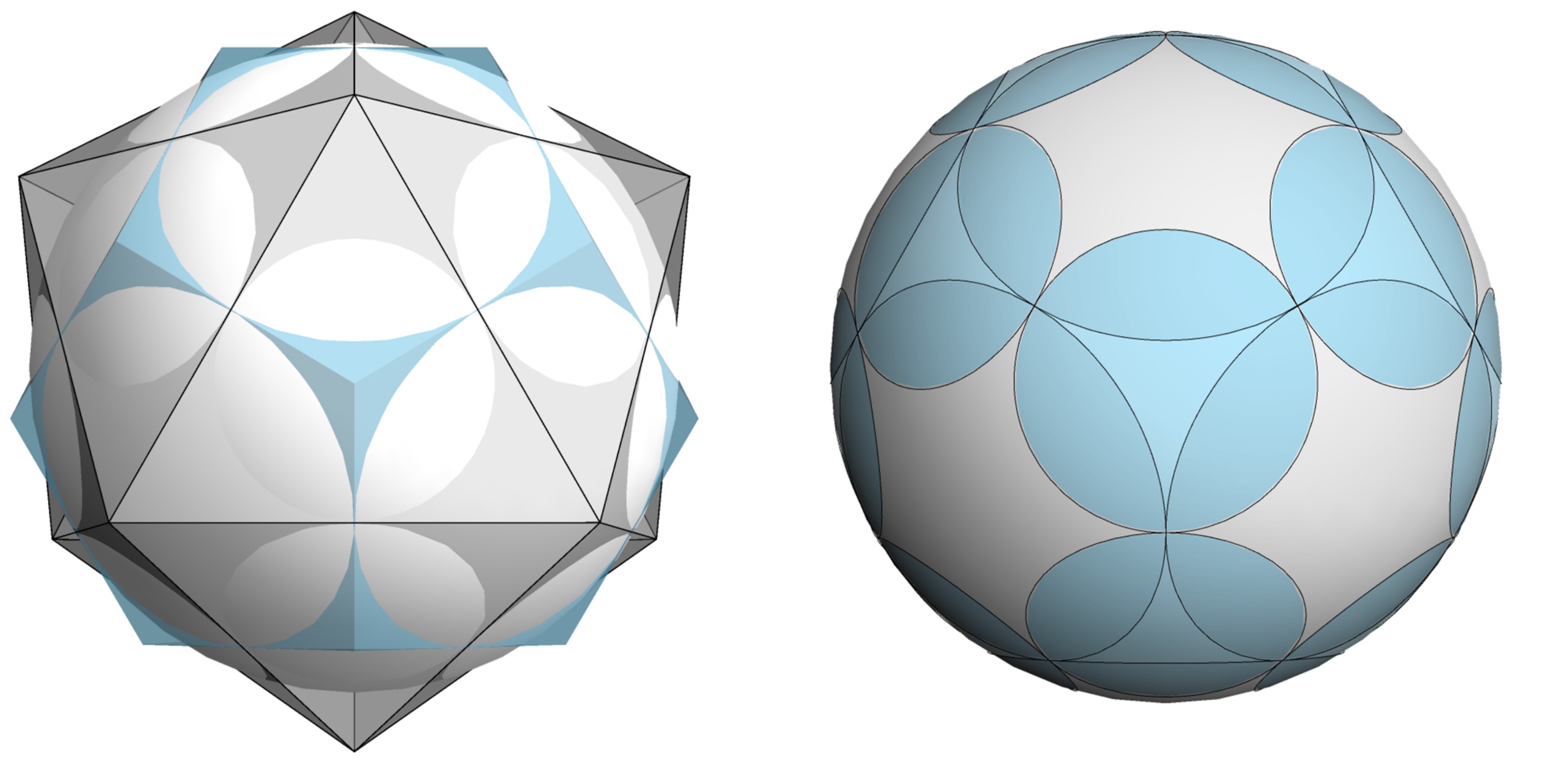}\hspace{.5cm}	\includegraphics[align=c,width=.3\textwidth]{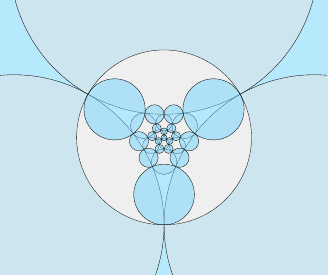}
	\caption{(Left) A canonical icosahedron and its polar in blue; (centre) the illuminated regions of the vertices of both; (right) the corresponding primal-dual circle representation.}   
	\label{fig:icodual}
\end{figure}	

We define the \textit{Apollonian arrangement} of $\S_\P$ as the union of the orbits $\mathscr P(\mathcal S_\P)=\langle \S_\P^*\rangle\cdot \mathcal S_\P$ where $\langle \S_\P^*\rangle$ is the group generated by the inversions through the dual spheres of $\S_\P$ (see Figure \ref{fig:apocolors}). The group $\langle \S_\P^*\rangle$ can be seen as the analogue of the \textit{Apollonian group} defined by Hirsch in \cite{hirst1967apollonian} for Apollonian packings. Similarly to dual arrangements, Apollonian arrangements in dimension $2$ are always packings. In dimensions $d>2$, even if $\S_\P^*$ is no longer a packing, $\mathscr P(\mathcal S_\P)$ might still be a packing. This occurs when the intersecting angles of all overlapping dual spheres satisfy the \textit{crystallographic restriction} \cite{boyd}. In such cases, $\mathscr P(\mathcal S_\P)$ belongs to the family of \textit{crystallographic sphere packings} introduced by Kontorovich and Nakamura in \cite{KontorovichNakamura}.

\begin{figure}[H]
	\centering
	\includegraphics[width=0.37\textwidth]{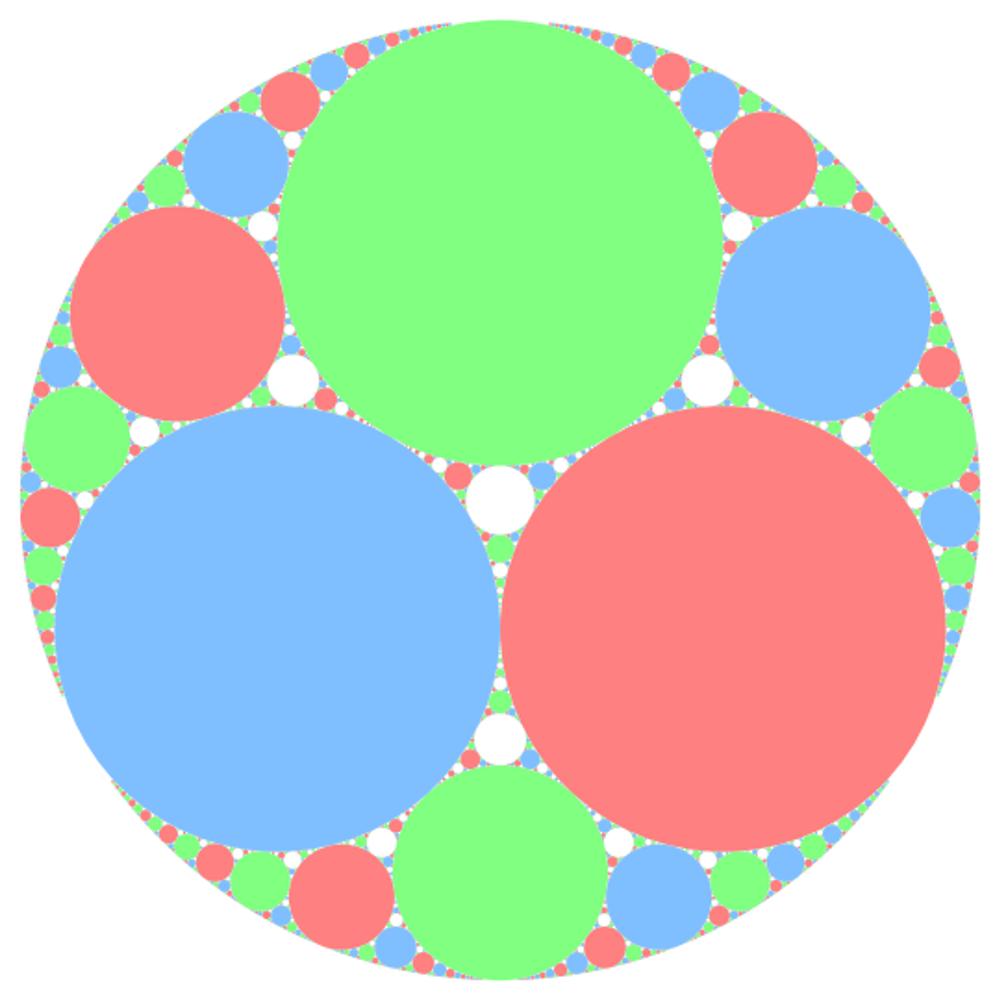}\hspace{.5cm}
	\includegraphics[width=0.37\textwidth]{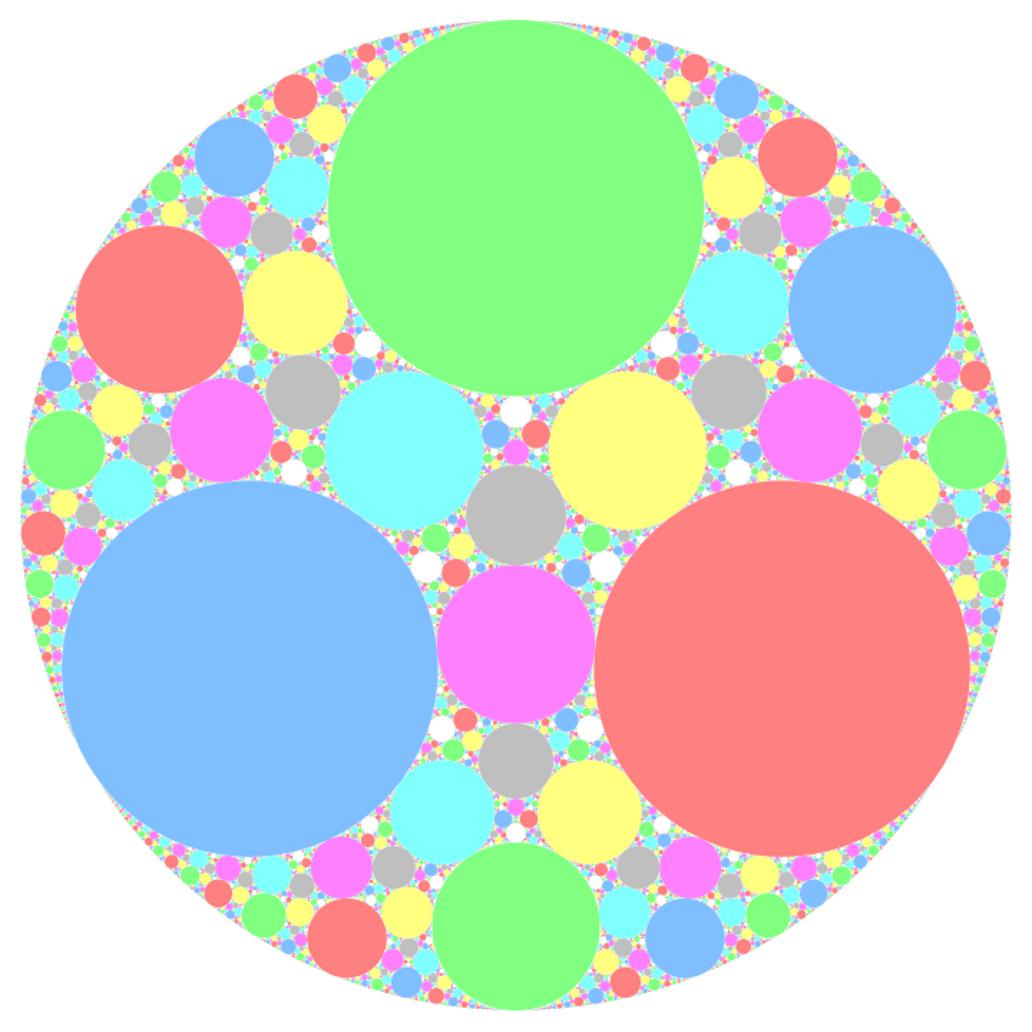}
	\caption{A tetrahedral (left) and cubic (right) crystallographic packing. Each colour represents an orbit under the action of the group generated by the inversions through the dual spheres. }
	\label{fig:apocolors}
\end{figure}

The Apollonian arrangement $\mathscr P(\mathcal S_\P)=\langle \S_\P^*\rangle\cdot \mathcal S_\P$ can be equally obtained as the union of the orbits of the action $(\langle \S_\P^*\rangle\rtimes \mathfrak S(\S_\P))\cdot \{\S_\P/\mathfrak S(\S_\P)\}$, where $\mathfrak S(\mathcal S_\P)\simeq \mathfrak S(\P)$ is the \textit{symmetry group} of $\mathcal{S}_\P$, defined as the group of Möbius transformations preserving $\mathcal{S}_\P$. We refer to the group  $\Gamma(\S_\P):=\langle \S_\P^*\rangle\rtimes \mathfrak S(\S_\P)$ as the \textit{full symmetry group} of $\mathscr P(\mathcal S_\P)$.  This group can be viewed as an analogue of the full symmetry group of the Apollonian-like packings defined by Baragar in \cite{baragar2018higher}.

\subsection{Möbius uniqueness and the Möbius spectrum}
Polytopal sphere packings and all their associated structures (such as the dual arrangement, Apollonian arrangement, and full symmetry group) are unique up to Möbius transformations. As noted in \cite{KontorovichNakamura}, this uniqueness follows from the Mostow Rigidity Theorem. Specifically, consider two polytopal sphere packings $\mathcal S_{\P_1},\mathcal S_{\P_2}$ in $\wrd$ for $d\ge2$, where $\P_1$ and $\P_2$ are edge-scribed realisations of an edge-scribable $(d+1)$-polytope. The polytopes $\mathcal Q_1=\P_1\cap\P_1^*$ and $\mathcal Q_2=\P_2\cap\P_2^*$ are inscribed in $\mathbb S^d$ and thus correspond to ideal hyperbolic polytopes in $\mathbb H^{d+1}$ with finite volume. Moreover, $\mathcal Q_1$ and $\mathcal Q_2$ are the fundamental domains of the hyperbolic reflection group generated by the reflections on their facets. Given that all dihedral angles of the facets are right angles due to polarity, these reflection groups are isomorphic. Therefore, by Mostow Rigidity Theorem, there exists a hyperbolic isometry of $\mathbb H^{d+1}$ mapping $\mathcal Q_1$ to $\mathcal Q_2$, which extends to Möbius transformation of $\wrd$ mapping $\mathcal S_{\P_1}$ to $\mathcal S_{\P_2}$.

\medskip
For every $d\ge3$, we define the \textit{Möbius spectrum} of an edge-scribable $d$-polytope $\P$ as the eigenvalues of $\mathrm{Gram}(\S_\P)$. Due to the Möbius uniqueness of edge-scribable polytopes with the invariance of the inversive product under Möbius transformations, the Möbius spectrum of $\P$ is independant of the specific packing $\S_\P$ and thus it is a well-defined invariant of edge-scribable polytopes. According to Steinitz's theorem \cite{steinitz1928isoperimetrische}, the graph of a $3$-polytope is a 3-connected simple planar graph, also known as a \textit{polyhedral graph}. Given that all $3$-polytopes are edge-scribable, the Möbius spectrum can be defined for any polyhedral graph. 

\begin{question}For any $d\ge3$, are there two combinatorially different edge-scribable $d$-polytopes with the same Möbius spectrum? In particular, are there two non isomorphic polyhedral graphs with the same Möbius spectrum?
\end{question}

The reader can find in the Appendix \ref{sec:appendix} the Möbius spectrum of each Platonic solid.

\section{A polytopal Descartes' theorem}\label{sec:DescartesRegPol}
In this section, we shall prove our generalization of Descartes' theorem for polytopal sphere packings induced by uniform polytopes, in terms of the polytope's geometry. We present first some notions and a lemma needed for the statement and the proof. Boyd's generalization of Descartes' theorem, presented in \cite{boyd}, asserts that for any arrangement $\mathcal A=(S_1,\ldots,S_{d+2})$ of spheres of $\wrd$ with full-rank Gramian, the bend vector $\mathbf b=(b_1,\ldots,b_{d+2})^\top$ satisfies
\begin{align}\label{eq:boydobs}
	\mathbf b^\top\mathrm{Gram}(\mathcal A)^{-1}\mathbf b=0.
\end{align}

 Let $X_N=\varepsilon_{d+1}+\varepsilon_{d+2}$ be the Lorentzian vector corresponding to the North Pole of $\mathbb S^d$ in the affine model of the space of spheres, where $\varepsilon_i$ denotes the $i$-th canonical vector of $\ed$ (see Figure \ref{fig:stereo}). Notice that $X_N$ lies on the \textit{light-cone} of $\ed$, i.e. 
\begin{align}\label{eq:boydxN}
	\langle  X_N, X_N\rangle=0.
\end{align}
We define the \textit{curvature} of any Lorentzian vector $X\in\mathbb R^{d+1,1}$ as
\begin{align}\label{eq:curvgen}
	\kappa(X):=-\langle X_N,X\rangle
\end{align}
Notice that, by Equation \ref{eq:invcoord}, $\kappa(X_S)=b(S)$ where $b(S)$ denotes the bend of $S$.
This notion of curvature, satisfies Boyd's equation for any basis of Lorentzian vectors.

\begin{lem}\label{lem:klemma}
	For any basis of Lorentzian vectors $\mathcal B=(X_1,\ldots,X_{d+2})$ of $\ed$ the curvature vector $\mathbf k=(\kappa_1,\ldots,\kappa_{d+2})^\top$
	satisfies
	\begin{align}
		\mathbf k^\top\mathrm{Gram}(\mathcal B)^{-1}\mathbf k=0
	\end{align}
\end{lem}
\begin{proof}Let $\mathbf B$ be the matrix of the Cartesian coordinates of $\mathcal B$. By combining Equations \eqref{eq:invprodmatrix} \eqref{eq:curvgen} and the definition of the Gramian, we have
	\begin{align*}
		\mathbf k^\top\mathrm{Gram}(\mathcal B)^{-1}\mathbf k=	(X_N^\top\mathbf Q_{d+2}\mathbf B^\top)((\mathbf B^\top)^{-1}\mathbf Q_{d+2}\mathbf B^{-1})(\mathbf B\mathbf Q_{d+2}X_N)=X_N^\top\mathbf Q_{d+2}X_N=\langle X_N,X_N\rangle=0
	\end{align*}
\end{proof}

%
Let $\mathcal S_\P$  be a polytopal sphere packing. For every face $f$ of $\P$, we define the \textit{polytopal curvature} of $\mathcal S_\P$ with respect to  $f$ as $\kappa_f:=\kappa(X_f)$ where $X_f:=\frac{1}{|V(f)|}\sum_{v\in V(f)}X_{v}$.
 By linearity, $\kappa_f$ corresponds to
\begin{align}
	\kappa_f=\frac{1}{|V(f)|}\sum_{v\in V(f)}b(S_v)
\end{align}
We now have all the ingredients to prove the main theorem, which we restate below.

\setcounter{thm}{1}
\begin{thm}\label{thm:poldescartes}
	Let $\mathcal S_\P$ be a polytopal sphere packing where $\P$ is a uniform $(d+1)$-polytope with $d\ge1$. The polytopal curvatures of $\mathcal S_\P$ with respect to the faces in any flag $(f_0,\ldots,f_d,f_{d+1}=\P)$ satisfy
	\begin{align}\label{eq:poldesth}
		(\kappa_{f_0}-\kappa_{f_1})^2+\ell_{f_2}^2(\kappa_{f_1}-\kappa_{f_2})^2+\sum_{i=2}^{d}\frac{1}{\ell_{f_{i+1}}^{-2}-\ell_{f_i}^{-2}}(\kappa_{f_i}-\kappa_{f_{i+1}})^2=\ell_\P^2\kappa_{\P}^2
	\end{align}
\end{thm}
\begin{proof} Let $\P$ be an edge-scribed uniform $(d+1)$-polytope with $d\ge1$ and let $\Phi=(f_0,f_1,\ldots,f_d,f_{d+1}=\P)$ be a flag of $\P$. 
	Let $\mathcal B=(Y_1, \ldots, Y_{d+2})$, where $Y_i:=X_{f_{i-1}}-X_{f_i}$ for every $i=1,\ldots,d+1$, and $Y_{d+2}:=X_\P$. Let us compute the Gramian of $\mathcal B$. Let $v=f_0$, $e=f_1$ and let $v'$ be the other vertex of $e$.
	By combining Equations \eqref{eq:lprod} and \eqref{eq:invcoord}, we have
	\begin{align*}
		&\langle X_v,X_v\rangle=1\\
		&\langle X_v,X_e\rangle=\langle X_v,\frac12(X_v+X_{v'})\rangle=\frac12(1-1)=0\\
		&\langle X_e,X_e\rangle=\langle \frac12(X_v+X_{v'}),\frac12(X_v+X_{v'})\rangle=\frac14(\langle X_v,X_v\rangle+2\langle X_v,X_{v'}\rangle+\langle X_{v'},X_{v'}\rangle)=\frac14(1-2+1)=0
	\end{align*}
	By definition of uniform polytope, $f_i$ is uniform for every $2\le j\le d+1$. Moreover, the intersection of $\mathbb S^{d}$ with the affine subspace spanned by $f_j$ induces an edge-scribed realization of $f_j$. Therefore, by Lemma \ref{lem:keylemma},  
    we have that for every $2\le j\le d+1$ and every $0\le i\le j$, $\langle X_{f_i},X_{f_j}\rangle=-\ell_{f_j}^{-2}$. Then,
	\begin{align*}
		\langle Y_1,Y_1\rangle&=\langle X_v-X_e, X_v-X_e\rangle=\langle X_v,X_v\rangle-2\langle X_v,X_e\rangle+ \langle X_e,X_e\rangle=1\\
		\langle Y_2,Y_2\rangle&=\langle X_e-X_{f_2}, X_e-X_{f_2}\rangle=\langle X_e,X_e\rangle-2\langle X_e,X_{f_2}\rangle+ \langle X_{f_2},X_{f_2}\rangle=\ell_{f_2}^{-2}
	\end{align*}
	For the rest of diagonal $i^\text{th}$-entries with $3\le i\le d+1$, we have
	\begin{align*}
		\langle Y_i,Y_i\rangle&=\langle X_{f_{i-1}}-X_{f_{i}}, X_{f_{i-1}}-X_{f_{i}}\rangle=\langle X_{f_{i-1}}, X_{f_{i-1}}\rangle-2\langle X_{f_{i-1}},X_{f_{i}}\rangle+\langle X_{f_{i-1}}, X_{f_{i}}\rangle=\ell_{f_{i+1}}^{-2}-\ell_{f_i}^{-2}
	\end{align*}
	and $\langle Y_{d+2},Y_{d+2}\rangle=\langle X_\P,X_\P\rangle=-\ell_{\P}^{-2}$.
	For the non-diagonal entries, we have
		\begin{align*}
		\langle Y_1,Y_2\rangle=\langle X_v-X_e, X_e-X_{f_2}\rangle&=\langle X_v,X_e\rangle-\langle X_v,X_{f_2}\rangle-\langle X_e,X_e\rangle+ \langle X_{f_2},X_{f_2}\rangle=0+\ell_{f_2}^{-2}-0-\ell_{f_2}^{-2}=0
	\end{align*}
	and for every  $3\le j\le d+2$ and every $1\le i< j$,
			\begin{align*}
		\langle Y_i,Y_j\rangle=\langle X_{f_{i-1}}-X_{f_{i}}, X_{f_{j-1}}-X_{f_{j}}\rangle&=\langle X_{f_{i-1}},X_{f_{i}}\rangle-\langle X_{f_{i-1}},X_{f_{j-1}}\rangle-\langle X_{f_{i}},X_{f_{j-1}}\rangle+ \langle X_{f_i},X_{f_j}\rangle\\
		&=-\ell_{f_j}^{-2}+\ell_{f_{j+1}}^{-2}+\ell_{f_{j}}^{-2}-\ell_{f_{j+1}}^{-2}=0
	\end{align*}
	
	Consequently, $\mathrm{Gram}(\mathcal B)=\mathrm{diag} (1,\ell_{f_2}^{-2},\ldots,\ell_{f_{i+1}}^{-2}-\ell_{f_i}^{-2},\ldots,-\ell_\P^{-2}) $.
	On the other hand, by linearity, the vector of curvatures of $\mathcal B$ is equal to $\mathbf k=(\kappa_{f_0}-\kappa_{f_1},\ldots,\kappa_{f_d}-\kappa_{\P},
	\kappa_\P)^\top$. Since $\mathcal B$ is a basis of $\ed$, we can apply Lemma \ref{lem:klemma}, which gives us 
	\begin{align*}
		\mathbf k^\top\mathrm{diag}(1,\ell_{f_2}^{2},\ldots,\frac{1}{\ell_{f_{i+1}}^{-2}-\ell_{f_i}^{-2}},\ldots,-\ell_\P^{2}) \mathbf k=0
	\end{align*}
	which is equivalent to \eqref{eq:poldesth}.		
\end{proof}

\section{Geometry and arithmetic of the Platonic crystallographic packings}\label{sec:platonic}
In this section, we explore the packings based on the simplest family of polytopes covered by Theorem \ref{thm:poldescartes}: the regular 3-polytopes, commonly known as the \textit{Platonic solids}.  By applying Theorem \ref{thm:poldescartes}, we present a generalization of Descartes' theorem in terms of the \textit{Schläfli symbol}, which provides the framework to construct analogues of integral Apollonian packings for the Platonic solids.  In Appendix \ref{sec:appendix}, we provide a catalogue summarising all the properties discussed in this paper for each Platonic solid.

\subsection{Fundamental symmetries and fundamental basis of the Platonic solids}
Let $\P$ be a Platonic solid and let $\Phi=(f_0=v,f_1=e,f_2=f,f_{3}=\P)$ be one of its flags. The tetrahedron $\Delta_\Phi$ whose vertices are the barycentres of every  $f_i\in\Phi$ is a fundamental domain of the symmetry group $\mathfrak S(\P)$. This group is the finite Coxeter group generated by the reflections $r_1,r_2,r_3 $, which we refer to as the \textit{fundamental symmetries} with respect to $\Phi$. Here, $r_i$ represents the reflection through the plane $R_i$ spanned by the facet of $\Delta_\Phi$ which is opposite to the barycentre of  $f_{i-1}$. 

\medskip
Each reflection $r_i$ fixes all the faces of $\Phi$ except $f_{i-1}$, swapping it with the unique $(i-1)$-face $f_{i-1}'$ such that the new flag $\Phi'=(\Phi\setminus \{f_{i-1}\})\cup\{f_{i-1}'\}$ is a valid flag of $\P$ (see Figure \ref{fig:fundamentals}). 

\medskip

We define the \textit{fundamental basis} of $\P$ with respect to $\Phi$, as the affine basis of vertices $(v_1,v_2,v_3,v_4)$ of $\P$, given by
\begin{align}
	v_1&=f_0,&&v_2=r_1(v_1),&&v_3=r_1r_2(v_2),&&v_4=r_1r_2r_3(v_3). 
\end{align}
We notice that for every $i=1,2,3,4$, $(v_1,\ldots,v_{i})$ is a fundamental basis of $f_{i-1}\in\Phi$. Additionally, $(v_1,v_2,v_3,v_4)$ describes a path of four distinct vertices  in the graph of $\P$, where  $(v_1,v_2) = e\in E(\P)$, $(v_2,v_3) =e'\in E(\P)\setminus e$, $(v_1,v_2,v_3)\in f\in F(\P)$ and $(v_2,v_3,v_4)\in f'\in F(\P)\setminus f$. Here $ F(\P)=F_2(\P)$ denotes the set of $2$-faces of $\P$.  
\begin{figure}[H]
	
	\begin{tikzpicture}[scale=1.1]
		\begin{scope}[xshift=-3cm]
			\node 	{\includegraphics[width=5cm]{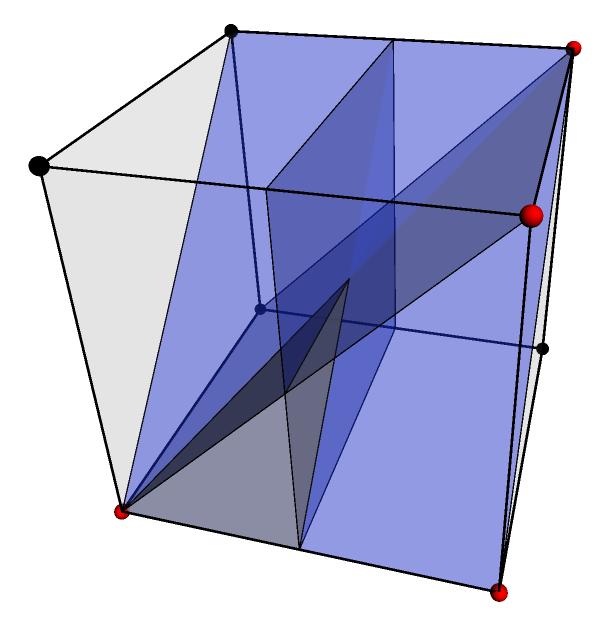}};
			\node at (-1.55,-1.75) {$v_1$};
			\node at (1.7,-2.3) {$v_2$};
			\node at (2.1,.7) {$v_3$};
			\node at (2.3,2.1) {$v_4$};
			\node at (0.4,1.4) {\color{white} $R_1$};
			\node at (1.3,1) {\color{white} $R_2$};
			\node at (1.2,-.5) {\color{white}  $R_3$};					
		\end{scope}		
		
		\begin{scope}[xshift=3cm,scale=.5]
			\draw (-2.414,2.414) circle (2.414cm);
			\draw (-.414,.414) circle (.414cm);
			\draw (.414,-.414) circle (.414cm);
			\draw (-.414,-.414) circle (.414cm);
			\draw[red] (2.414,2.414) circle (2.414cm);
			\draw[red] (2.414,-2.414) circle (2.414cm);
			\draw[red] (-2.414,-2.414) circle (2.414cm);
			\draw[red] (.414,.414) circle (.414cm);
			
			\draw[thick,blue] (0,-5) -- (0,5) node at (.7,5.2) {$R_1$} ;
			\draw[thick,blue] (-5,-5) -- (5,5) node at (5.6,4.8) {$R_2$} ;
			\draw[thick,blue] (0,-1) circle (1.414cm) node at (1.9,-1) {$R_3$};			
		\end{scope}		
	\end{tikzpicture}
	\caption{(Left) A cube with a fundamental domain of its symmetry group (in dark gray), a fundamental basis ($v_1,v_2,v_3,v_4$) (in red) and the fixed planes of the fundamental symmetries (in blue). 
		(Right) A cubic circle packing with the corresponding fundamental basis (in red) and the fundamental symmetries (in blue).
	} 
	\label{fig:fundamentals}
\end{figure}

\subsection{The full symmetry groups} We extend the definitions of fundamental symmetries and fundamental bases to any polytopal circle packing $\S_\P$ derived from a Platonic solid $\P$, as illustrated in Figure \ref{fig:fundamentals}.  We also define the \textit{fundamental bend vector} of $\S_\P$ with respect to $\Phi$ as the bend vector of the circles corresponding to the fundamental basis with respect to $\Phi$. The transitivity of $\mathfrak S(\P)=\langle r_1,r_2,r_3\rangle$ on the set of vertices and faces implies that the full symmetry group $\Gamma(\S_\P)=\langle s_f\rangle\rtimes \mathfrak S(\S_\P)=\langle r_1,r_2,r_3,s_f\rangle$ and $\{\S_\P/\mathfrak S(\S_\P)\}=\{S_v\}$, where $v$ and $f$ are any vertex and face of $\P$, respectively. Therefore, the Apollonian arrangement $\mathscr P(\S_\P)=\langle \S_{\P}^*\rangle\cdot \mathcal S_{\P}=\Gamma(\S_\P)\cdot \{S_v\}$. In Figure \ref{fig:apollonianclassic}, we illustrate the classic Apollonian strip packing obtained as the orbit space of the two group actions. 
\begin{figure}[H]
	\centering
	\begin{tabular}{cc}
		\includestandalone[align=c,scale=1.2]{tikzs/33apollonianmirrors}&\hspace{.5cm}	
		\includestandalone[align=c,scale=1.2]{tikzs/33standard}
	\end{tabular}
	\caption{The Apollonian strip packing  $\mathscr{P}_{\{3,3\}}$ obtained as crystallographic packing given by the action of the Apollonian group on a tetrahedral circle packing (left) and the full symmetry group $\Gamma_{\{3,3\}}$ on a single circle (right).
	}
	\label{fig:apollonianclassic}
\end{figure}

\subsection{The Platonic crystallographic packings}
The packing depicted in Figure \ref{fig:apollonianclassic} is commonly known as the \textit{Apollonian strip packing} and serves as a standard configuration for various purposes. We extend this notion for every Platonic solid $\P$ by stating that a packing $\S_\P$ is \textit{strip} for a given flag $(v,e,f,\P)$ if:
\begin{enumerate}
	\item The circle $S_v\in \S_\P$ is the half-space $\{y\le0\}$.
	\item The dual circle $S_f\in \S_\P^*$  is the half-space $\{x\le0\}$.
	\item The fundamental symmetry $r_2$ is the inversion through the unit sphere.
\end{enumerate}
Clearly, strip packings are unique up to Euclidean isometries. We denote by $\mathscr P_{\{p,q\}}$ the Apollonian arrangement (up to Möbius transformations)  and by $\Gamma_{\{p,q\}}$ its full symmetry group, for each Platonic solid ${\{p,q\}}$. We refer to these five packings as the \textit{Platonic crystallographic packings.} The Apollonian strip packing of Figure \ref{fig:apollonianclassic} corresponds to $\mathscr{P}_{\{3,3\}}$. In Figures \ref{fig:simpmirrors} and \ref{fig:simpmirrors2}, we illustrate the remaining Platonic crystallographic (strip) packings.
\begin{figure}[H]
	\centering
	\begin{tabular}{cc}
		\includegraphics[align=c,scale=1.2]{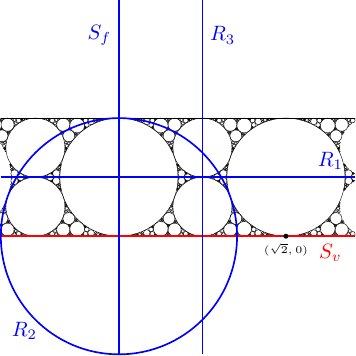}\hspace{.5cm}	&\includegraphics[align=c,scale=1.2]{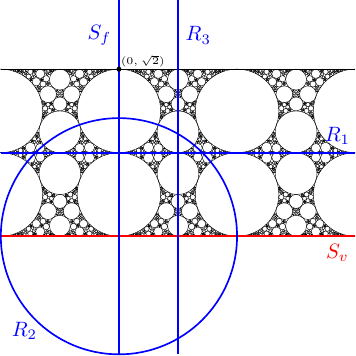}\\
	\end{tabular}
	\caption{The octahedral  $\mathscr{P}_{\{3,4\}}$ (left) and the cubic $\mathscr{P}_{\{4,3\}}$ (right) crystallographic packings.}
	\label{fig:simpmirrors}
\end{figure}
\begin{figure}[H]
	\centering
	\begin{tabular}{cc}
		\includegraphics[align=c,scale=1.2]{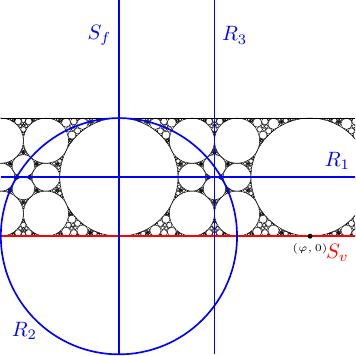}\hspace{.5cm}	&\includegraphics[align=c,scale=1.2]{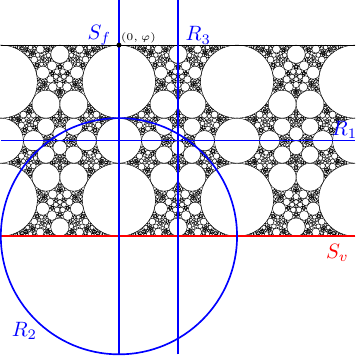}\\
	\end{tabular}
	\caption{The icosahedral  $\mathscr{P}_{\{3,5\}}$ (left) and the dodecahedral $\mathscr{P}_{\{5,3\}}$ (right) crystallographic packings.}
	\label{fig:simpmirrors2}
\end{figure}

\subsection{The Platonic Descartes Theorem} 
The classic Descartes' theorem is usually stated in terms of the \textit{Descartes quadratic form} \cite{Fuchs2013Counting} whose corresponding matrix is
\begin{align}
	\mathbf Q_D:=\left(
	\begin{array}{cccc}
		1 & -1 & -1 & -1 \\
		-1 & 1 & -1 & -1 \\
		-1 &-1 &  1 & -1 \\
		-1 &-1 & -1 &  1 \\
	\end{array}\right)
\end{align}
The matrix $\mathbf Q_D$ derives from Boyd's generalization \eqref{eq:boydobs}. For every Platonic solid $\{p,q\}$ , we define the \textit{Platonic quadratic form}
\begin{align}\label{eq:descartes3d}Q_{\{p,q\}}(b_1,b_2,b_3,b_4)=\mathbf{b}^T\mathbf{Q}_{\{p,q\}} \mathbf{b}
\end{align}
where $\mathbf Q_{\{p,q\}}$ is the bisymmetric matrix
\begin{align}	\mathbf Q_{\{p,q\}}=
	\left(
	\begin{array}{ccccc}
		1 & a & b & -1  \\
		a & d & c & b \\
		b & c & d & a \\
		-1 & b& a & 1\\
	\end{array}
	\right)&&\text{ with }&&\begin{array}{rl}
		a&=-1-\omega _p-\omega _q\\
		b&=-1+\omega _p-\omega _q\\
		c&=-1-\omega _p^2+\omega _q^2\\		
		d&=\left(1+\omega _p+\omega _q\right)^2 \\
	\end{array}
\end{align}
and $\omega_{n}:=1+2\cos\frac {2\pi}n$. Notice that $\omega_3=0$, $\omega_4=1$, $\omega_5=\varphi$ and $Q_{\{3,3\}}$  the Descartes quadratic form.

\begin{prop} \label{prop:platonicdesth}
	 Let $\mathcal S_{\P}$ be a polytopal circle packing where $\P$ is the Platonic solid $\{p,q\}$. For any fundamental bend vector $\mathbf b=(b_1,b_2,b_3,b_4)^\top$ of $\mathcal S_{\P}$ we have
	\begin{align}\label{eq:platonicdescartes}
		Q_{\{p,q\}}(b_1,b_2,b_3,b_4)=0
	\end{align}

\end{prop}

\begin{proof}
	Let $\Phi=(v,e,f,\P)$ be a flag of $\P$. Theorem \ref{thm:poldescartes} states that
	\begin{align}\label{eq:platonicflag}
		(\kappa_v-\kappa_e)^2+\ell_f^2(\kappa_e-\kappa_f)^2+\frac{1}{\ell_\P^{-2}-\ell_f^{-2}}(\kappa_f-\kappa_\P)^2=\ell_\P^2\kappa_\P^2
	\end{align}
	where $\ell_f$ and $\ell_\P$ are the canonical lengths of $f=\{p\}$ and $\P=\{p,q\}$. These values can be easily computed in terms of $\omega_p$ and $\omega_q$ by
	\begin{align}\label{eq:lvaluesfp}
		\ell_f=\sqrt{\frac{3-\omega_p}{1+\omega_p}}&&\ell_\P=\sqrt{\frac{2-\omega_p-\omega_q}{1+\omega_p}}
	\end{align}
	Let $(v_1,v_2,v_3,v_4)$ be a fundamental basis of $\P$ with respect to $\Phi=(v,e,f,\P)$, and let $\Phi'=(v',e',f',\P)$ be the flag of $\P$ where the corresponding fundamental basis $(v_1',v_2',v_3',v_4')$ satisfies that $v_2'=v_1$, $v_3'=v_2$, $v_4'=v_3$. Then, the edges $e=(v_1,v_2)$ and $e'=(v_1',v_2')$ share the vertex $v=v_1$. Similarly, the faces $f$ and $f'$ share the edge $e$.
	By solving \eqref{eq:platonicflag} in each of the polytopal curvatures $\kappa_e$, $\kappa_f$, after replacing the values in \eqref{eq:lvaluesfp}, the corresponding two solutions gives the polytopal curvatures of the faces in $\Phi$ and $\Phi'$. Then, by adding both solutions, we obtain the following relations.
	\begin{align}
		\label{eq:keke'}\kappa_e+\kappa_{e'}=&\frac{(1+\omega_p)\kappa_v+(3-\omega_p)\kappa_f}{2}\\
		\label{eq:kfkf'}\kappa_f+\kappa_{f'}=&2\frac{(1+\omega_q)\kappa_e+(2-\omega_p-\omega_q)\kappa_\P}{3-\omega_p}
	\end{align}
	Let  $\mathbf b=(b_1,b_2,b_3,b_4)^\top$ be the fundamental bend vector of $\mathcal S_{\P}$ with respect to $\Phi$. By combining the equations \eqref{eq:keke'} and \eqref{eq:kfkf'} with the definition of polytopal curvatures of the faces in $\Phi\cup\Phi'$, we obtain the following equations
	\begin{align}
		\label{eq:cycvertex}	\kappa_v=&b_1\\
		\label{eq:cycedge}	\kappa_e=&\frac{b_1+b_2}2 \\
		\label{eq:cycface}	\kappa_f=&\frac{b_1+b_2+b_3-\omega_p b_2}{3-\omega_p}\\
		\label{eq:cycpol}	\kappa_\P=&\frac{b_1+b_2+b_3+b_4-(\omega_p+\omega_q)(b_2+b_3)}{4-2(\omega_p+\omega_q)}
	\end{align}
	The above relations define a transition matrix $\mathbf T$ satisfying
	\begin{align}\label{eq:transition}
		\mathbf k=\mathbf T \mathbf b
	\end{align} 
	where $\mathbf k=(\kappa_v,\kappa_e,\kappa_f,\kappa_\P)^\top$. Let $\mathbf Q_\Phi$ be the matrix of the quadratic form induced by \eqref{eq:platonicflag} after combining with \eqref{eq:lvaluesfp}. Then, equation \eqref{eq:platonicflag} becomes
	\begin{align}
		\mathbf k^\top\mathbf Q_\Phi \mathbf k=0\Leftrightarrow \mathbf b^\top\mathbf T^\top\mathbf Q_\Phi\mathbf T\mathbf b=0.
	\end{align} 
	It can be checked by direct computations that $\mathbf Q_{\{p,q\}}=4(1+\omega_p)(1+\omega_q)\mathbf T^\top\mathbf Q_\Phi\mathbf T$.
\end{proof}

\subsection{Matrix representations of the full symmetry groups and integrality conditions} 
The classic Apollonian group introduced by Hirst in \cite{hirst1967apollonian}, and studied for the first time from the number theoretical point of view by Graham et al. in \cite{apoGI}, is defined as the discrete orthogonal subgroup $\langle \mathbf S_1,\mathbf S_2,\mathbf S_3,\mathbf S_4\rangle<\mathrm{SL}_4(\mathbb Z)$ with respect to the Descartes quadratic form, where 
\begin{equation*}
	\mathbf{S}_1=\left(\begin{matrix}
		-1& 2 & 2 & 2 \\
		0 	&1  &0  &0  \\
		0	&0  & 1 & 0 \\
		0	&0  & 0 & 1
	\end{matrix}\right),\, 
	\mathbf{S}_2=\left(\begin{matrix}
		1& 0 & 0 & 0 \\
		2&2 & -1 & 2 \\
		0&0 & 1 & 0 \\
		0	&0  &0  & 1
	\end{matrix},\right),\, 
	\mathbf{S}_3=\left(\begin{matrix}
		1& 0 & 0 & 0 \\
		0&1  &0  &0  \\
		2&2 & -1 & 2 \\
		0	&0  &0  & 1
	\end{matrix},\right),\, 
	\mathbf{S}_4=\left(\begin{matrix}
		1& 0 & 0 & 0 \\
		0 	&1  &0  &0  \\
		0	&0  & 1 & 0 \\
		2	&2  & 2 & -1
	\end{matrix}\right).
\end{equation*}

The generating matrices, called \textit{bend matrices} in \cite{chait2020taxonomy}, give the linear relations on the bends under the action of the Apollonian group. We give a similar linear representation of the full symmetry group of each Platonic solid.

\begin{prop}\label{prop:groups}For each Platonic solid $\{p,q\}$, the full symmetry group $\Gamma_{\{p,q\}}$ admits a linear representation as
	a discrete orthogonal subgroup $\langle \mathbf R_1,\mathbf R_2,\mathbf R_3, \mathbf S_f\rangle< \mathrm{SL_4}(\mathbb Z[\omega_p,\omega_q])$ with respect to $\mathbf{Q}_{\{p,q\}}$, where
	\begin{align}
		\mathbf R_1&=\left(
		\begin{array}{cccc}
			0 & 1 & 0 & 0 \\
			1 & 0 & 0 & 0 \\
			\omega _p & -\omega _p & 1 & 0 \\
			\omega _p \left(\omega _p+\omega _q\right)+\omega _q & -\omega _p \left(\omega _p+\omega _q\right)-\omega _q & 0 & 1
			\\
		\end{array}
		\right)\\
		\mathbf R_2&=\left(
		\begin{array}{cccc}
			1 & 0 & 0 & 0 \\
			\omega _p & -\omega _p & 1 & 0 \\
			\omega _p \left(-1+\omega _p\right) & 1-\omega _p^2 & \omega _p & 0 \\
			\omega _p^2 \left(-1+\omega _p+\omega _q\right) & -\omega _p \left(1+\omega _p\right) \left(-1+\omega _p+\omega _q\right) & \omega _p \left(-1+\omega _p+\omega _q\right) & 1 \\
		\end{array}
		\right)\\
		\mathbf R_3&=\left(
		\begin{array}{cccc}
			1 & 0 & 0 & 0 \\
			0 & 1 & 0 & 0 \\
			\omega _q & \omega _p & -(\omega _p+\omega _q) & 1 \\
			\omega _q \left(-1+\omega _p+\omega _q\right) & \omega _p \left(-1+\omega _p+\omega _q\right) & 1-\left(\omega
			_p+\omega _q\right)^2 & \omega _p+\omega _q \\
		\end{array}
		\right)\\
		\mathbf S_f&=\left(
		\begin{array}{cccc}
			1 & 0 & 0 & 0 \\
			0 & 1 & 0 & 0 \\
			0 & 0 & 1 & 0 \\
			2 & 2 \left(1-\omega _p+\omega _q\right) & 2 \left(1+\omega _p+\omega _q\right) & -1 \\
		\end{array}
		\right)
	\end{align}
\end{prop}

\begin{proof}
	Let $\S_\P$ a Platonic circle packing where $\P=\{p,q\}$. We consider the following elements of $\S_\P$ with respect to a given flag $\Phi=(v,e,f,\P)$:
	\begin{enumerate}
		\item The circles $S_1,S_2,S_3,S_4\in\mathcal S_\P$  corresponding to the fundamental basis,
		\item the fundamental bend vector $(b_1,b_2,b_3,b_4)^\top$,
		\item the polytopal curvatures $\kappa_v,\kappa_e,\kappa_f,\kappa_\P$,
		\item the fundamental symmetries $r_1, r_2, r_3$ and the dual inversion $s_f$.
	\end{enumerate}
	
	From equations \eqref{eq:keke'} and \eqref{eq:kfkf'}, we have that each $r_i$ corresponds to the matrix $\mathbf R_i'$ where
	\begin{align}
		\mathbf R_1'\left(\begin{matrix}
			\kappa_v \\
			\kappa_e \\
			\kappa_f \\
			\kappa_\P
		\end{matrix}\right)=&
		\left(\begin{matrix}
			-\kappa_v+2\kappa_e \\
			\kappa_e \\
			\kappa_f \\
			\kappa_\P
		\end{matrix}\right)\\
		\mathbf R_2'\left(\begin{matrix}
			\kappa_v \\
			\kappa_e \\
			\kappa_f \\
			\kappa_\P
		\end{matrix}\right)=&
		\left(\begin{matrix}
			\kappa_v\\
			\frac{1+\omega_p}{2}\kappa_v-\kappa_e+\frac{3-\omega_p}{2}\kappa_f \\
			\kappa_f \\
			\kappa_\P
		\end{matrix}\right)\\
		\mathbf R_3'\left(\begin{matrix}
			\kappa_v \\
			\kappa_e \\
			\kappa_f \\
			\kappa_\P
		\end{matrix}\right)=&
		\left(\begin{matrix}
			\kappa_v\\
			\kappa_e\\
			\frac{2(1+\omega_q)}{3-\omega_p}\kappa_e-\kappa_f+\frac{2(2-\omega_p-\omega_q)}{3-\omega_p}\kappa_\P \\
			\kappa_\P
		\end{matrix}\right)
	\end{align}
	By conjugating $\mathbf R_1',\mathbf R_2',\mathbf R_3'$ with the transition matrix $\mathbf T$ described in \eqref{eq:transition}, we obtain the matrices  $\mathbf R_1,\mathbf R_2,\mathbf R_3$. Then, by resolving \eqref{eq:platonicdescartes} on $b_4$, we obtain
	\begin{align}\label{eq:curv3pm}
		b_4,b_4'=b_1+(1-\omega_p+\omega_q)b_2+(1+\omega_p+\omega_q)b_3\pm2\sqrt{(1+\omega_q)(b_1b_2+b_1b_3+b_2b_3-\omega_pb_2^2)}
	\end{align}
	where $b_4'$ is the bend of the circle $s_f(S_4)$. Since $s_f$ fixes $S_1,S_2,S_3$, then it acts on the bend vectors  of $\mathscr P(\S_\P)$ as
	\begin{align}
		\mathbf S_f\left(\begin{matrix}
			b_1 \\
			b_2 \\
			b_3 \\
			b_4
		\end{matrix}\right)=&
		\left(\begin{matrix}
			b_1\\
			b_2 \\
			b_3 \\
			2b_1+2(1-\omega_p+\omega_q)b_2+2(1+\omega_p+\omega_q)b_3-b_4
		\end{matrix}\right)
	\end{align}
	where the last row is obtained by expressing $b_4'$ after adding both solutions in \eqref{eq:curv3pm}. Since the fundamental symmetries and $s_f$ generate the full symmetry group $\Gamma_{\{p,q\}}$, the matrix group $\langle \mathbf R_1, \mathbf R_2,\mathbf R_3,\mathbf S_f\rangle$ is a linear representation of $\Gamma_{\{p,q\}}$.  The four matrices are in $\mathrm{SL}_4(\mathbb Z[\omega_p,\omega_q])$ and are orthogonal with respect to $\mathbf Q_{\{p,q\}}$.
\end{proof}

The set of bends of a Platonic circle packing can be obtained by the action of $\langle \mathbf R_1,\mathbf R_2,\mathbf R_3\rangle$ on a single fundamental bend vector. It can be equally obtained by applying the following linear relations, which can be deduced from Equations \eqref{eq:cycface},  \eqref{eq:cycpol}, \eqref{eq:curv3pm}.
\begin{cor}[Face relation]\label{cor:facerel} Let  $\S_\P$ be a Platonic circle packing with $\P=\{p,q\}$ and let $b_1,b_2,b_3,b_4$ be the bends of four circles corresponding to four consecutive vertices in a face of $\P$. Then,
	\begin{align}
		\label{eq:faceki}b_{4}=b_1-\omega_p(b_2-b_3).
	\end{align} 
\end{cor}	

\begin{cor}[Consecutive fundamental bases]\label{cor:fundarel} Let  $\S_\P$ be a Platonic circle packing with $\P=\{p,q\}$. If $(b_1,b_2,b_3,b_4)^\top$ and $(b_2,b_3,b_4,b_5)^\top$ are two fundamental bend vectors of   $\S_\P$ then
	\begin{align}
	\label{eq:polki}	
	b_5= b_1-(\omega_p+\omega_q)(b_2-b_4).
\end{align}
\end{cor}	

\begin{cor}[Dual inversion]\label{cor:polrel} Let  $\S_\P$  be a Platonic circle packings with $\P=\{p,q\}$. If $(b_1,b_2,b_3,b_4)^\top$ and $(b_1,b_2,b_3,b_4')^\top$ are two fundamental bend vectors of $\S_\P$ and $\S_\P'$, respectively, where $\mathcal S_{\P}'$ is obtained from $\mathcal S_{\P}$ by the inversion through the dual circle which is orthogonal to the circles corresponding to $b_1,b_2,b_3$, then
\begin{align}
	\label{eq:kappa4}	b_4'=2b_1+2(1-\omega_p+\omega_q)b_2+2(1+\omega_p+\omega_q)b_3-b_4.
\end{align}
\end{cor}

Similarly, the set of bends of a Platonic crystallographic packing can be obtained by the action of $\langle \mathbf R_1,\mathbf R_2,\mathbf R_3,\mathbf S_f\rangle$ on a single fundamental bend vector. In this case, this set is fully determinated by three bends. We denote by $\mathscr P_{\{p,q\}}(b_1,b_2,b_3)$ the Platonic crystallographic packing where $b_1,b_2,b_3$ are three consecutive entries of a fundamental bend vector.
We say that $\mathscr P_{\{p,q\}}(b_1,b_2,b_3)$  is $\mathbb Z[\omega_p,\omega_q]$-\textit{integral} (or simply \textit{integral} if $\mathbb Z[\omega_p,\omega_q]=\mathbb Z$) if its set of bends is contained in $\mathbb Z[\omega_p,\omega_q]$. By combining equation \eqref{eq:curv3pm} with the action of the matrices in Proposition \ref{prop:groups}, we obtain the following.

\begin{cor}[Integrality condition]\label{cor:integrality} Let $b_1,b_2,b_3$ be three consecutive entries of a fundamental bend vector of the Platonic crystallographic packing $\mathscr P_{\{p,q\}}(b_1,b_2,b_3)$. If $b_1,b_2,b_3,\sqrt{\Delta_{\{p,q\}}}$ are in $\mathbb Z[\omega_p,\omega_q]$ where
	\begin{align}\label{eq:integralitycond}\Delta_{\{p,q\}}:=(1+\omega_q)(b_1b_2+b_2b_3+b_3b_1-\omega_pb_2^2)
	\end{align}
	then $\mathscr P_{\{p,q\}}(b_1,b_2,b_3)$ is $\mathbb Z[\omega_p,\omega_q]$-integral.
\end{cor}

\begin{cor}
	For every Platonic solid $\{p,q\}$, $\mathscr P_{\{p,q\}}(0,0,1)$ is $\mathbb Z[\omega_p,\omega_q]$-integral. Moreover, the set of bends of $\mathscr P_{\{3,3\}}(0,0,1)$, $\mathscr P_{\{4,3\}}(0,0,1)$ and $\mathscr P_{\{5,3\}}(0,0,1)$ contains the sequence of perfect squares.
\end{cor}

\begin{proof}\label{cor:squares} 
	The integrality follows from Corollary \ref{cor:integrality}. From \eqref{eq:curv3pm}, if $b_1=b_2=0$ and $b_3=1$, then $b_4=1+\omega_p+\omega_q$. Let $\mathbf b_0:=(0,0,1,1+\omega_p+\omega_q)^\top$ and let, for every $n\ge0$, $\mathbf M_n:=\mathbf R_2 (\mathbf R_3\mathbf S_f)^n\mathbf R_2$ and $\mathbf b_n=(b_1^{(n)},b_2^{(n)},b_3^{(n)},b_4^{(n)})^\top:=(\mathbf M_n \mathbf b_0)^\top$.  It can be proved by induction on $n$ that $b_2^{(n)}=n^2(1+\omega_q)$. Therefore, for $q=3$, $b_2^{(n)}=n^2$.
\end{proof}
The integral packings $\mathscr P_{\{p,q\}}(0,0,1)$ are illustred in Figures \ref{fig:apotetra}, \ref{fig:apoct}, \ref{fig:apocube}, \ref{fig:apoico} and \ref{fig:apodode}. We can use the previous two corollaries to parametrize the triples satisfying the integrality condition. To do so, we must solve the Diophantine equation derived from \eqref{eq:integralitycond}. 
\begin{align}\label{eq:diophantine}
-x_0^2+	(1+\omega_q)(x_1x_2+x_2x_3+x_3x_1-\omega_px^2_2)=0
\end{align}
for $x_0,x_1,x_2,x_3\in\mathbb Z[\omega_p,\omega_q]$. Since $\mathbb Z[\omega_p,\omega_q]$ has class number 1 and the above equation is homogeneous of degree 2, we can parameterize the integer solutions  by standard methods.
\begin{cor}\label{cor:parametrization}
	For every Platonic solid $\{p,q\}$, the triples in $\mathbb Z[\omega_p,\omega_q]^3$ satisfying \eqref{eq:integralitycond} are given by $(b_1,b_2,b_3)=\frac kg(h_1, h_2, h_3)$ where
	\begin{align}\label{eq:hs}
	h_1=(1+\omega_q)(t_2+t_3)t_2,&&h_2=(1+\omega_q)(t_2+t_3)t_3,&&h_3=t_1^2-(1+\omega_q)(t_2-\omega_p t_3)t_3,
	\end{align}
	 $t_1,t_2,t_3\in\mathbb Z[\omega_p,\omega_q]$ are coprimes and $k\in\mathbb Z[\omega_p,\omega_q]$ coprime with $g=\mathrm{gcd}(h_0,h_1,h_2,h_3)$ with $h_0=(1+\omega_q)(t_2+t_3)t_1$.
\end{cor}
\begin{proof}
	Solving \eqref{eq:diophantine} in $\mathbb Z[\omega_p,\omega_q]^4$ is equivalent to find all points in $\mathbb Q[\omega_p,\omega_q]^4$ contained in the projective hypersurface defined by
	$$\mathcal H=\{Q(x_0,x_1,x_2,x_3)=-x_0^2+	(1+\omega_q)(x_1x_2+x_2x_3+x_3x_1-\omega_px^2_2)=0\}$$
	The point $P=(0,0,0,1)\in\mathcal H$ and we can find the other rational points by intersecting $\mathcal H$ with lines passing through $P$ with rational parameters. We consider the affine quadric surface $\{R(x_0,x_1,x_2)=Q(x_0,x_1,x_2,1)=0\}$
	which has the rational point $O=(0,0,0)$. The parametric equation of a line passing through $O$ is given by taking $x_1=t_2 x_0$ and $x_2=t_3 x_0$. Then we have that
	$$\frac{R(x_0,t_2 x_0,t_3 x_0)}{x_0}=-x_0+(1+\omega_q)(t_2t_3x_0+t_2+t_3-\omega_pt_3^2x_0)$$
	By solving the right-hand side of the above equation for $x_0$, and then using the equalities $x_1=t_2 x_0$ and $x_2=t_3 x_0$, we obtain $x_0=f_0(t_2,t_3)/f_3(t_2,t_3)$, $x_1=f_1(t_2,t_3)/f_3(t_2,t_3)$ and $x_2=f_2(t_2,t_3)/f_3(t_2,t_3)$, where
		\begin{align*}
		&	f_0(t_2,t_3)=(1+\omega_q)(t_2+t_3)\\
		&	f_1(t_2,t_3)=(1+\omega_q)(t_2+t_3)t_2\\
		&	f_2(t_2,t_3)=(1+\omega_q)(t_2+t_3)t_3\\
		&	f_3(t_2,t_3)=1-(1+\omega_q)(t_2t_3-\omega_p t_3^2)
	\end{align*}
	This gives us a parameterization of $\mathcal H$ with quadratic polynomials with  coefficients in $\mathbb Z[\omega_p,\omega_q]$ by taking
				\begin{align*}
		&	x_i=t_1^2f_i\left(\frac {t_2}{t_1},\frac{t_3}{t_1}\right)=h_i
	\end{align*}
	for every $i=0,1,2,3$.
A point in $\mathcal{H}$ belongs to $\mathbb{Q}[\omega_p,\omega_q]^4$ if and only if it can be generated from values $t_1, t_2, t_3 \in \mathbb{Q}[\omega_p, \omega_q]$. Since $h_0, h_1, h_2, h_3$ are homogeneous, the point remains unchanged when all the $t_i$ are multiplied by the same element in $\mathbb{Q}[\omega_p, \omega_q]$. Therefore, we can assume that $t_1, t_2, t_3$ are coprime integers in $\mathbb{Z}[\omega_p, \omega_q]$. As a result, the integer solutions to the Diophantine equation \eqref{eq:diophantine} are precisely the sequences $\frac{k}{g}(h_0, h_1, h_2, h_3)$, where $k \in \mathbb{Z}[\omega_p, \omega_q]$ and $g = \gcd(h_0, h_1, h_2, h_3)$.
\end{proof}

We will use the parameterization of Corollary \ref{cor:parametrization} to find generating triples for primitive and integral Platonic crystallographic packings, i.e. where the gcd of all the bends is $1$ (see
Figures \ref{fig:apotetra}, \ref{fig:apoct}, \ref{fig:apocube}, \ref{fig:apoico} and \ref{fig:apodode}). This is achieved by setting $k=1$ in the given parameterization. Note that different triples can produce the same packing, so this method alone cannot be used to count distinct primitive packings. For methods of enumeration, refer to \cite{apoNumber}.

\printbibliography[
title={References}
] 

\appendix

\section{The Platonic crystallographic packings}\label{sec:appendix}
\subsection{Tetrahedron $\{3,3\}$} This is the classical case which has been extensively studied \cite{apoGI}.
 In Figure \ref{fig:canonicaltetra}, we show three tetrahedral circle packings obtained by the arrangement projections of canonical tetrahedra. The canonical length is $\sqrt2$.
 \vspace{-.5cm}

\begin{figure}[H]
\begin{tikzpicture}
	\begin{scope}[yshift=4.2cm]
		\node at (-5.75,0) {\includegraphics[align=c,width=4.4cm]{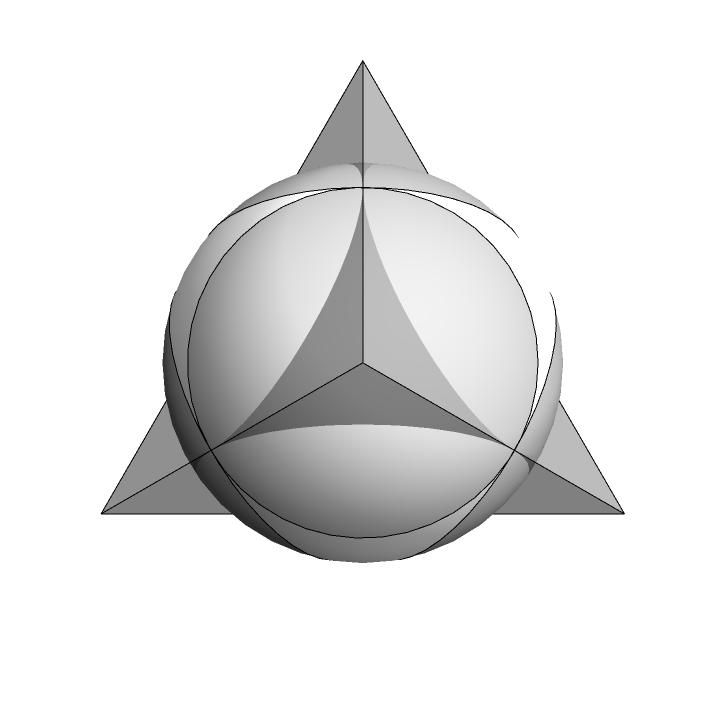} };
		\node at (0,0) {\includegraphics[align=c,width=4.4cm]{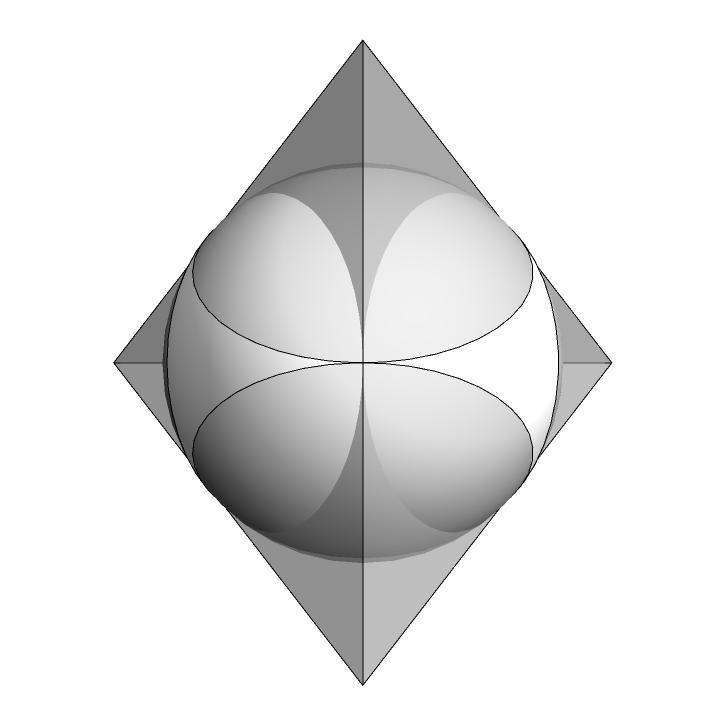}  };
		\node at (5.75,0) {\includegraphics[align=c,width=4.4cm]{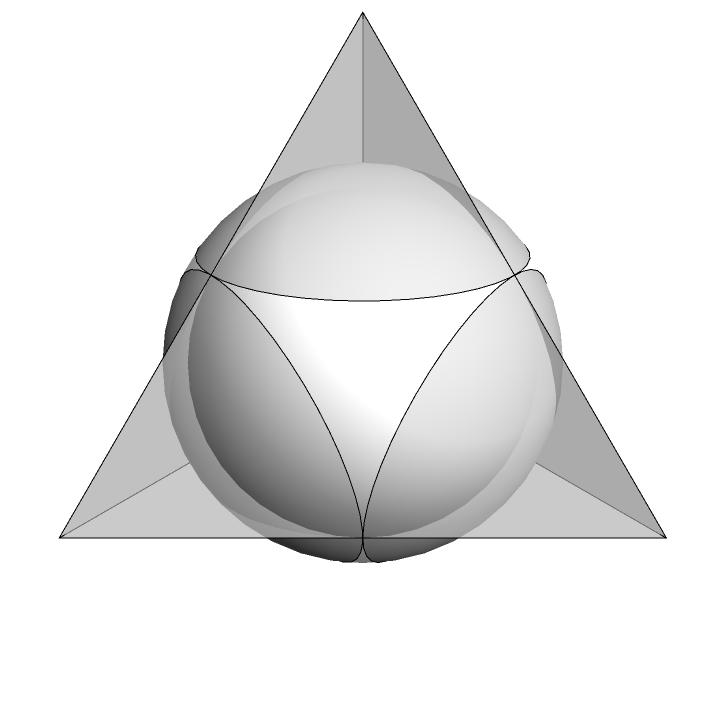}  };
	\end{scope}
	\begin{scope}
		\node at (-5.75,0) {\begin{tikzpicture}[scale=1.5] 

\fill[darkgray,opacity=\opac] (-1.2,-1.2) rectangle (1.2,1.2);
\draw[very thin,draw=black, fill=white] (0,0) circle (1cm);

\draw[very thin, draw=black, fill=darkgray, fill opacity=\opac] (0,0.5359) circle (0.4641cm);
\draw[very thin, draw=black, fill=darkgray, fill opacity=\opac] (0.4641,-0.2679) circle (0.4641cm);
\draw[very thin, draw=black, fill=darkgray, fill opacity=\opac] (-0.4641,-0.2679) circle (0.4641cm);

\end{tikzpicture}};
		\node at (0,0) {\begin{tikzpicture}[scale=1] 

\newcommand\wid{2.3}
\draw[very thin] (-\wid,1) -- (\wid,1);
\fill[darkgray,opacity=\opac] (-\wid,1) rectangle (\wid,1.618);
\draw[very thin] (-\wid,-1) -- (\wid,-1);
\fill[darkgray,opacity=\opac] (-\wid,-1) rectangle (\wid,-1.618);
\draw[very thin, draw=black, fill=darkgray, fill opacity=\opac] (1,0) circle (1.00cm);
\draw[very thin, draw=black, fill=darkgray, fill opacity=\opac] (-1,0) circle (1.00cm);

\end{tikzpicture} };
		\node at (5.75,0) {\begin{tikzpicture}[scale=1] 
 \clip (-1.2,-1.2) rectangle (1.2,1.2);
\begin{scope}[scale=2.1547]

\draw[very thin, draw=black, fill=darkgray, fill opacity=\opac] (0.4641,-0.2679) circle (0.4641cm);
\draw[very thin, draw=black, fill=darkgray, fill opacity=\opac] (-0.4641,-0.2679) circle (0.4641cm);
\draw[very thin, draw=black, fill=darkgray, fill opacity=\opac] (0,0.5359) circle (0.4641cm);
\draw[very thin, draw=black, fill=darkgray, fill opacity=\opac] (0,0) circle (0.07180cm);

\end{scope}

\end{tikzpicture}};
	\end{scope}
\end{tikzpicture}
\vspace{-.5cm}
	\caption{
	(Top figures, from left to right) A canonical tetrahedron with its spherical illuminated regions viewed from above. The images show three orientations: a vertex, an edge, and a face centred at the North Pole. (Bottom figures) The corresponding tetrahedral circle packings given by the arrangement projection.
	}
\label{fig:canonicaltetra}	
\end{figure}
 The Gramian is equal to the Descartes quadratic form and the Möbius spectrum of the tetrahedron is $(-2_{(1)},2_{(3)})$.
The linear representation of the full symmetry group $\Gamma_{\{3,3\}}<\mathrm{SL}_4(\mathbb Z)$ is generated by the following matrices
\begin{align*}
	&\mathbf{R}_1=\left(\begin{array}{cccc}
		0& 1 & 0 & 0 \\
		1 &0  &0  &0  \\
		0	&0  & 1 & 0 \\
		0	&0  &0  & 1
	\end{array}\right) &  
	\mathbf{R}_2=\left(\begin{array}{cccc}
		1& 0 & 0 & 0 \\
		0 &0  &1  &0  \\
		0	&1 & 0 & 0 \\
		0	&0  &0  & 1
	\end{array}\right) \\
	&	\mathbf{R}_3=\left(\begin{array}{cccc}
		1& 0 & 0 & 0 \\
		0 &1  &0  &0  \\
		0	&0 & 0 & 1 \\
		0	&0  &1 & 0
	\end{array}\right)&
	\mathbf{S}_f=\left(\begin{array}{cccc}
		1& 0 & 0 & 0 \\
		0 	&1  &0  &0  \\
		0	&0  & 1 & 0 \\
		2	&2  & 2 & -1
	\end{array}\right)
\end{align*}
The action by conjugation of $\langle\mathbf R_1 ,\mathbf R_2 ,\mathbf R_3\rangle $ on $\mathbf S_f$ gives the four bend matrices which generate the Apollonian group. The tetrahedral quadratic form 
coincides with the Descartes quadratic form
\begin{align}\mathbf Q_{\{3,3\}}=
	\left(
	\begin{array}{cccc}
		1 & -1 & -1 & -1 \\
		-1 & 1 & -1 & -1 \\
		-1 & -1 & 1 & -1 \\
		-1 & -1 & -1 & 1 \\
	\end{array}
	\right)
\end{align}
which implies the following relations on any fundamental bend vector $\mathbf b=(b_1,b_2,b_3,b_4)^\top$ of a tetrahedral circle packing $\S_\P$: 
\begin{enumerate}[-]
	\item (Descartes' theorem) for any fundamental bend vector $\mathbf b=(b_1,b_2,b_3,b_4)^\top$, 
	\begin{align}
		(b_1+b_2+b_3+b_4)^2= 2(b_1^2+b_2^2+b_3^2+b_4^2)
	\end{align}
		\item  (Dual inversion) if $\mathbf b'=(b_1,b_2,b_3,b_4')^\top$ is the fundamental bend vector of $\S'_\P$ obtained from $\S_\P$ after applying the inversion through the dual sphere orthogonal to $(b_1,b_2,b_3)$, then
%
	\begin{align}
		b_4+b_4'=2(b_1+b_2+b_3)
	\end{align}
	
	\item  (Integrality condition) if $b_1,b_2,b_3,\Delta_{\{3,3\}}\in\mathbb Z$ where
	\begin{align}\label{eq:inttetra}
		\Delta_{\{3,3\}}=b_1b_2+b_2b_3+b_3b_1
	\end{align}
	then the tetrahedral crystallographic packing $\mathscr P_{\{3,3\}}(b_1,b_2,b_3)$ is integral. The primitive triples satisfying the previous condition are paremeterized by $(b_1,b_2,b_3)=\frac 1g(h_1, h_2, h_3)$ where
	\begin{align}
		h_1=t_2(t_2+t_3),&&h_2=t_3(t_2+t_3),&&h_3=t_1^2-t_2t_3
	\end{align}
	 $t_1,t_2,t_3$ are three coprime integers and $g=\mathrm{gcd}((t_1(t_2+t_3),h_1,h_2,h_3)$.
	The two integral packings of Figure  \ref{fig:apotetra} are generated by taking $(t_1,t_2,t_3)=(1,0,0),(1,2,-4)$.
\end{enumerate}
\vspace{-.5cm}
\begin{figure}[H]
	\centering
\begin{tikzpicture}
	\begin{scope}
		\node at (0,0) {	\includegraphics[align=c,trim=0 70 0 70,clip,align=c,width=0.48\textwidth]{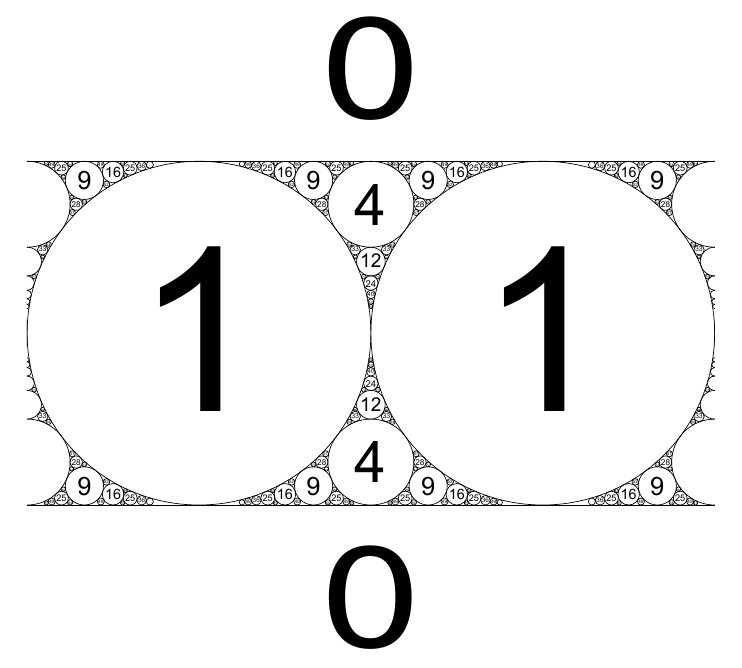}	}; 
	\end{scope}
	\begin{scope}[xshift=8cm]
		\clip (0,0)  circle (.2\textwidth) ;
		\node[anchor=center] at (-.5,-.5) {	\includegraphics[trim=0 20 40 60,clip,align=c,width=0.5\textwidth]{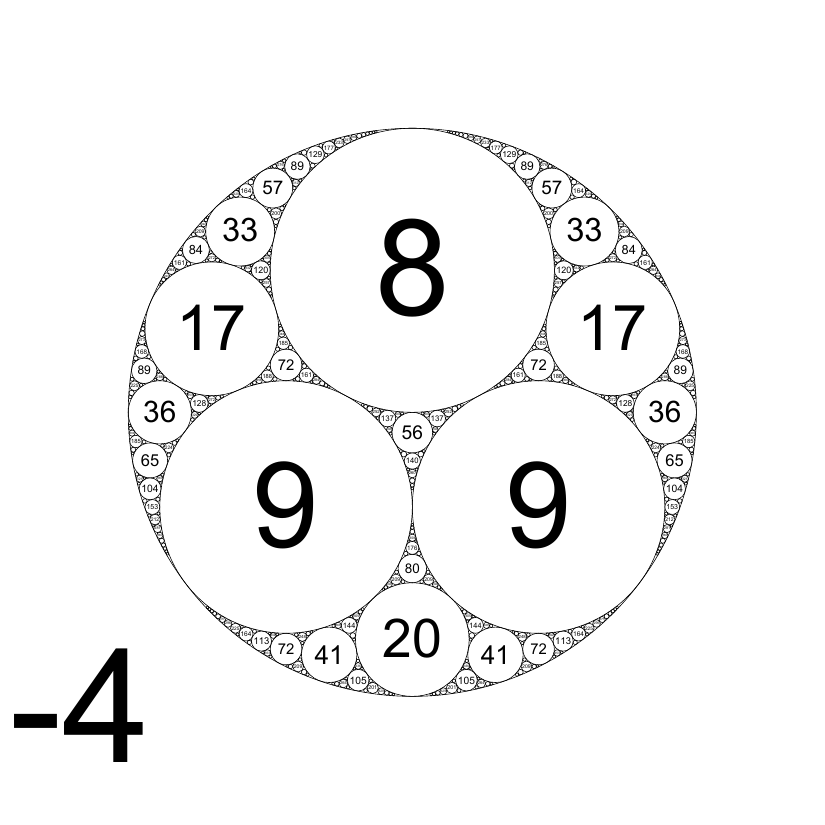}	}; 
	\end{scope}
\end{tikzpicture}
	\caption{The integral tetrahedral crystallographic packings $\mathscr P_{\{3,3\}}(0,0,1)$ (left) and $\mathscr P_{\{3,3\}}(-4,8,9)$ (right).}
	\label{fig:apotetra}
\end{figure}

\subsection{Octahedron $\{3,4\}$} Octahedral circle packings appear in the works of Boyd in \cite{boyd}, Guettler and Mallows \cite{guettler}, Zhang \cite{Zhang+2018+71+110} and Lautzenheiser \cite{lautzenheiser2024residual}. In Figure \ref{fig:canonicaloct}, we show three octahedral circle packings obtained by the arrangement projections of three octahedra. The canonical length is $1$.

\begin{figure}[H]
	\begin{tikzpicture}
		\begin{scope}[yshift=3.7cm]
			\node at (-5.75,0) {\includegraphics[trim=0 0 0 40,clip,align=c,width=4.4cm]{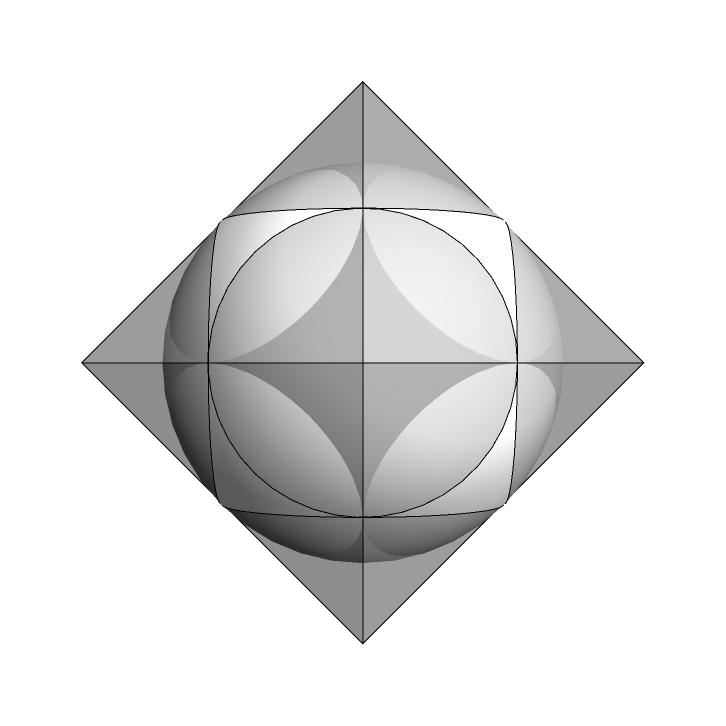} };
			\node at (0,0) {\includegraphics[trim=0 0 0 40,clip,align=c,width=4.4cm]{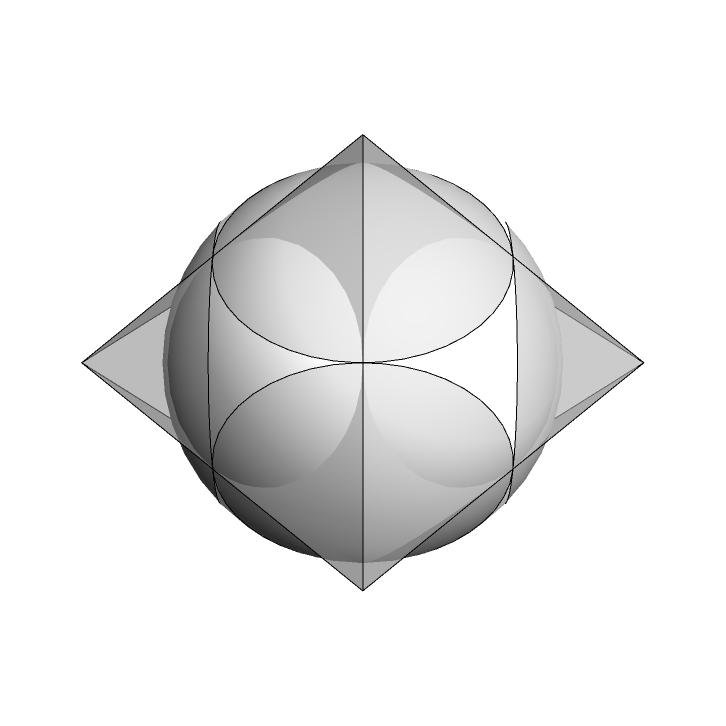}  };
			\node at (5.75,0) {\includegraphics[trim=0 0 0 40,clip,align=c,width=4.4cm]{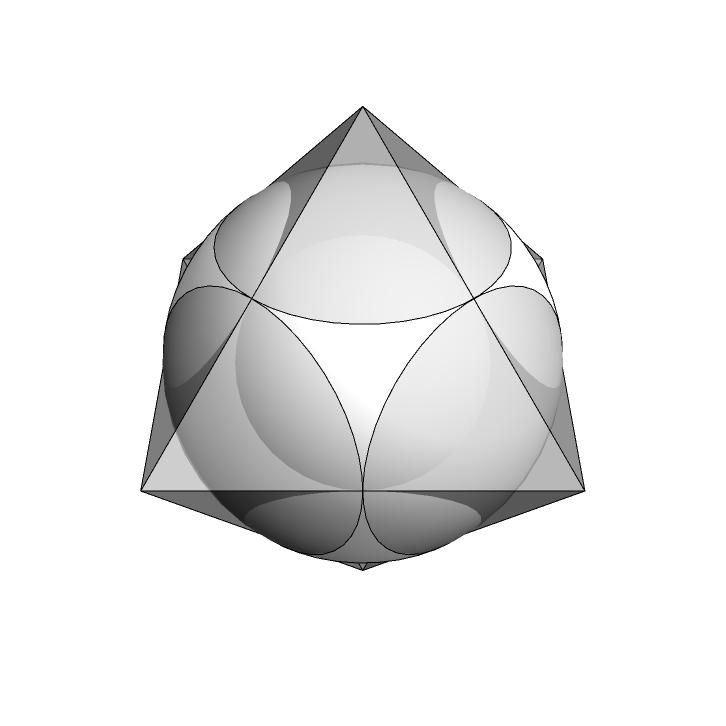}  };
		\end{scope}
		\begin{scope}
			\node at (-5.75,0) {\begin{tikzpicture}[scale=1.5] 

\fill[darkgray,opacity=\opac] (-1.2,-1.2) rectangle (1.2,1.2);
\draw[very thin, draw=black, fill=white] (0,0) circle (-1.00cm);

\draw[very thin, draw=black, fill=darkgray, fill opacity=\opac] (-0.5858,0) circle (0.4142cm);
\draw[very thin, draw=black, fill=darkgray, fill opacity=\opac] (0,-0.5858) circle (0.4142cm);
\draw[very thin, draw=black, fill=darkgray, fill opacity=\opac] (0,0) circle (0.1716cm);

\draw[very thin, draw=black, fill=darkgray, fill opacity=\opac] (0,0.5858) circle (0.4142cm);
\draw[very thin, draw=black, fill=darkgray, fill opacity=\opac] (0.5858,0) circle (0.4142cm);

\end{tikzpicture}};
			\node at (0,0) {\begin{tikzpicture}[scale=1] 

\newcommand\wid{2.7}
\draw[very thin] (-\wid,1) -- (\wid,1);
\fill[darkgray,opacity=\opac] (-\wid,1) rectangle (\wid,1.618);
\draw[very thin] (-\wid,-1) -- (\wid,-1);
\fill[darkgray,opacity=\opac] (-\wid,-1) rectangle (\wid,-1.618);
\draw[very thin, draw=black, fill=darkgray, fill opacity=\opac] (1.414,0) circle (1.000cm);
\draw[very thin, draw=black, fill=darkgray, fill opacity=\opac] (-1.414,0) circle (1.000cm);
\draw[very thin, draw=black, fill=darkgray, fill opacity=\opac] (0,0.5000) circle (0.5000cm);
\draw[very thin, draw=black, fill=darkgray, fill opacity=\opac] (0,-0.5000) circle (0.5000cm);

\end{tikzpicture} };
			\node at (5.75,0) {\begin{tikzpicture}[scale=1] 
 \clip (-1.2,-1.2) rectangle (1.2,1.2);
\draw[very thin, draw=black, fill=darkgray, fill opacity=\opac] (0,1.155) circle (1.000cm);
\draw[very thin, draw=black, fill=darkgray, fill opacity=\opac] (-1.000,-0.5774) circle (1.000cm);
\draw[very thin, draw=black, fill=darkgray, fill opacity=\opac] (1.000,-0.5774) circle (1.000cm);
\draw[very thin, draw=black, fill=darkgray, fill opacity=\opac] (-0.1010,0.05832) circle (0.1010cm);
\draw[very thin, draw=black, fill=darkgray, fill opacity=\opac] (0.1010,0.05832) circle (0.1010cm);
\draw[very thin, draw=black, fill=darkgray, fill opacity=\opac] (0,-0.1166) circle (0.1010cm);

\end{tikzpicture}};
		\end{scope}
	\end{tikzpicture}
	\vspace{-.5cm}
	\caption{
			(Top figures, from left to right) A canonical octahedron with its spherical illuminated regions viewed from above. The images show three orientations: a vertex, an edge, and a face centred at the North Pole. (Bottom figures) The corresponding octahedral circle packings given by the arrangement projection.}
	\label{fig:canonicaloct}	
\end{figure}

 The Gramian can be computed from Table \ref{tab:distoct} and the Möbius spectrum is $(-4_{(1)},0_{(2)},6_{(3)})$.
\begin{table}[H]
	\begin{tabular}{c|cccccc}
		\rule[-1ex]{0pt}{2.5ex} Graph distance & 0 & 1 & 2  \\
		\hline
		\rule[-1ex]{0pt}{2.5ex} Inversive product & 1 & $-1$ & $-3$ \\
	\end{tabular}
	\caption{The inversive product compared to the graph-distance of octahedral circle packings.}
	\label{tab:distoct}
\end{table}

The linear representation of the full symmetry group $\Gamma_{\{3,4\}}<\mathrm{SL}_4(\mathbb Z)$ is generated by the matrices
\begin{align*}
	\mathbf{R}_1=\left(\begin{array}{cccc}
		0& 1 & 0 & 0 \\
		1 &0  &0  &0  \\
		0	&0  & 1 & 0 \\
		1	&-1  &0  & 1
	\end{array}\right) &&
	\mathbf{R}_2=\left(\begin{array}{cccc}
		1& 0 & 0 & 0 \\
		0 &0  &1  &0  \\
		0	&1 & 0 & 0 \\
		0	&0  &0  & 1
	\end{array}\right) \\
	\mathbf{R}_3=\left(\begin{array}{cccc}
		1& 0 & 0 & 0 \\
		0 &1  &0  &0  \\
		1	&0 & -1 & 1 \\
		0	&0  &0 & 1
	\end{array}\right)&&
	\mathbf{S}_f=\left(\begin{array}{cccc}
		1& 0 & 0 & 0 \\
		0 	&1  &0  &0  \\
		0	&0  & 1 & 0 \\
		2	&4  & 4 & -1
	\end{array}\right)
\end{align*}

The matrix of the octahedral quadratic form is 
\begin{align}\mathbf Q_{\{3,4\}}=
	\left(
	\begin{array}{cccc}
		1 & -2 & -2 & -1 \\
		-2 & 4 & 0 & -2 \\
		-2 & 0 & 4 & -2 \\
		-1 & -2 & -2 & 1 \\
	\end{array}
	\right)
\end{align}
which implies the following relations on any fundamental bend vector $\mathbf b=(b_1,b_2,b_3,b_4)^\top$ of an octahedral circle packing $\S_\P$:
\begin{enumerate}[-]
	\item (Octahedral Descartes' theorem) 
	\begin{align}
	(b_1-b_4)^2+(b_1-2b_2+b_4)^2+(b_1-2b_3+b_4)^2=2(b_1+b_4)^2
	\end{align}
	\item (Consecutive fundamental bases) If $(b_2,b_3,b_4,b_5)^\top$ is the fundamental bend vector of $\S_\P$ whose first three entries are the last three of $\mathbf b$, then
	\begin{align}
		b_5=b_1-b_2+b_4
	\end{align}
	\item (Dual inversion) if $\mathbf b'=(b_1,b_2,b_3,b_4')^\top$ is the fundamental bend vector of $\S'_\P$ obtained from $\S_\P$ after applying the inversion through the dual sphere orthogonal to $(b_1,b_2,b_3)$, then
	%
	\begin{align}
		b_4'=2b_1+4b_2+4b_3-b_4
	\end{align}
	
	\item  (Integrality condition) if $b_1,b_2,b_3,\Delta_{\{3,4\}}\in\mathbb Z$ where
	\begin{align}\label{eq:inttetra}
		\Delta_{\{3,4\}}=2(b_1b_2+b_2b_3+b_3b_1)
	\end{align}
	then the octahedral crystallographic packing $\mathscr P_{\{3,4\}}(b_1,b_2,b_3)$ is integral. The primitive triples satisfying the previous condition are paremeterized by $(b_1,b_2,b_3)=\frac 1g(h_1, h_2, h_3)$ where
	\begin{align}
		h_1=2t_2(t_2+t_3),&&h_2=2t_3(t_2+t_3),&&h_3=t_1^2-2t_2t_3
	\end{align}
	$t_1,t_2,t_3$ are three coprime integers and $g=\mathrm{gcd}(2t_1(t_2+t_3),h_1,h_2,h_3)$.
	The two integral packings of Figure  \ref{fig:apoct} are generated by taking $(t_1,t_2,t_3)=(1,0,0),(1,1,-2)$.
\end{enumerate}
\begin{figure}[H]
	\centering
	\begin{tikzpicture}
		\begin{scope}
			\node at (0,0) {\includegraphics[align=c,trim=0 80 0 80,clip,width=0.5\textwidth]{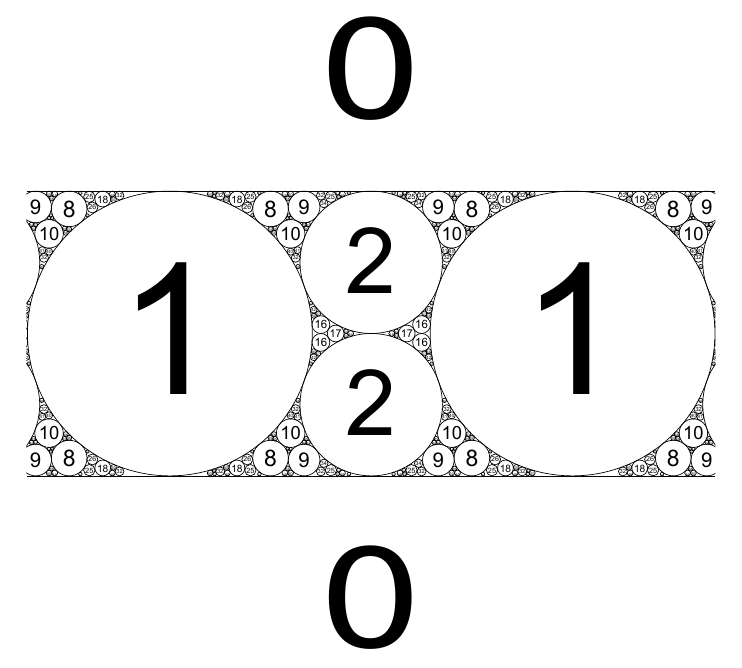}	}; 
		\end{scope}
		\begin{scope}[xshift=8cm]
			\clip (0,0)  circle (.2\textwidth) ;
			\node[anchor=center] at (-.5,-.5) {	\includegraphics[align=c,trim=0 20 40 60,clip,width=0.48\textwidth]{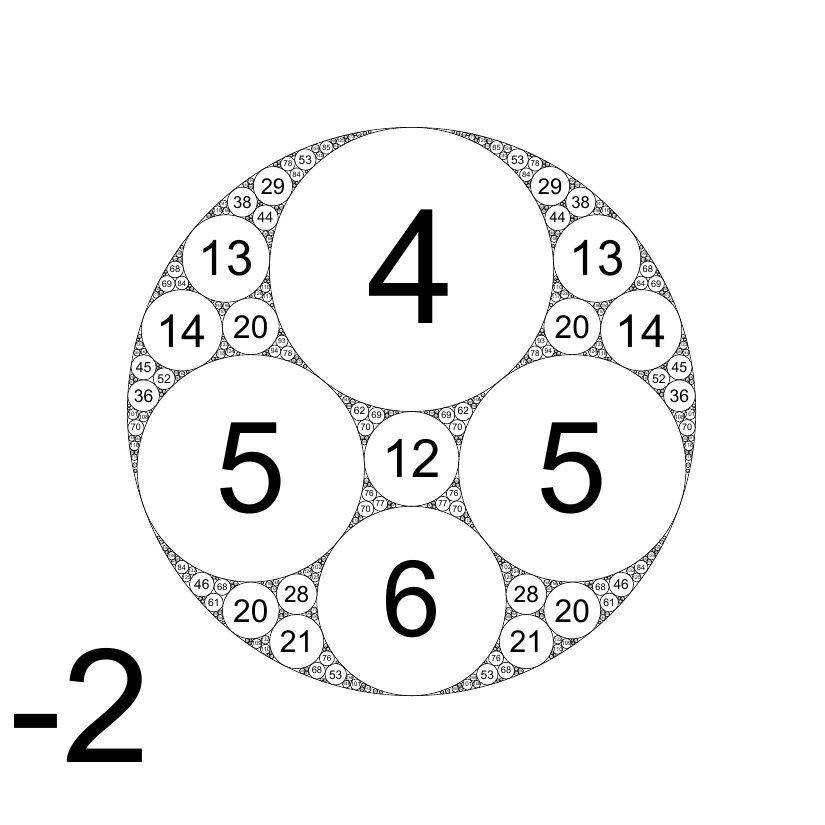}	}; 
		\end{scope}
	\end{tikzpicture}
	\caption{The integral octahedral crystallographic packings $\mathscr P_{\{3,4\}}(0,0,1)$ (left) and $\mathscr P_{\{3,4\}}(-2,4,5)$ (right).	}
	\label{fig:apoct}
\end{figure}

\subsection{Cube $\{4,3\}$} 

Cubic circle packings were studied by Stange in \cite{stange2015bianchi} as a particular case of Schmidt arrangements. In Figure \ref{fig:canonicalcube}, we show three cubic circle packings obtained by the arrangement projection of canonical cubes. The canonical length is $1/\sqrt 2$.
\begin{figure}[H]
	\begin{tikzpicture}
		\begin{scope}[yshift=3.3cm]
			\node at (-5.75,0) {\includegraphics[trim=0 0 0 60,clip,align=c,width=4.4cm]{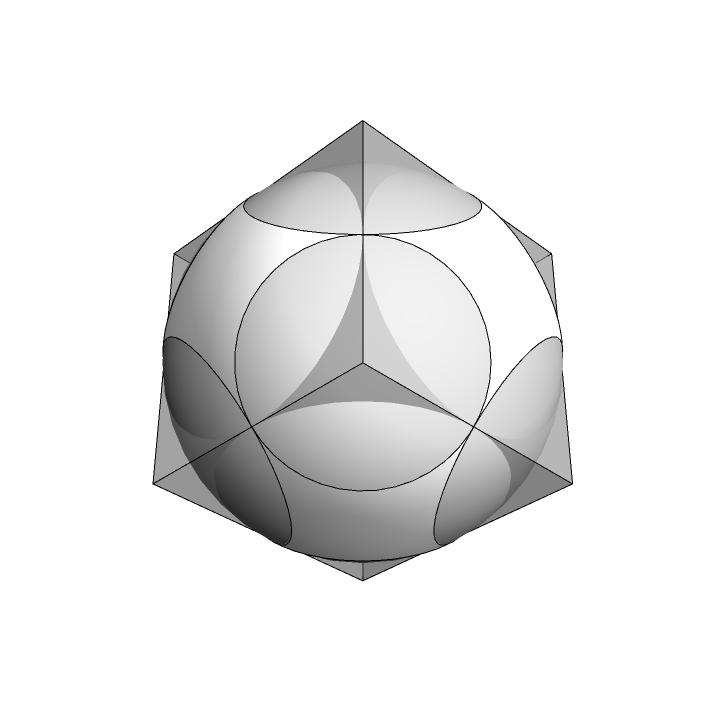} };
			\node at (0,0) {\includegraphics[trim=0 0 0 60,clip,align=c,width=4.4cm]{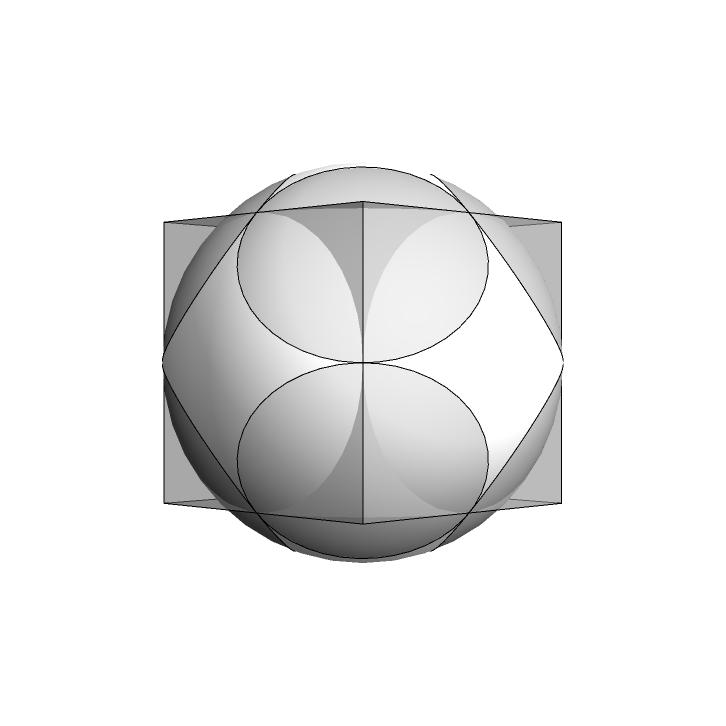}  };
			\node at (5.75,0) {\includegraphics[trim=0 0 0 60,clip,align=c,width=4.4cm]{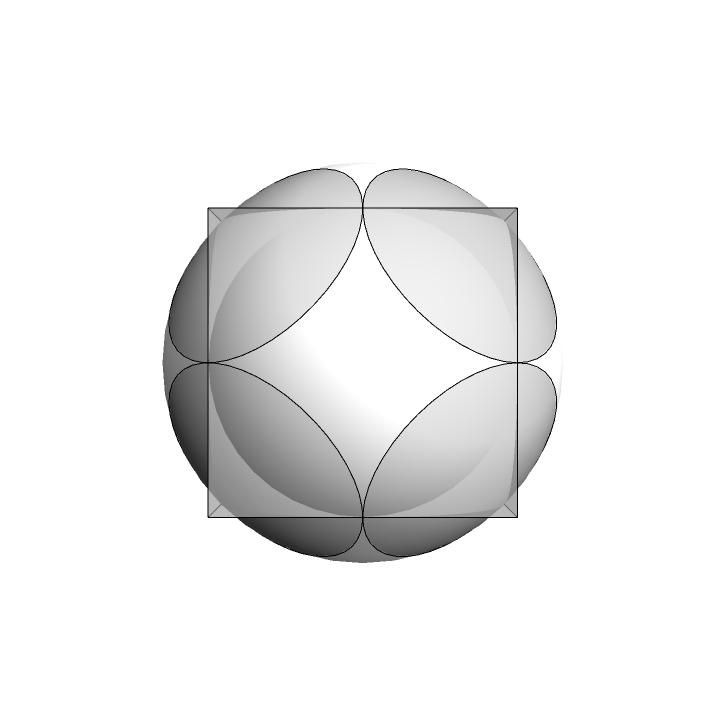}  };
		\end{scope}
		\begin{scope}
			\node at (-5.75,0) {\begin{tikzpicture}[scale=1.5] 

\fill[darkgray,opacity=\opac] (-1.2,-1.2) rectangle (1.2,1.2);
\fill[white] (0,0) circle (1cm);

\draw[very thin, draw=black] (0,0) circle (-1.000cm);
\draw[very thin, draw=black, fill=darkgray, fill opacity=\opac] (-0.5371,-0.3101) circle (0.3798cm);
\draw[very thin, draw=black, fill=darkgray, fill opacity=\opac] (0.5371,-0.3101) circle (0.3798cm);
\draw[very thin, draw=black, fill=darkgray, fill opacity=\opac] (0,-0.2606) circle (0.1596cm);
\draw[very thin, draw=black, fill=darkgray, fill opacity=\opac] (0,0.6202) circle (0.3798cm);
\draw[very thin, draw=black, fill=darkgray, fill opacity=\opac] (-0.2257,0.1303) circle (0.1596cm);
\draw[very thin, draw=black, fill=darkgray, fill opacity=\opac] (0.2257,0.1303) circle (0.1596cm);
\draw[very thin, draw=black, fill=darkgray, fill opacity=\opac] (0,0) circle (0.1010cm);

\end{tikzpicture}};
			\node at (0,0) {\begin{tikzpicture}[scale=1] 

\newcommand\wid{1.5}
\draw[very thin] (-\wid,1) -- (\wid,1);
\fill[darkgray,opacity=\opac] (-\wid,1) rectangle (\wid,1.618);
\draw[very thin] (-\wid,-1) -- (\wid,-1);
\fill[darkgray,opacity=\opac] (-\wid,-1) rectangle (\wid,-1.618);
\draw[very thin, draw=black, fill=darkgray, fill opacity=\opac] (-0.7071,0.5000) circle (0.5000cm);
\draw[very thin, draw=black, fill=darkgray, fill opacity=\opac] (-0.7071,-0.5000) circle (0.5000cm);
\draw[very thin, draw=black, fill=darkgray, fill opacity=\opac] (0.7071,0.5000) circle (0.5000cm);
\draw[very thin, draw=black, fill=darkgray, fill opacity=\opac] (0.7071,-0.5000) circle (0.5000cm);
\draw[very thin, draw=black, fill=darkgray, fill opacity=\opac] (0,0.2500) circle (0.2500cm);
\draw[very thin, draw=black, fill=darkgray, fill opacity=\opac] (0,-0.2500) circle (0.2500cm);

\end{tikzpicture} };
			\node at (5.75,0) {\begin{tikzpicture}[scale=.36] 
\draw[very thin] (2.414,2.414) circle (2.414cm);
\draw[very thin] (-2.414,2.414) circle (2.414cm);
\draw[very thin] (2.414,-2.414) circle (2.414cm);
\draw[very thin] (-2.414,-2.414) circle (2.414cm);
\draw[very thin] (.414,.414) circle (.414cm);
\draw[very thin] (-.414,.414) circle (.414cm);
\draw[very thin] (.414,-.414) circle (.414cm);
\draw[very thin] (-.414,-.414) circle (.414cm);

\fill[darkgray,opacity=\opac] (2.414,2.414) circle (2.414cm);
\fill[darkgray,opacity=\opac] (-2.414,2.414) circle (2.414cm);
\fill[darkgray,opacity=\opac] (2.414,-2.414) circle (2.414cm);
\fill[darkgray,opacity=\opac] (-2.414,-2.414) circle (2.414cm);
\fill[darkgray,opacity=\opac] (.414,.414) circle (.414cm);
\fill[darkgray,opacity=\opac] (-.414,.414) circle (.414cm);
\fill[darkgray,opacity=\opac] (.414,-.414) circle (.414cm);
\fill[darkgray,opacity=\opac] (-.414,-.414) circle (.414cm);

\end{tikzpicture}};
		\end{scope}
	\end{tikzpicture}
	\caption{
			(Top figures, from left to right) A canonical cube with its spherical illuminated regions viewed from above. The images show three orientations: a vertex, an edge, and a face centred at the North Pole. (Bottom figures) The corresponding cubic circle packings given by the arrangement projection.
		}
	\label{fig:canonicalcube}	
\end{figure}

The Gramian can be computed from Table \ref{tab:distcube} and the Möbius spectrum is $(-16_{(1)},0_{(4)},8_{(3)})$.
\begin{table}[H]
	\begin{tabular}{c|cccccc}
		\rule[-1ex]{0pt}{2.5ex} Graph distance & 0 & 1 & 2 & 3 \\
		\hline
		\rule[-1ex]{0pt}{2.5ex} Inversive product & 1 & $-1$ & $-3$ & $-5$ 
	\end{tabular}
	\caption{The inversive product compared to the graph-distance of cubic circle packings.}
	\label{tab:distcube}
\end{table}

The linear representation of the full symmetry group $\Gamma_{\{4,3\}}<\mathrm{SL}_4(\mathbb Z)$ is generated by the matrices
\begin{align*}
	&\mathbf{R}_1=\left(\begin{array}{cccc}
		0 & 1 & 0 & 0 \\
		1 & 0 & 0 & 0 \\
		1 & -1 & 1 & 0 \\
		1 & -1 & 0 & 1 \\
	\end{array}\right) &
	\mathbf{R}_2=\left(\begin{array}{cccc}
		1 & 0 & 0 & 0 \\
		1 & -1 & 1 & 0 \\
		0 & 0 & 1 & 0 \\
		0 & 0 & 0 & 1 \\
	\end{array}\right) \\
	&\mathbf{R}_3=\left(\begin{array}{cccc}
		1 & 0 & 0 & 0 \\
		0 & 1 & 0 & 0 \\
		0 & 1 & -1 & 1 \\
		0 & 0 & 0 & 1 \\
	\end{array}\right) &
	\mathbf{S}_f=\left(\begin{array}{cccc}
		1& 0 & 0 & 0 \\
		0 	&1  &0  &0  \\
		0	&0  & 1 & 0 \\
		2	&0 & 4 & -1
	\end{array}\right)
\end{align*}

The matrix of the cubic quadratic form is 
\begin{align}\mathbf Q_{\{4,3\}}=
	\left(
	\begin{array}{cccc}
		1 & -2 & 0 & -1 \\
		-2 & 4 & -2 & 0 \\
		0 & -2 & 4 & -2 \\
		-1 & 0 & -2 & 1 \\
	\end{array}
	\right)
\end{align}
which implies the following relations on any fundamental bend vector $\mathbf b=(b_1,b_2,b_3,b_4)^\top$ of a cubic circle packing $\S_\P$:
\begin{enumerate}[-]
	\item (Cubic Descartes' theorem) 
	\begin{align}
		2(b_1-b_2)^2+2(b_2-b_3)^2+2(b_3-b_4)^2=(b_1+b_4)^2
	\end{align}
	\item (Face relation) if $b_1,b_2,b_3,b_4$ are the bends of four circles of  $\S_\P$ corresponding to four consecutive vertices in a square face, then
	\begin{align}
		b_4=b_1-b_2+b_3
	\end{align}
	\item (Consecutive fundamental bases) If $(b_2,b_3,b_4,b_5)^\top$ is the fundamental bend vector of $\S_\P$ whose first three entries are the last three of $\mathbf b$, then
	\begin{align}
		b_5=b_1-b_2+b_4
	\end{align}
	\item (Dual inversion) if $\mathbf b'=(b_1,b_2,b_3,b_4')^\top$ is the fundamental bend vector of $\S'_\P$ obtained from $\S_\P$ after applying the inversion through the dual sphere which is orthogonal to the circles corresponding to $(b_1,b_2,b_3)$, then
	\begin{align}
		b_4'=2b_1+4b_2+b_3-b_4
	\end{align}
	\item  (Integrality condition) if $b_1,b_2,b_3,\Delta_{\{4,3\}}\in\mathbb Z$ where
	\begin{align}\label{eq:intcube}
		\Delta_{\{4,3\}}=b_1b_2+b_2b_3+b_3b_1-b_2^2
	\end{align}
	then the cubic crystallographic packing $\mathscr P_{\{4,3\}}(b_1,b_2,b_3)$ is integral. The primitive triples satisfying the previous condition are paremeterized by $(b_1,b_2,b_3)=\frac 1g(h_1, h_2, h_3)$ where
	\begin{align}
		h_1=2t_2^2+2t_2t_3,&&h_2=2t_2t_3+2t_3^2,&&h_3=t_1^2-2t_2t_3
	\end{align}
$t_1,t_2,t_3$ are three coprime integers and $g=\mathrm{gcd}(t_1(t_2+t_3),h_1,h_2,h_3)$.	The two integral packings of Figure  \ref{fig:apocube} are generated by taking $(t_1,t_2,t_3)=(1,0,0),(0,-1,4)$.
\end{enumerate}

\begin{figure}[H]
	\centering
	\begin{tikzpicture}
		\begin{scope}
			\node at (0,0) {\includegraphics[align=c,trim=0 80 0 80,clip,width=0.45\textwidth]{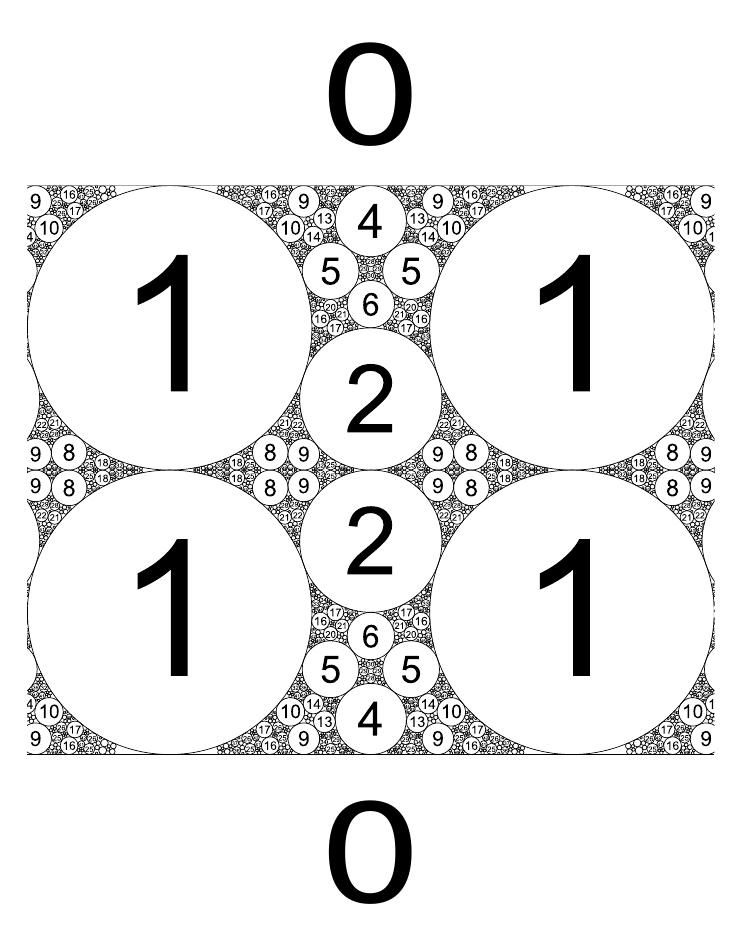}	}; 
		\end{scope}
		\begin{scope}[xshift=8cm,yshift=.1cm]
			\clip (0,0)  circle (.2\textwidth) ;
			\node[anchor=center] at (-.5,0) {	\includegraphics[align=c,trim=0 20 40 0,clip,width=0.47\textwidth]{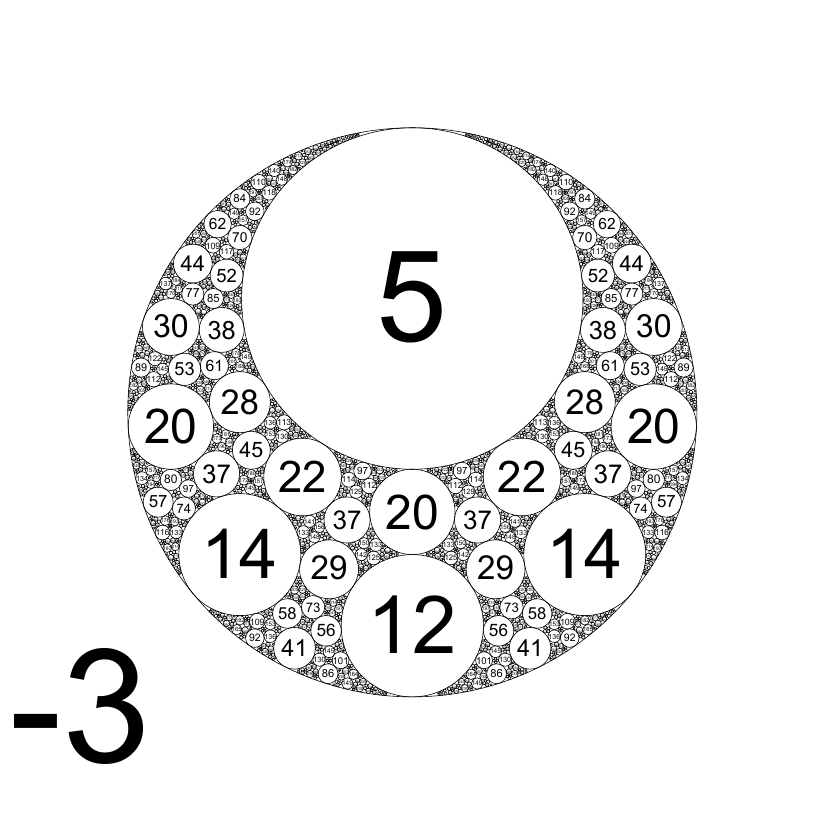}	}; 
		\end{scope}
	\end{tikzpicture}
	\caption{The integral cubic crystallographic packings $\mathscr P_{\{4,3\}}(0,0,1)$  (left) and $\mathscr P_{\{4,3\}}(-3,12,20)$ (right).	}
	\label{fig:apocube}
\end{figure}

\subsection{Icosahedron $\{3,5\}$} Icosahedral circle packings were studied by Bolt, Butler and Hovland as a particular case of \textit{Apollonian ring packings} \cite{aporingpacks}. In Figure \ref{fig:canonicalico}, we show three icosahedral circle packings obtained by the arrangement projections of three canonical icosahedra. The canonical length is $1/\varphi$.

\begin{figure}[H]
	\begin{tikzpicture}
		\begin{scope}[yshift=3.2cm]
			\node at (-5.75,0) {\includegraphics[trim=0 0 0 70,clip,align=c,width=4.4cm]{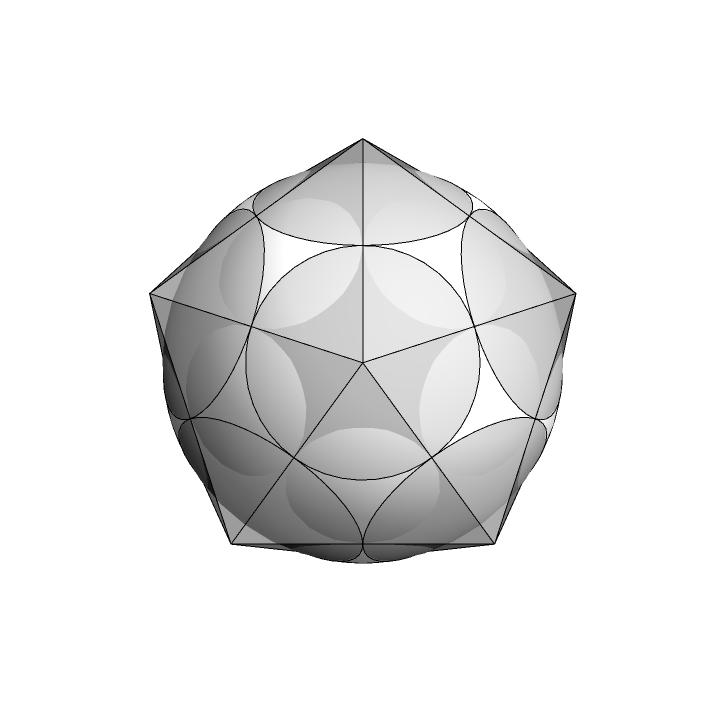} };
			\node at (0,0) {\includegraphics[trim=0 0 0 70,clip,align=c,width=4.4cm]{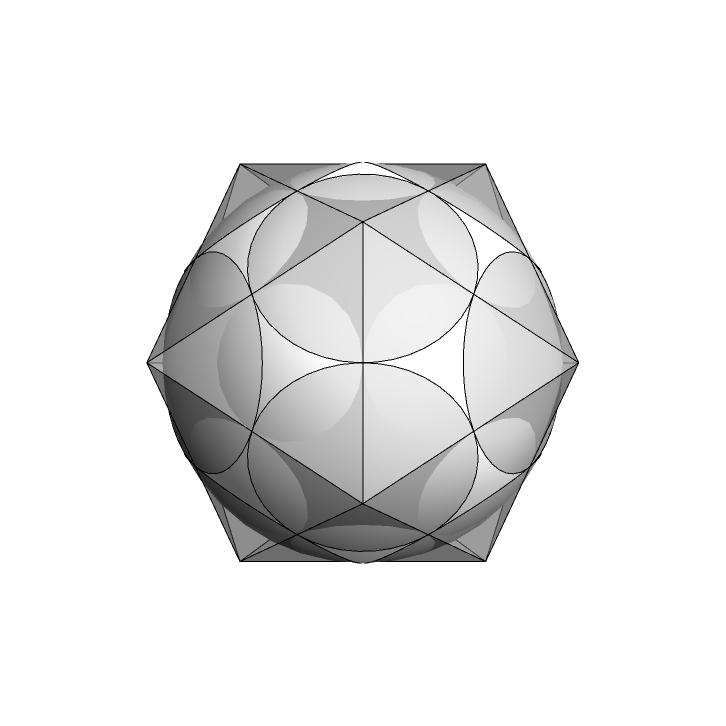}  };
			\node at (5.75,0) {\includegraphics[trim=0 0 0 70,clip,align=c,width=4.4cm]{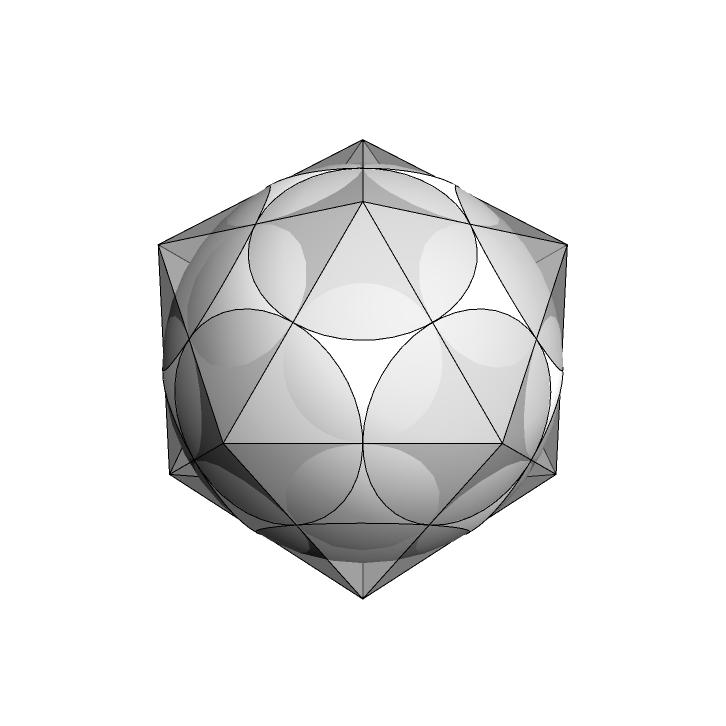}  };
		\end{scope}
		\begin{scope}
			\node at (-5.75,0) {\begin{tikzpicture}[scale=1.5] 

\fill[darkgray,opacity=\opac] (-1.2,-1.2) rectangle (1.2,1.2);
\draw[very thin,draw=black, fill=white] (0,0) circle (1cm);

\draw[very thin, draw=black, fill=darkgray, fill opacity=\opac] (0.1862,-0.06050) circle (0.1151cm);
\draw[very thin, draw=black, fill=darkgray, fill opacity=\opac] (0,-0.1958) circle (0.1151cm);
\draw[very thin, draw=black, fill=darkgray, fill opacity=\opac] (0.3702,-0.5095) circle (0.3702cm);
\draw[very thin, draw=black, fill=darkgray, fill opacity=\opac] (-0.1862,-0.06050) circle (0.1151cm);
\draw[very thin, draw=black, fill=darkgray, fill opacity=\opac] (0.1151,0.1584) circle (0.1151cm);
\draw[very thin, draw=black, fill=darkgray, fill opacity=\opac] (0.5990,0.1946) circle (0.3702cm);
\draw[very thin, draw=black, fill=darkgray, fill opacity=\opac] (0,0) circle (0.08070cm);
\draw[very thin, draw=black, fill=darkgray, fill opacity=\opac] (-0.3702,-0.5095) circle (0.3702cm);
\draw[very thin, draw=black, fill=darkgray, fill opacity=\opac] (-0.5990,0.1946) circle (0.3702cm);
\draw[very thin, draw=black, fill=darkgray, fill opacity=\opac] (0,0.6298) circle (0.3702cm);
\draw[very thin, draw=black, fill=darkgray, fill opacity=\opac] (-0.1151,0.1584) circle (0.1151cm);

\end{tikzpicture}};
			\node at (0,0) {\begin{tikzpicture}[scale=1] 

\newcommand\wid{2.9}
\draw[very thin] (-\wid,1) -- (\wid,1);
\fill[darkgray,opacity=\opac] (-\wid,1) rectangle (\wid,1.618);
\draw[very thin] (-\wid,-1) -- (\wid,-1);
\fill[darkgray,opacity=\opac] (-\wid,-1) rectangle (\wid,-1.618);
\draw[very thin, draw=black, fill=darkgray, fill opacity=\opac] (-0.3820,0.6180) circle (0.3820cm);
\draw[very thin, draw=black, fill=darkgray, fill opacity=\opac] (-1.618,0) circle (1.00cm);
\draw[very thin, draw=black, fill=darkgray, fill opacity=\opac] (-0.3820,-0.6180) circle (0.3820cm);
\draw[very thin, draw=black, fill=darkgray, fill opacity=\opac] (1.618,0) circle (1.00cm);
\draw[very thin, draw=black, fill=darkgray, fill opacity=\opac] (0,0.1910) circle (0.1910cm);
\draw[very thin, draw=black, fill=darkgray, fill opacity=\opac] (0.3820,0.6180) circle (0.3820cm);
\draw[very thin, draw=black, fill=darkgray, fill opacity=\opac] (-0.3820,0) circle (0.2361cm);
\draw[very thin, draw=black, fill=darkgray, fill opacity=\opac] (0.3820,-0.6180) circle (0.3820cm);
\draw[very thin, draw=black, fill=darkgray, fill opacity=\opac] (0.3820,0) circle (0.2361cm);
\draw[very thin, draw=black, fill=darkgray, fill opacity=\opac] (0,-0.1910) circle (0.1910cm);

\end{tikzpicture} };
			\node at (5.75,0) {\begin{tikzpicture}[scale=1] 
 \clip (-1.2,-1.2) rectangle (1.2,1.2);

\draw[very thin, draw=black, fill=darkgray, fill opacity=\opac] (0.03403,0.01965) circle (0.03403cm);
\draw[very thin, draw=black, fill=darkgray, fill opacity=\opac] (0,-0.03930) circle (0.03403cm);
\draw[very thin, draw=black, fill=darkgray, fill opacity=\opac] (0.08727,-0.05038) circle (0.05393cm);
\draw[very thin, draw=black, fill=darkgray, fill opacity=\opac] (1.000,-0.5774) circle (1.000cm);
\draw[very thin, draw=black, fill=darkgray, fill opacity=\opac] (0,0.1008) circle (0.05393cm);
\draw[very thin, draw=black, fill=darkgray, fill opacity=\opac] (-0.08727,-0.05038) circle (0.05393cm);
\draw[very thin, draw=black, fill=darkgray, fill opacity=\opac] (-0.03403,0.01965) circle (0.03403cm);
\draw[very thin, draw=black, fill=darkgray, fill opacity=\opac] (0,-0.1578) circle (0.08445cm);
\draw[very thin, draw=black, fill=darkgray, fill opacity=\opac] (0.1366,0.07889) circle (0.08445cm);
\draw[very thin, draw=black, fill=darkgray, fill opacity=\opac] (-1.000,-0.5774) circle (1.000cm);
\draw[very thin, draw=black, fill=darkgray, fill opacity=\opac] (0,1.155) circle (1.000cm);
\draw[very thin, draw=black, fill=darkgray, fill opacity=\opac] (-0.1366,0.07889) circle (0.08445cm);

\end{tikzpicture}};
		\end{scope}
	\end{tikzpicture}
	\caption{
			(Top figures, from left to right) A canonical icosahedron with its spherical illuminated regions viewed from above. The images show three orientations: a vertex, an edge, and a face centred at the North Pole. (Bottom figures) The corresponding icosahedral circle packings given by the arrangement projection.
		}
	\label{fig:canonicalico}	
\end{figure}

The Gramian can be computed from Table \ref{tab:distico} and the Möbius spectrum is $(-12\varphi^2_{(1)},0_{(8)},4(1+\varphi^2)_{(3)})$.
\begin{table}[H]
	\begin{tabular}{c|cccccc}
		\rule[-1ex]{0pt}{2.5ex} Graph distance & 0 & 1 & 2 & 3 \\
		\hline
		\rule[-1ex]{0pt}{2.5ex} Inversive product & 1 & $-1$ & $1-2\varphi^2$ & $-1-2\varphi^2$ \\
	\end{tabular}
	\caption{The inversive product compared to the graph-distance of icosahedral circle packings.}
	\label{tab:distico}
\end{table}

The linear representation of the full symmetry group $\Gamma_{\{3,5\}}<\mathrm{SL}_4(\mathbb Z[\varphi])$ is generated by the matrices
\begin{align*}
	&\mathbf{R}_1=\left(\begin{array}{cccc}
		0 & 1 & 0 & 0 \\
		1 & 0 & 0 & 0 \\
		0 & 0 & 1 & 0 \\
		\varphi  & -\varphi  & 0 & 1 \\
	\end{array}\right) 
	&\mathbf{R}_2=\left(\begin{array}{cccc}
		1& 0 & 0 & 0 \\
		0 &0  &1  &0  \\
		0	&1 & 0 & 0 \\
		0	&0  &0  & 1
	\end{array}\right)\\
	&
	\mathbf{R}_3=\left(\begin{array}{cccc}
		1 & 0 & 0 & 0 \\
		0 & 1 & 0 & 0 \\
		\varphi  & 0 & -\varphi  & 1 \\
		1 & 0 & -\varphi  & \varphi  \\
	\end{array}\right) &
	\mathbf{S}_f=\left(\begin{array}{cccc}
		1& 0 & 0 & 0 \\
		0 	&1  &0  &0  \\
		0	&0  & 1 & 0 \\
		2	&2\varphi^2  & 2\varphi^2 & -1
	\end{array}\right)
\end{align*}

The matrix of the icosahedral quadratic form is 
\begin{align}\mathbf Q_{\{3,5\}}=
	\left(
	\begin{array}{cccc}
		1 & -\varphi ^2 & -\varphi ^2 & -1 \\
		-\varphi ^2 & \varphi ^4 & \varphi & -\varphi ^2 \\
		-\varphi ^2 & \varphi & \varphi ^4 & -\varphi ^2 \\
		-1 & -\varphi ^2 & -\varphi ^2 & 1 \\
	\end{array}
	\right)
\end{align}
which implies the following relations on any fundamental bend vector $\mathbf b=(b_1,b_2,b_3,b_4)^\top$ of an icosahedral circle packing $\S_\P$:
\begin{enumerate}[-]
	\item (Icosahedral Descartes' theorem) 
\begin{align}
	(b_1+b_2+b_3+b_4)^2= 2(b_1^2+b_2^2+b_3^2+b_4^2)-\varphi^{-1}(b_1-b_4)^2+\varphi(b_2+b_3)^2	
\end{align}
	\item (Consecutive fundamental bases) If $(b_2,b_3,b_4,b_5)^\top$ is the fundamental bend vector of $\S_\P$ whose first three entries are the last three of $\mathbf b$, then
	\begin{align}
		b_5=b_1-\varphi b_2+\varphi b_4
	\end{align}
	\item (Dual inversion) if $\mathbf b'=(b_1,b_2,b_3,b_4')^\top$ is the fundamental bend vector of $\S'_\P$ obtained from $\S_\P$ after applying the inversion through the dual sphere which is orthogonal to the circles corresponding to $(b_1,b_2,b_3)$, then
	\begin{align}
		b_4'=2b_1+2(1+\varphi)b_2+2(1+\varphi)b_3-b_4
	\end{align}
	\item  (Integrality condition) if $b_1,b_2,b_3,\Delta_{\{3,5\}}\in\mathbb Z[\varphi]$ where
	\begin{align}\label{eq:intcube}
		\Delta_{\{3,5\}}=(1+\varphi)(b_1b_2+b_2b_3+b_3b_1)
	\end{align}
	then the icosahedral crystallographic packing $\mathscr P_{\{3,5\}}(b_1,b_2,b_3)$ is integral. 
		The primitive triples satisfying the previous condition are paremetrized by $(b_1,b_2,b_3)=\frac 1g(h_1, h_2, h_3)$ where
	\begin{align}
		h_1=\varphi^2t_2(t_2+t_3),&&h_2=\varphi^2t_3(t_2+t_3),&&h_3=t_1^2-\varphi^2t_2t_3
	\end{align}
	 $t_1,t_2,t_3$ are three coprime integers in $\mathbb Z[\varphi]$ and $g=\mathrm{gcd}(\varphi^2t_1(t_2+t_3),h_1,h_2,h_3)$. The two integral packings of Figure  \ref{fig:apoico} are generated by taking $(t_1,t_2,t_3)=(1,0,0),
	(1,2/\varphi,-4/\varphi)$.
\end{enumerate}

\begin{figure}[H]
	\centering
	\begin{tikzpicture}
		\begin{scope}
			\node at (0,0) {\includegraphics[align=c,trim=0 80 0 80,clip,width=0.53\textwidth]{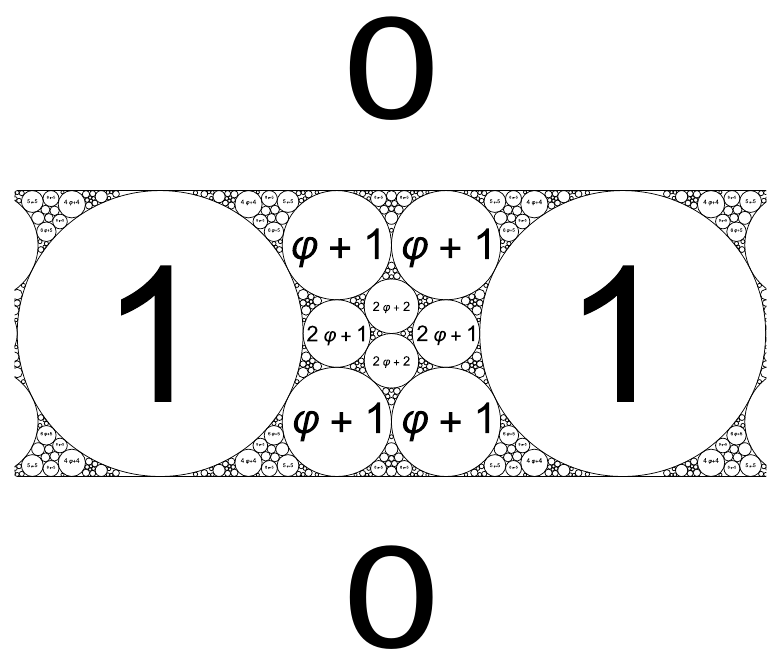}	}; 
		\end{scope}
		\begin{scope}[xshift=8cm,yshift=.0cm]
			\clip (0,0)  circle (.2\textwidth) ;
			\node[anchor=center] at (-.5,0) {	\includegraphics[align=c,trim=0 20 40 60,clip,width=0.45\textwidth]{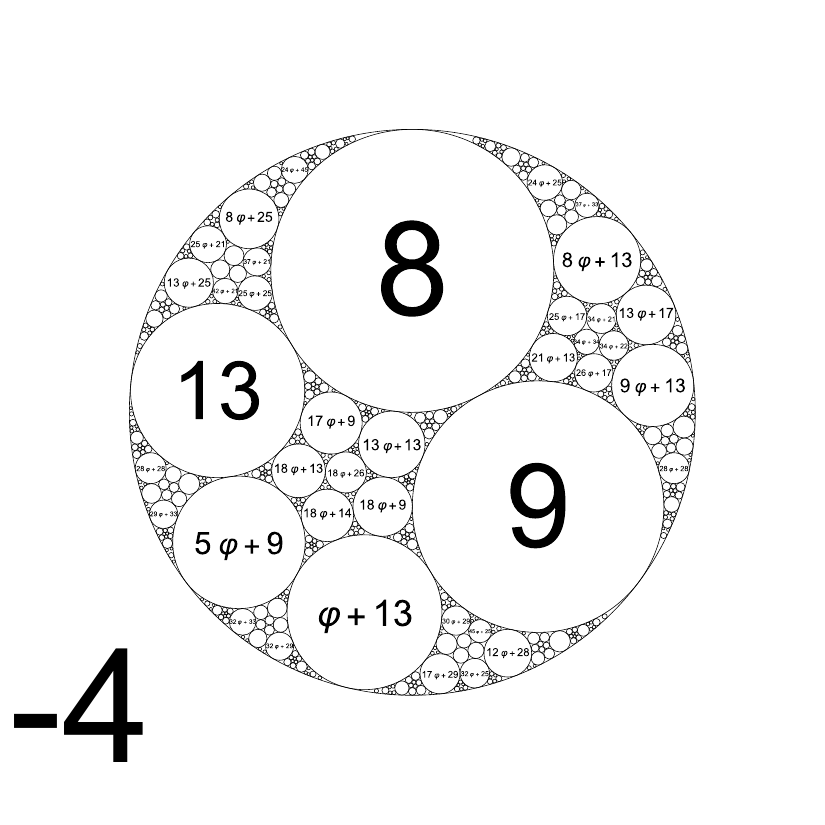}		}; 
		\end{scope}
	\end{tikzpicture}
		\vspace{-.5cm}
	\caption{The $\mathbb Z[\varphi]$-integral icosahedral crystallographic packings $\mathscr P_{\{3,5\}}(0,0,1)$ (left) and $\mathscr P_{\{3,5\}}(-4,8,9)$ (right).	}
	\label{fig:apoico}
\end{figure}

\subsection{Dodecahedron $\{5,3\}$}
In Figure \ref{fig:canonicaldode} we show three dodecahedral circle packings obtained by the arrangement projections of canonical dodecahedra. The canonical length is $1/\varphi^2$.
\begin{figure}[H]
	\begin{tikzpicture}
		\begin{scope}[yshift=3.1cm]
			\node at (-5.75,0) {\includegraphics[trim=0 0 0 70,clip,align=c,width=4.4cm]{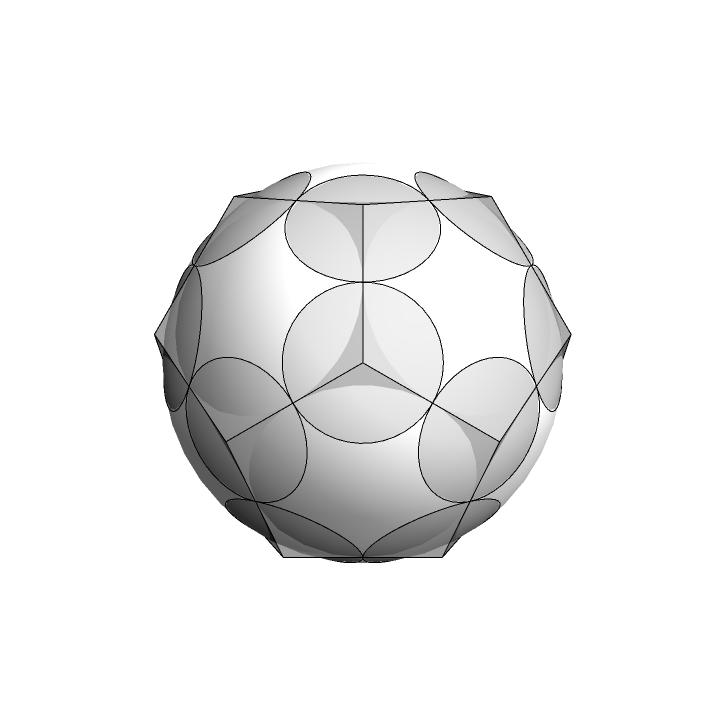} };
			\node at (0,0) {\includegraphics[trim=0 0 0 70,clip,align=c,width=4.4cm]{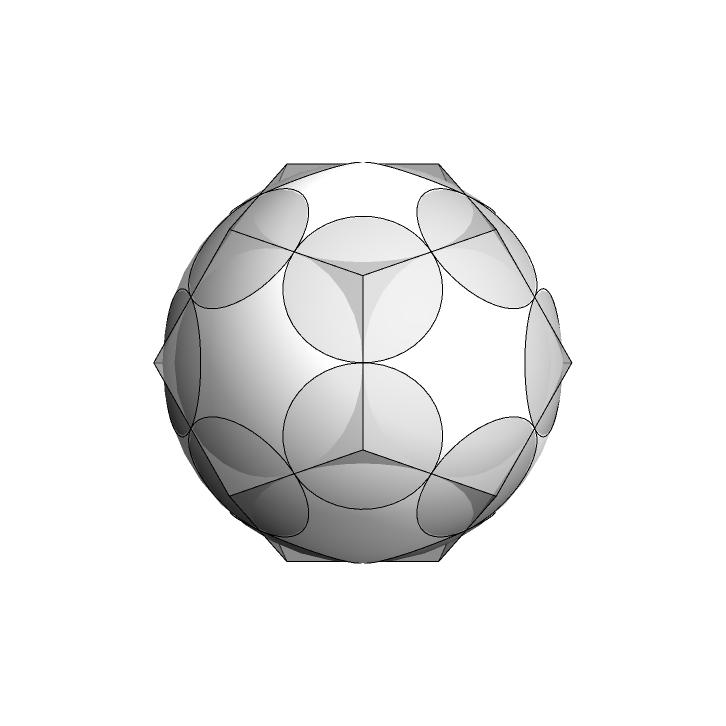}  };
			\node at (5.75,0) {\includegraphics[trim=0 0 0 70,clip,align=c,width=4.4cm]{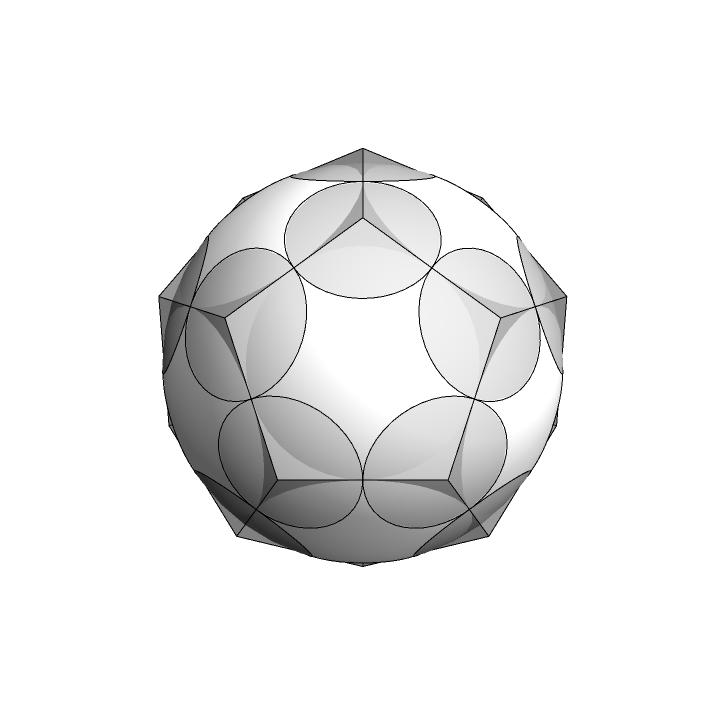}  };
		\end{scope}
		\begin{scope}
			\node at (-5.75,0) {\begin{tikzpicture}[scale=1.5] 

\fill[darkgray,opacity=\opac] (-1.2,-1.2) rectangle (1.2,1.2);
\draw[very thin,draw=black, fill=white] (0,0) circle (1cm);

\draw[very thin, draw=black, fill=darkgray, fill opacity=\opac] (0.5641,-0.3257) circle (0.3486cm);
\draw[very thin, draw=black, fill=darkgray, fill opacity=\opac] (0.1096,-0.2679) circle (0.1096cm);
\draw[very thin, draw=black, fill=darkgray, fill opacity=\opac] (0.2868,0.03909) circle (0.1096cm);
\draw[very thin, draw=black, fill=darkgray, fill opacity=\opac] (0.08403,-0.1085) circle (0.05193cm);
\draw[very thin, draw=black, fill=darkgray, fill opacity=\opac] (0.1360,-0.01853) circle (0.05193cm);
\draw[very thin, draw=black, fill=darkgray, fill opacity=\opac] (-0.1096,-0.2679) circle (0.1096cm);
\draw[very thin, draw=black, fill=darkgray, fill opacity=\opac] (-0.5641,-0.3257) circle (0.3486cm);
\draw[very thin, draw=black, fill=darkgray, fill opacity=\opac] (-0.2868,0.03909) circle (0.1096cm);
\draw[very thin, draw=black, fill=darkgray, fill opacity=\opac] (-0.1360,-0.01853) circle (0.05193cm);
\draw[very thin, draw=black, fill=darkgray, fill opacity=\opac] (-0.08403,-0.1085) circle (0.05193cm);

\draw[very thin, draw=black, fill=darkgray, fill opacity=\opac] (0,0.6514) circle (0.3486cm);
\draw[very thin, draw=black, fill=darkgray, fill opacity=\opac] (-0.1773,0.2289) circle (0.1096cm);
\draw[very thin, draw=black, fill=darkgray, fill opacity=\opac] (0.1773,0.2289) circle (0.1096cm);
\draw[very thin, draw=black, fill=darkgray, fill opacity=\opac] (0.05193,0.1270) circle (0.05193cm);
\draw[very thin, draw=black, fill=darkgray, fill opacity=\opac] (-0.05193,0.1270) circle (0.05193cm);
\draw[very thin, draw=black, fill=darkgray, fill opacity=\opac] (-0.06342,0.03661) circle (0.03919cm);
\draw[very thin, draw=black, fill=darkgray, fill opacity=\opac] (0,-0.07323) circle (0.03919cm);
\draw[very thin, draw=black, fill=darkgray, fill opacity=\opac] (0,0) circle (0.03403cm);
\draw[very thin, draw=black, fill=darkgray, fill opacity=\opac] (0.06342,0.03661) circle (0.03919cm);

\end{tikzpicture}};
			\node at (0,0) {\begin{tikzpicture}[scale=1] 

\draw[very thin] (-3,2.618) -- (3,2.618);
\fill[darkgray,opacity=\opac] (-3,2.618) rectangle (3,3.6);

\draw[very thin] (-3,-2.618) -- (3,-2.618);
\fill[darkgray,opacity=\opac] (-3,-2.618) rectangle (3,-3.6);

\draw[very thin, draw=black, fill=darkgray, fill opacity=\opac] (1.618,1.618) circle (1.000cm) ;
\draw[very thin, draw=black, fill=darkgray, fill opacity=\opac] (-1.618,1.618) circle (1.000cm) ;
\draw[very thin, draw=black, fill=darkgray, fill opacity=\opac] (-1.618,-1.618) circle (1.000cm) ;
\draw[very thin, draw=black, fill=darkgray, fill opacity=\opac] (1.618,-1.618) circle (1.000cm) ;
\draw[very thin, draw=black, fill=darkgray, fill opacity=\opac] (1.618,0) circle (0.6180cm) ;
\draw[very thin, draw=black, fill=darkgray, fill opacity=\opac] (-1.618,0) circle (0.6180cm);
\draw[very thin, draw=black, fill=darkgray, fill opacity=\opac] (0.3820,1.000) circle (0.3820cm) ;
\draw[very thin, draw=black, fill=darkgray, fill opacity=\opac] (-0.3820,1.000) circle (0.3820cm) ;
\draw[very thin, draw=black, fill=darkgray, fill opacity=\opac] (0.3820,-1.000) circle (0.3820cm) ;
\draw[very thin, draw=black, fill=darkgray, fill opacity=\opac] (-0.3820,-1.000) circle (0.3820cm) ;
\draw[very thin, draw=black, fill=darkgray, fill opacity=\opac] (0.7236,0) circle (0.2764cm) ;
\draw[very thin, draw=black, fill=darkgray, fill opacity=\opac] (-0.7236,0) circle (0.2764cm) ;
\draw[very thin, draw=black, fill=darkgray, fill opacity=\opac] (0.3820,0.3820) circle (0.2361cm) ;
\draw[very thin, draw=black, fill=darkgray, fill opacity=\opac] (-0.3820,0.3820) circle (0.2361cm) ;
\draw[very thin, draw=black, fill=darkgray, fill opacity=\opac] (-0.3820,-0.3820) circle (0.2361cm) ;
\draw[very thin, draw=black, fill=darkgray, fill opacity=\opac] (0.3820,-0.3820) circle (0.2361cm);

\draw[very thin, draw=black, fill=darkgray, fill opacity=\opac] (0,0.1910) circle (0.1910cm) ;
\draw[very thin, draw=black, fill=darkgray, fill opacity=\opac] (0,-0.1910) circle (0.1910cm) ;

\end{tikzpicture} };
			\node at (5.75,0) {\begin{tikzpicture}[scale=1] 

\draw[very thin, draw=black, fill=darkgray, fill opacity=\opac] (0.08683,-0.02821) circle (0.05366cm);
\draw[very thin, draw=black, fill=darkgray, fill opacity=\opac] (0.05366,0.07386) circle (0.05366cm);
\draw[very thin, draw=black, fill=darkgray, fill opacity=\opac] (0.2165,-0.07035) circle (0.08270cm);
\draw[very thin, draw=black, fill=darkgray, fill opacity=\opac] (0.1338,0.1842) circle (0.08270cm);
\draw[very thin, draw=black, fill=darkgray, fill opacity=\opac] (0.3253,0.1057) circle (0.1243cm);
\draw[very thin, draw=black, fill=darkgray, fill opacity=\opac] (-0.05366,0.07386) circle (0.05366cm);
\draw[very thin, draw=black, fill=darkgray, fill opacity=\opac] (-0.08683,-0.02821) circle (0.05366cm);
\draw[very thin, draw=black, fill=darkgray, fill opacity=\opac] (-0.2165,-0.07035) circle (0.08270cm);
\draw[very thin, draw=black, fill=darkgray, fill opacity=\opac] (-0.3253,0.1057) circle (0.1243cm);
\draw[very thin, draw=black, fill=darkgray, fill opacity=\opac] (-0.1338,0.1842) circle (0.08270cm);
\draw[very thin, draw=black, fill=darkgray, fill opacity=\opac] (0,-0.09130) circle (0.05366cm);
\draw[very thin, draw=black, fill=darkgray, fill opacity=\opac] (0,-0.2277) circle (0.08270cm);
\draw[very thin, draw=black, fill=darkgray, fill opacity=\opac] (-0.2011,-0.2767) circle (0.1243cm);
\draw[very thin, draw=black, fill=darkgray, fill opacity=\opac] (0.2011,-0.2767) circle (0.1243cm);
\draw[very thin, draw=black, fill=darkgray, fill opacity=\opac] (0.6650,-0.9152) circle (0.6650cm);
\draw[very thin, draw=black, fill=darkgray, fill opacity=\opac] (-0.6650,-0.9152) circle (0.6650cm);
\draw[very thin, draw=black, fill=darkgray, fill opacity=\opac] (-1.076,0.3496) circle (0.6650cm);
\draw[very thin, draw=black, fill=darkgray, fill opacity=\opac] (0,0.3421) circle (0.1243cm);
\draw[very thin, draw=black, fill=darkgray, fill opacity=\opac] (0,1.131) circle (0.6650cm);
\draw[very thin, draw=black, fill=darkgray, fill opacity=\opac] (1.076,0.3496) circle (0.6650cm);

\end{tikzpicture}};
		\end{scope}
	\end{tikzpicture}
	\vspace{-.5cm}
	\caption{
			(Top figures, from left to right) A canonical dodecahedron with its spherical illuminated regions viewed from above. The images show three orientations: a vertex, an edge, and a face centred at the North Pole. (Bottom figures) The corresponding dodecahedral circle packings given by the arrangement projection.
		}
	\label{fig:canonicaldode}	
\end{figure}

The Gramian can be computed from Table \ref{tab:distdode} and the Möbius spectrum is $(-20\varphi^4_{(1)},0_{(16)},20\varphi^2_{(3)})$.
\begin{table}[H]
	\begin{tabular}{c|cccccc}
		\rule[-1ex]{0pt}{2.5ex} Graph distance & 0 & 1 & 2 & 3 & 4 & 5 \\
		\hline
		\rule[-1ex]{0pt}{2.5ex} Inversive product & 1 & $-1$ & $1-2\varphi^2$ & $1-4\varphi^2$ & $1-2\varphi^4$ & $1-6\varphi^2$ \\
	\end{tabular}
	\caption{The inversive product compared to the graph-distance for dodecahedral circle packings.}
	\label{tab:distdode}
\end{table}

The linear representation of the full symmetry group $\Gamma_{\{5,3\}}<\mathrm{SL}_4(\mathbb Z[\varphi])$ is generated by the matrices
\begin{align*}
	&\mathbf{R}_1=
	\left(\begin{array}{cccc}
		0 & 1 & 0 & 0 \\
		1 & 0 & 0 & 0 \\
		\varphi  & -\varphi  & 1 & 0 \\
		\varphi ^2 & -\varphi ^2 & 0 & 1 \\
	\end{array}\right) 
	&
	\mathbf{R}_2=\left(\begin{array}{cccc}
		1 & 0 & 0 & 0 \\
		\varphi  & -\varphi  & 1 & 0 \\
		1 &-\varphi & \varphi  & 0 \\
		\varphi  & -\varphi ^2 & 1 & 1 \\
	\end{array}\right) \\
	&		
	\mathbf{R}_3=\left(\begin{array}{cccc}
		1 & 0 & 0 & 0 \\
		0 & 1 & 0 & 0 \\
		0 & \varphi  & -\varphi  & 1 \\
		0 & 1 & -\varphi & \varphi  \\
	\end{array}\right)  
	&\mathbf{S}_f=\left(\begin{array}{cccc}
		1& 0 & 0 & 0 \\
		0 	&1  &0  &0  \\
		0	&0  & 1 & 0 \\
		2	&2\varphi^{-1}  & 2\varphi^2 & -1
	\end{array}\right)
\end{align*}

The matrix of the dodecahedral quadratic form is 
\begin{align}\mathbf Q_{\{5,3\}}=
	\left(
	\begin{array}{cccc}
		1 & -\varphi ^2 & \varphi ^{-1} & -1 \\
		-\varphi ^2 & \varphi ^4 & -1-\varphi ^2 & \varphi ^{-1} \\
		\varphi ^{-1} & -1-\varphi ^2 & \varphi ^4 & -\varphi ^2 \\
		-1 & \varphi ^{-1} & -\varphi ^2 & 1 \\
	\end{array}
	\right)
\end{align}
which implies the following relations on any fundamental bend vector $\mathbf b=(b_1,b_2,b_3,b_4)^\top$ of a dodecahedral circle packing $\S_\P$:
\begin{enumerate}[-]
	\item (Dodecahedral Descartes' theorem) 
\begin{align}
	\begin{split}
		\varphi^{-1}((b_1+b_3)^2+(b_2+b_4)^2) +&(b_1-b_4)^2+(b_2-b_3)^2	\\
		+\varphi^2&\left((b_1-b_2)^2+(b_2-b_3)^2+(b_3-b_4)^2\right) =2\varphi(b_1+b_4)^2
	\end{split}
\end{align}
	\item (Face relation) if $b_1,b_2,b_3,b_4$ are the bends of four circles of  $\S_\P$ corresponding to four consecutive vertices in a pentagonal face, then
	\begin{align}
		b_4=b_1-\varphi b_2+\varphi b_3
	\end{align}
	\item (Consecutive fundamental bases) If $(b_2,b_3,b_4,b_5)^\top$ is the fundamental bend vector of $\S_\P$ whose first three entries are the last three of $\mathbf b$, then
	\begin{align}
		b_5=b_1-\varphi b_2+-\varphi b_4
	\end{align}
	\item (Dual inversion) if $\mathbf b'=(b_1,b_2,b_3,b_4')^\top$ is the fundamental bend vector of $\S'_\P$ obtained from $\S_\P$ after applying the inversion through the dual sphere which is orthogonal to the circles corresponding to $(b_1,b_2,b_3)$, then
	\begin{align}
		b_4'=2b_1+2(1-\varphi) b_2+2(1+\varphi)b_3-b_4
	\end{align}
	\item  (Integrality condition) if $b_1,b_2,b_3,\Delta_{\{5,3\}}\in\mathbb Z$ where
	\begin{align}\label{eq:intcube}
		\Delta_{\{5,3\}}=b_1b_2+b_2b_3+b_3b_1-\varphi b_2^2
	\end{align}
	then the dodecahedral crystallographic packing $\mathscr P_{\{5,3\}}(b_1,b_2,b_3)$ is integral. The primitive triples satisfying the previous condition are paremeterized by $(b_1,b_2,b_3)=\frac 1g(h_1, h_2, h_3)$ where
	\begin{align}
		h_1=t_2(t_2+t_3),&&h_2=t_3(t_2+t_3),&&h_3=t_1^2-t_2t_3+\varphi t_3^2
	\end{align}
$t_1,t_2,t_3$ are three coprime integers in $\mathbb Z[\varphi]$ and $g=\mathrm{gcd}(t_1(t_2+t_3),h_1,h_2,h_3)$. 	The two integral packings of Figure  \ref{fig:apodode} are generated by taking $(t_1,t_2,t_3)=(1,0,0),(1,\varphi+1,-1)\}$.
\end{enumerate}

\begin{figure}[H]
	\centering
	\begin{tikzpicture}
		\begin{scope}
			\node at (0,0) {\includegraphics[trim=0 100 0 100,clip,width=0.4\textwidth,align=c]{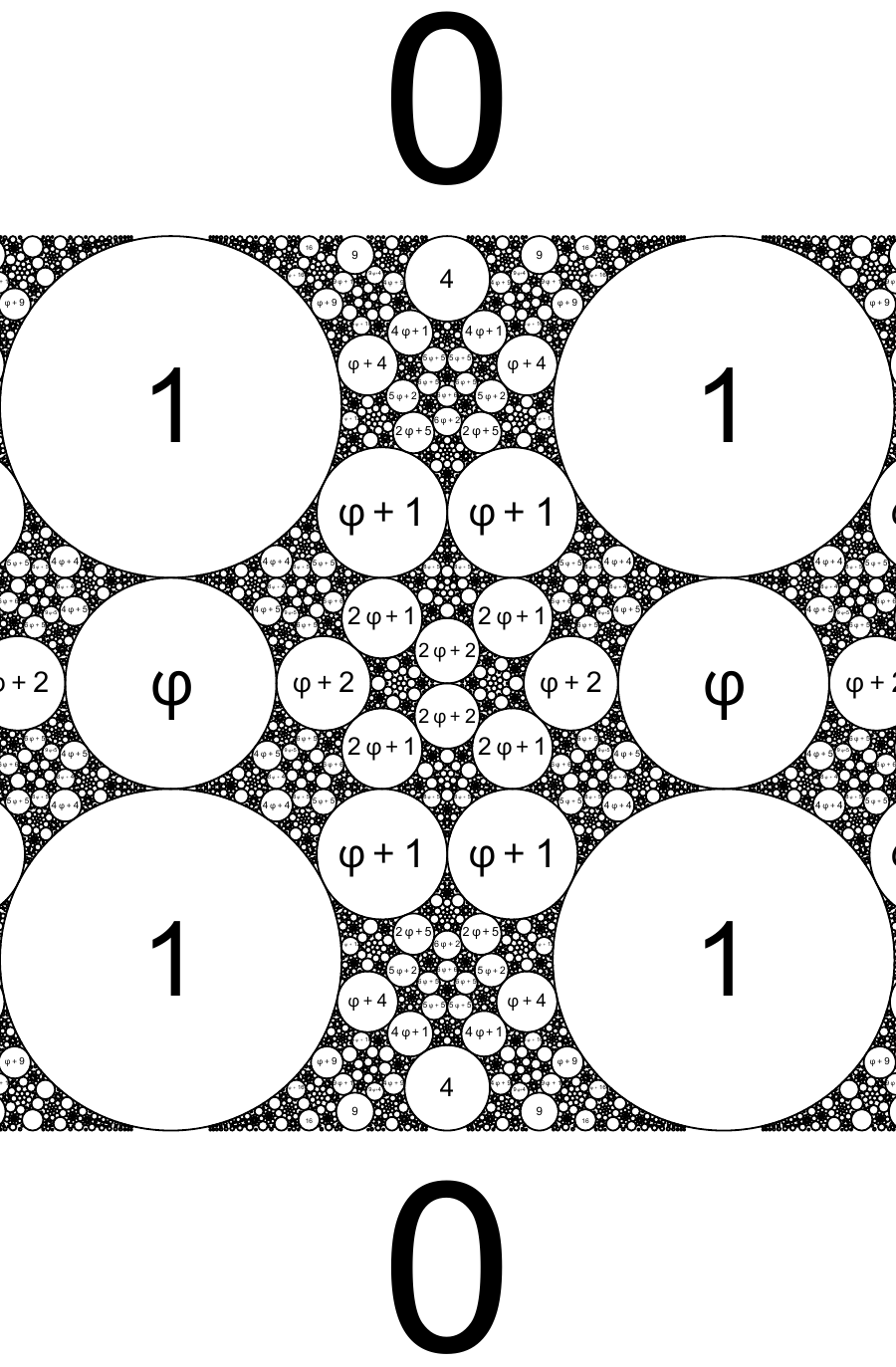}	}; 
		\end{scope}
		\begin{scope}[xshift=8cm,yshift=0cm]
			\clip (0,0)  circle (.2\textwidth) ;
			\node[anchor=center] at (-.5,-.5) {		\includegraphics[width=0.48\textwidth,align=c]{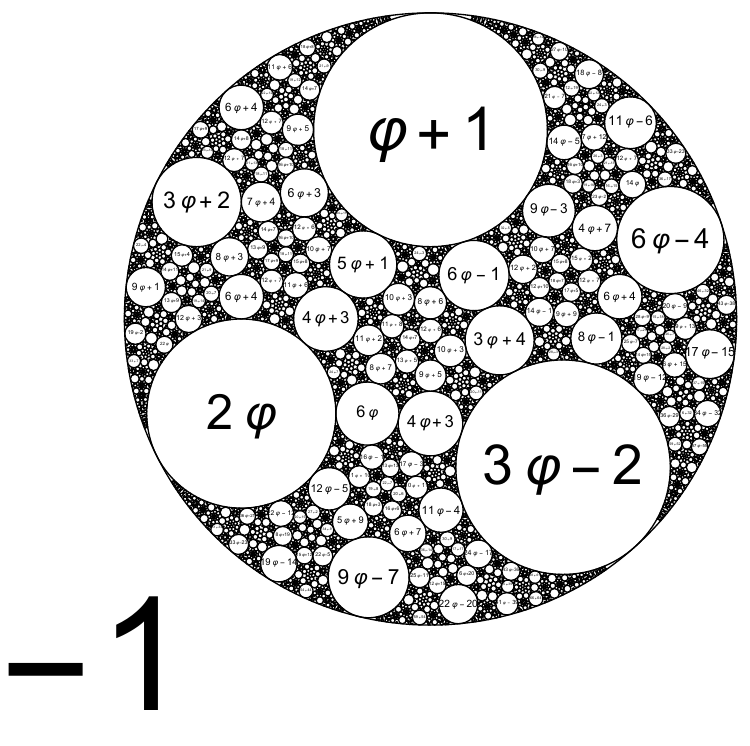}	}; 
		\end{scope}
	\end{tikzpicture}
	\caption{The $\mathbb Z[\varphi]$-integral dodecahedral crystallographic packings $\mathscr P_{\{5,3\}}(0,0,1)$  (left) and $\mathscr P_{\{3,5\}}(\varphi+1,-1,2\varphi)$. (right).	}
	\label{fig:apodode}
\end{figure}

\end{document}